\documentclass[a4paper,12pt]{amsart}
\usepackage{graphicx}
\usepackage{amsmath,amssymb}
\usepackage{pb-diagram,pb-xy}
\usepackage[matrix,arrow]{xy}
\usepackage[notext]{stix2}
\usepackage[bookmarks=true,bookmarksnumbered=true,bookmarkstype=toc,colorlinks=true,linkcolor=black]{hyperref}
\usepackage[shortlabels,inline]{enumitem}
\def\vitem{\\[.5ex]\text{\hskip1.175em}\item}
\def\hitem{\text{\hskip.5em}\item}

\def\Par{\par\,\!\hskip.55pc\,\!}
\renewcommand{\theenumi}{\arabic{enumi}}

\renewcommand{\theenumii}{\alph{enumii}}

\renewcommand{\theenumiii}{\roman{enumiii}}

\setlist[enumerate,1]{label=(\arabic*),ref=(\arabic*)}
\setlist[enumerate,2]{label=\alph*),ref=\theenumi\,\alph*)}
\setlist[enumerate,3]{label=\roman*.,ref=\theenumii\,\roman*.}
\setlist[enumerate,4]{label=\arabic*),ref=\theenumiii\,\arabic*)}
\theoremstyle{plain}
\newtheorem{thm}{Theorem}
\newtheorem{lem}[thm]{Lemma}
\newtheorem{prop}[thm]{Proposition}

\newtheorem{cor}[thm]{Corollary}

\newtheorem{expl}[thm]{Example}
\newtheorem{expls}[thm]{Examples}
\theoremstyle{definition}
\newtheorem{defn}[thm]{Definition}

\theoremstyle{remark}
\newtheorem{rem}[thm]{Remark}
\def\homeo{\approx}
\def\Hom{\operatorname{Hom}}

\def\Min{\operatorname{Min}}
\def\Max{\operatorname{Max}}
\def\Map{\operatorname{Map}}
\def\Wis{\operatorname{W}}
\def\Ent{\operatorname{E}}
\def\fatvee#1{\operatorname{\top}^{#1}\!}
\def\mnode{@}
\def\unode{\sharp}
\def\lnode{\flat}

\def\fracinline#1/#2{\mbox{\raise0.4ex\hbox{\footnotesize$#1$}{\hskip-.1em\mbox{\raise0.125ex\hbox{\small$/$}}\hskip-.1em}\raise-0.25ex\hbox{\footnotesize$#2$}}}
\def\fracinlines#1/#2{\mbox{\raise0.125ex\hbox{\tiny$#1$}{\hskip-.075em\mbox{\raise0.125ex\hbox{\tiny$/$}}\hskip-.1em}\raise-0.15ex\hbox{\tiny$#2$}}}
\renewcommand{\theenumi}{\arabic{enumi}}

\renewcommand{\theenumii}{\alph{enumii}}

\renewcommand{\theenumiii}{\roman{enumiii}}

\makeatletter
\def\proofname{{\it Proof}\,:}
\def\aproofname{{\it Proof of }}
\def\@proof[#1]{\aproofname{\it #1}\,:}
\renewenvironment{proof}{\par\noindent\@ifnextchar[{\@proof}{{\proofname}\quad }}{{\unskip\nobreak\hfill{\it\qedsymbol}}\par\vskip 9pt}
\ifx\undefined\bysame\newcommand{\bysame}{\leavevmode\hbox to3em{\hrulefill}\,}\fi
\makeatother

\numberwithin{equation}{section}
\numberwithin{thm}{section}
\makeatletter
\@namedef{dgo@ral}{\let\dg@VECTOR=\rightleftvector}
\def\rightleftvector(#1,#2)#3{%
   \begingroup
   \dg@XTEMP=#1\relax\multiply\dg@XTEMP\m@ne\relax
   \dg@YTEMP=#2\relax\multiply\dg@YTEMP\m@ne\relax
   \begin{picture}(0,0)%
      \thinlines
      \put(0,60){\vector(#1,#2){#3}}%
      \put(#3,-60){\vector(\dg@XTEMP,\dg@YTEMP){#3}}%
   \end{picture}%
   \endgroup}%
\@namedef{dgo@rar}{\let\dg@VECTOR=\rightrightvector}
\def\rightrightvector(#1,#2)#3{%
   \begingroup
   \dg@XTEMP=#1\relax\multiply\dg@XTEMP\m@ne\relax
   \dg@YTEMP=#2\relax\multiply\dg@YTEMP\m@ne\relax
   \begin{picture}(0,0)%
      \thinlines
      \put(0,60){\vector(#1,#2){#3}}%
      \put(0,-60){\vector(#1,#2){#3}}%
   \end{picture}%
   \endgroup}%
\makeatother
\def\id{\operatorname{id}}
\def\ad{\operatorname{ad}}
\def\incl{\operatorname{incl}}

\def\Img{\operatorname{Im}}
\def\comp{\circ}
\def\midvert{\,\mathstrut\vrule\,}

\def\numeric{\mathbb N}
\def\cardinal{\mathbb N_{0}}
\def\ordinal{\mathbb N_{1}}

\def\integral{\mathbb Z}
\def\real{\mathbb R}
\def\mathbold#1{\mathbb{#1}}
\def\Category#1{\mathsf{#1}}
\def\Object#1{\mathcal{O}(\text{\(#1\)})}
\def\Morphism#1{\mathcal{M}(\text{\(#1\)})}

\newcommand{\algebra}[1]{\mathcal{#1}}
\def\cyclic{\operatorname{\it C}}
\def\dgm{DGM}\def\dgc{DGC}\def\topology{Top}

\def\DGM{\Category{\dgm}}

\def\DGCR{\Category{\dgc}_{R}}

\def\Top{\Category{\topology}}

\def\qedsymbol{\vbox{\hrule\hbox{\vrule height1.2ex\hskip1.2ex\vrule}\hrule}\hskip0.25ex} 
\def\cotensor{\operatorname{\qedsymbol}}

\def\bigcotensor{\operatorname{\large\qedsymbol}}
\def\Center{\gamma}
\def\Tail{\beta}
\def\unital#1{#1}
\def\free#1{\widehat{#1}}

\usepackage{comment}

\def\sass{Ass}
\def\uass{uAss}

\def\vtilde{\,\,\widetilde{\vphantom{\large I}}\!\!}

\makeatletter
\def\emptyarg{}
\def\ass@@@(#1){K(#1)}
\def\ass@@[#1](#2){K_{#1}(#2)}
\def\ass@{K}
\def\ass{\@ifnextchar[{\ass@@}{\@ifnextchar({\ass@@@}{\ass@}}}
\def\mlt@@@(#1){J(#1)}
\def\mlt@@[#1](#2){J^{#1}(#2)}
\def\mlt@{J}
\def\mlt{\@ifnextchar[{\mlt@@}{\@ifnextchar({\mlt@@@}{\mlt@}}}
\def\mlta@@@(#1){J^{a}(#1)}
\def\mlta@@[#1](#2){\def\thisarg{#2}\ifx\thisarg\emptyarg J^{a}_{#1}\else J^{a}_{#1}(#2)\fi}
\def\mlta@{J^{a}}
\def\mlta{\@ifnextchar[{\mlta@@}{\@ifnextchar({\mlta@@@}{\mlta@}}}
\def\Ass@@[#1]{\Category{\sass}_{#1}}
\def\Ass@{\Category{\sass}}
\def\Ass{\@ifnextchar[{\Ass@@}{\Ass@}}
\def\Assho@@[#1]{\Category{\sass}^{ho}_{#1}}
\def\Assho@{\Category{\sass}^{ho}}
\def\Assho{\@ifnextchar[{\Assho@@}{\Assho@}}
\def\uAss@@[#1]{\Category{\uass}_{#1}}
\def\uAss@{\Category{\uass}}
\def\uAss{\@ifnextchar[{\uAss@@}{\uAss@}}
\makeatother
\title{Associahedra, Multiplihedra and units in $A_{m}$-form}
\dedicatory{In memory of Mamoru Mimura}
\author{Norio IWASE}
\email{iwase@math.kyushu-u.ac.jp}
\address{Faculty of Mathematics,
 Kyushu University,
 Motooka 744,
 Fukuoka 819-0395, Japan}
\date{\today}
\thanks{
The author is supported by the Grant-in-Aid for Scientific Research (B) \#22340014 and Grant-in-Aid for Exploratory Research \#24654013 from Japan Society for the Promotion of Science.}
\begin{document}
\baselineskip18pt
\begin{abstract}
A higher associativity was introduced by Jim Stasheff in \cite{MR0158400} with higher coherence conditions and now becomes one of the most important structures on spaces and algebras.
He also claims that the condition on unit can be weakened, using James retractile arguments \cite{MR133132}, while the proof given in \cite{MR0158400} for the equivalence of two definitions is not very clear for us.
We had been puzzled for years, and decided to prove it in a different way by constructing an $A_{m}$-structure.
To justify that our construction is natural, we bring our ideas into the theory of an internal precategory which is a weak version of Aguiar's internal category \cite{MR2696373}.
Using that construction, we show the equivalence of two definitions under the `loop-like' condition. 
That condition is not necessary to manipulate higher forms using retractile arguments as is performed in \cite{MR0158400}, but is necessary to construct an $A_{m}$-structure from the given $A_{m}$-form with {\em strict-unit} as is mentioned in Stasheff \cite{MR0270372}.
\end{abstract}
%
%
%
%
%
\maketitle
%
%
%
%
\setcounter{section}{0}

\section*{Introduction}\label{sect:introduction}

Stasheff cells denoted originally by $K_{n}$, which is now known as Associahedron, were introduced in \cite{MR0158400} in 1963 to characterize an $A_{m}$-structure for a space, while the origin goes back to Tamari \cite{MR51833} in 1951 (see also M\"uller-Hoissen, Pallo and Stasheff \cite{MR3235205}).
In 1973, Boardman and Vogt designed both an Associahedron and a Multiplihedron in \cite{MR420609} using trivalent trees.
Later, they are redesigned to be convex polytopes with faces decomposed into piecewise-linear subfaces independently by Haiman \cite{haiman1984constructing} in 1984 and by the author \cite{Iwase:1983,iwase1983map} in 1983 denoted by $K_{n}$ and $\Gamma_{n}$, to manipulate $A_{m}$-structures for maps between $A_{m}$-spaces.
In 1989, Mimura and the author \cite{MR1000378} unify these polytopes as $\bar{J}(\alpha)=\mlta(n)$, $n=\lfloor{\alpha}\rfloor$ and $a=\alpha\!-\!n$, where we have $\mlt[0](n)=\ass(n)$ an Associahedron and $\mlt[\fracinlines1/2](n)=\mlt(n)$ a Multiplihedron.
The author is also aware of constructions of a Multiplihedron by Forcey \cite{MR2475119} in 2008 and Mau and Woodward \cite{MR2660682} in 2010, which are used extensively in combinatorics.

From the begginning, we have two apparently distinct definitions for an `$A_{m}$-form' of a space: the original definition in \cite{MR0158400} requires an existence of {\em strict-unit} (strict unital conditions for $A_{m}$-forms) which is used to construct a standard projective spaces, while the second definition appeared years later in \cite{MR0270372} which is claimed to be equivalent with the original one in \cite{MR0158400}, while it requires just a {\em hopf-unit} (strict unit as an h-space) or even an {\em h-unit} (homotopy unit as an h-space).

To follow the arguments in \cite{MR0158400} on the equivalence of two definitions, we must recursively use retractile arguments introduced in James \cite{MR133132}.
However, that argument does not give an explicit construction but only an existence.
Hence, if we obtain new $3$-form (a path giving homotopy associativity) with strict unit using retractile arguments, the homotopy classes of old and new $3$-forms (two paths giving homotopy associativity) are possibly different (see \cite{MR133132}, \cite{MR0158400} or Zabrodsky \cite{MR440542}), and so we don't know whether there is an $A_{m}$-form extending the new $3$-form or not.
That means, we can not proceed to obtain higher forms for the new $3$-form.
In particular, it is shown in \cite{MR1000378} that some $A_{\infty}$-space has exotic $n$-form which does not admit $n{+}1$-form for any $n \!\geq\! 2$.
Alternatively, if one adopts another approach using $A_{m}$-structure as in the present paper, one should encounter problems to provide an $A_{m}$-structure from the given $A_{m}$-form without {\em strict-unit}, which puzzled the author for many years.

An algebraic loop was introduced in \cite{MR133132} as a set endowed with a product that has a strict unit and both right and left inverses.
We say a space $X$ is loop-like, if it is an h-space with both right and left translations are homotopy-equivalences.
So a connected CW complex h-space $X$ with {\em h-unit} is a loop-like h-space with {\em h-unit}, if $\pi_{0}(X)$ is an algebraic loop with the multiplication inherited by the product in $X$.
In addition, ``{\em loop-like}'' is equivalent to ``{\em group-like}'' for $A_{m}$-space with $m \ge 3$, 
as is claimed in \cite{MR0270372}.
But we are working also in the case when $m = 2$, we prefer to use ``{\em loop-like}'' rather than to use ``{\em group-like}'' which restrict ourself into the case when $m \ge 3$.

This paper is organized as follows: 
In \S\ref{section:preliminaries}, we summarise briefly about topological $A_{m}$-structures and $A_{m}$-forms.
We present our main result in \S\ref{section:results}.
In \S\ref{section:topological-operad}-\ref{section:algeraic-operad}, we give alternative definitions for topological and algebraic $A_{\infty}$-operads in some consistent way.
In \S\ref{sect:internal-precategory}, we introduce our internal $A_{\infty}$-categories and internal $A_{\infty}$-functors, which are kinds of generalisations of a category and a functor.
Fukaya's $A_{\infty}$-category \cite{MR1270931} would be regarded as such a precategory.
The next two sections \S\ref{section:internal-multiplications}-\ref{section:internal-actions} are devoted to introduce an $A_{\infty}$-forms for internal multiplications and internal actions to justify that our two-sided bar construction is natural.
In \S8, we sums up the above things to introduce two-sided bar constructions with {\em strict-unit} and with {\em h-unit} using $\ass(n)$ and $\mlt{0}(n)$ in either case.
Finally in \S\ref{section:units}, we give a proof of Theorem \ref{thm:ASIM}.

In appendix, we present an explicit definition of degeneracy operators using shift $1$ map, which does satisfy all the conditions required in \cite{MR0158400} and \cite{Iwase:1983,MR1000378}, and give some explanations on other miscellaneous results related to our constructions.
We also describe a relationship between trees and our {\em Associa-Multiplihedron}.

\section{Preliminaries}\label{section:preliminaries}

Before telling our main story, let us summarise briefly about terminology on h-spaces and its units from a book by Mimura \cite{Mimura86hopf}, and $A_{m}$-structures and ideas on $A_{m}$-forms of spaces and of maps from \cite{MR0158400} and also \cite{Iwase:1983,MR1000378}.

In this paper, we always assume that a space is a based space, i.e, a space with a non-degenerate base point, and that a map is a based map, i.e, a map preserving base points.
We denote by $\fatvee{n}X$ the fat-wedge of $n$ copies of a space $X$ and $X \vee Y$ a one-point-sum of two spaces $X$ and $Y$.
So, a multiplication or an action is based.

Let $X$ be a space with base point $e \in X$.
For a given multiplication $\mu : X \times X \to X$, we say that $(X,\mu,e)$ is an h-space with {\em h-unit}, if the restriction of $\mu$ to $X \vee X$ $\subset$ $X \times X$ is (based) homotopic to the folding map $\nabla_{X}$ $:$ $X \vee X$ $\to$ $X$ given by $\nabla_{X}(x,e)=\nabla_{X}(e,x)=x$.
We also say $(X,\mu,e)$ is an h-space with {\em hopf-unit} if $e$ is a two-sided strict unit of $\mu$.
Similarly, we say $f : X \to X'$ is an h-map regarding {\em h-units}, if $(X,\mu,e)$ and $(X',\mu',e')$ are h-spaces with {\em h-units} and $f(e)$ lies in the same connected component of $e' \in X'$.
We also say $f : X \to X'$ is an h-map regarding {\em hopf-units}, if $(X,\mu,e)$ and $(X',\mu',e')$ are h-spaces with {\em hopf-units} and $f$ preserves units, i.e, $f(e)=e'$.

For a CW complex $A_{m}$-space $X$, an {\em $A_{m}$-structure} of $X$ means the existence of a sequence of maps $q_{n} : (D_{n},E_{n})$ $\to$ $(B_{n},B_{n-1})$, \,$n\!\leq\!m$, such that $D_{n-1} \!\subset\! E_{n} \!\subset\! D_{n}$, $B_{n-1} \!\subset\! B_{n}$, $p_{n}=q_{n}|_{E_{n}}$ is a quasi-fibration and that $E_{n}$ is contractible in $D_{n}$, \,$n\!\leq\!m$.
By \cite{MR0158400}, we obtain that the above condition implies the existence of a sequence of maps $\{\,a(k) : \ass(k) \times {X}^{k} \to {X} \mid k \!\le\! m \,\}$ satisfying coherency conditions, which is called an $A_{m}$-form with {\em strict unit} on ${X}$.
Also by \cite{MR0158400}, the latter condition on $X$ implies the existence of a ``standard'' $A_{m}$-structure $q^{X}_{n} : (D_{n}(X),E_{n}(X))$ $\to$ $(B_{n}(X),B_{n-1}(X))$, \,$n \!\leq\! m$, for $X$.

For an $A_{m}$-map $f : X \to Y$ between CW complexes $A_{m}$-spaces, an $A_{m}$-structure of $f$ means the existence of a commutative diagram of sequences of maps $D^{f}_{n} : D_{n}(X) \to D_{n}(Y)$, $E^{f}_{n} : E_{n}(X) \to E_{n}(Y)$ and $B^{f}_{n} : B_{n}(X) \to B_{n}(Y)$ with $q^{X}_{n}$ and $q^{Y}_{n}$, \,$n\!\leq\!m$, such that $D^{f}_{n}|_{E_{n}(X)} = E^{f}_{n}$, $E^{f}_{n}|_{D_{n-1}(X)} = D^{f}_{n-1}$ and $B^{f}_{n}|_{B_{n-1}(X)} = B^{f}_{n-1}$.
By \cite{MR1000378}, we obtain that the above condition implies the existence of a sequence of maps $\{\, b_{k} : \mlt_{k} \times {X}^{k} \to {Y} \mid k \!\le\! m\,\}$ satisfying a coherency condition regarding {\em strict units}, which is called an $A_{m}$-form regarding {\em strict units} on $f$.
Also by \cite{MR1000378}, the latter condition on $f$ implies the existence of a ``standard'' $A_{m}$-structure $D^{f}_{n} : (D'_{n}(X),E'_{n}(X))$ $\to$ $(D_{n}(Y),E_{n-1}(Y))$ and $B^{f}_{n} : B'_{n}(X) \to B_{n}(Y)$, \,$n \!\leq\! m$, for $f$, where $B'_{n}(X)$, $D'_{n}(X)$ and $E'_{n}(X)$ are homeomorphic to $B_{n}(X)$, $D_{n}(X)$ and $E_{n}(X)$, respectively.

Let us remark that $B_{n}(X)$, $D_{n}(X)$ and $E_{n}(X)$ are given by two-sided bar constructions with {\em strict unit} using $\ass(n)$.
On the other hand, $B'_{n}(X)$, $D'_{n}(X)$ and $E'_{n}(X)$ are given using $\mlt_{0}(n)$ in place of $\ass(n)$ in the above construction.
So, we need two two-sided bar constructions with {\em strict units}, if we want to play with our game on spaces and maps simultaneously.
Further, if we try to play with {\em h-units} as well as {\em strict units}, we need two more two-sided bar constructions with {\em h-units} using $\ass(n)$ and $\mlt_{0}(n)$, again.

\section{Results}\label{section:results}

Let $2 \!\le\! m \!\le\! \infty$.
Then the proof of the following theorem is given in Appendix \ref{appendix:prop:AS}, which requires a precise definition of  two-sided bar construction with {\em strict-unit}.
\begin{thm}\label{thm:AS}
If a CW complex $X$ admits an $A_{m}$-structure in the sense of \cite{MR0158400}, then there exists a space $\widetilde{X}$, an $A_{m}$-form with {strict-unit} for $\widetilde{X}$, and a homotopy-equivalence inclusion $j : X \hookrightarrow \widetilde{X}$, which is an $A_{m}$-map in the sense of \cite{Iwase:1983} or \cite{MR1000378}.
\end{thm}

Our main result stated below is shown in \S\ref{section:units}:
using Lemma \ref{lem:structure-null-homotopic}, we construct an $A_{m}$-structure from a given $A_{m}$-form with {\em h-unit}, and then, by Theorem \ref{thm:AS}, we can deform the $A_{m}$-form with h-unit into an $A_{m}$-form with {\em strict-unit}\,\footnote{The author realize that, in algebraic context, Jacob Lurie has obtained a similar result using higher algebra and higher topos theories.}.
We remark that Lemma \ref{lem:structure-null-homotopic} requires a precise definition of two-sided bar construction with {\em h-unit}.

\begin{thm}\label{thm:ASIM}
Let $(X,\mu,e)$ be a CW complex loop-like h-space with {h-unit}.
If $X$ has an $A_{m}$-form $\{\,a(n),n\!\geq\!2\,\}$ with $a(2)=\mu$, then $X$ has an $A_{m}$-form with {strict-unit} and a deformation between two $A_{m}$-forms.
\end{thm}

In case when an h-space $(X,\mu,e)$ is a connected CW complex, it is loop-like by the James shear-map argument, and we obtain the following.
\begin{cor}\label{cor:AS}
Let $(X,\mu,e)$ be a connected CW complex h-space with {h-unit}.
If $X$ has an $A_{m}$-form $\{\,a(n),n\!\geq\!2\,\}$ with $a(2)=\mu$, then $X$ has an $A_{m}$-form with {strict-unit} and a deformation between two $A_{m}$-forms.
\end{cor}

We also state here a slightly stronger result of Theorem 11.5 of \cite{MR0270372} using an $A_{\infty}$-form {\em without unit}, whose proof can be found in Appendix \ref{appendix:thm:ABS}:
\begin{thm}\label{thm:ABS}
Let $\mu : X \times X \to X$ be a multiplication on a space $X$ {without unit}.
If $X$ has an $A_{\infty}$-form $\{\,a(n),n\!\geq\!2\,\}$ with $a(2)=\mu$, then $X$ is a deformation retract of a space $M$ with associative multiplication such that the inclusion $X \subset M$ has an $A_{\infty}$-form in the sense of Stasheff \cite{MR0270372}.
If the $A_{\infty}$-form has further a {strict-unit}, then we can choose $M$ as a monoid, and the inclusion becomes an $A_{\infty}$-map regarding {strict-units}.
\end{thm}

In this paper, we work in $\Top$ the category of spaces and continuous maps, unless otherwise stated.
We remark that an exponential law does not always hold in $\Top$.
let $\real_{+}=[0,\infty)$, and $\numeric_{a}$, $a=0$ \!or\! $1$, be the set of integers greater than or equal to $a$.

\section{Topological $A_{\infty}$-operads}\label{section:topological-operad}

We give here a concrete description of our {\em Associa-Multiplihedron}:

\begin{defn}[\cite{MR1000378}]
For $a  \in [0,1]$ and $n \ge 1$, a closed set $\mlta(n) \subset \real^{n}$ is given by
\begin{align*}
\mlta(n) &= \left\{\,\begin{textstyle}
(u_{1},\dots,u_{n}) \in \real_{+}^{n}\end{textstyle}
\,\midvert\, \underset{1 \le j \le n}\Max\left\{\,\underset{i=1}{\overset{j}{\textstyle\sum}}(u_{i}\!-\!1)\,\right\} \le a\!-\!1 = \underset{i=1}{\overset{n}{\sum}}(u_{i}\!-\!1) \,\right\}
\end{align*}
\end{defn}

\subsection{Topological $A_{\infty}$-operad for objects}\label{subsect:top-operad}

\begin{defn}[A topological $A_{\infty}$-operad for objects in $\Top$]
\begin{align*}&
\ass(\infty) = \{\,\ass(n) \mid n \ge 1\,\},\quad \ass(n) = \mlt[0](n).
\end{align*}
We denote $\Tail^{K}_{n} = (0,1,\dots,1) \in \ass(n)$. For example, $\Tail^{K}_{1}=(0)$ and $\Tail^{K}_{2}=(0,1)$.
\end{defn}
\begin{rem}%
\begin{enumerate*}
\item
For any $(u_{1},\dots,u_{n}) \in \ass(n)$, we have always $u_{1} = 0$.
\vitem
For any $(u_{1},\dots,u_{n}) \in \ass(n)$ with $n \ge 2$, we have $u_{n} \geq 1$.
\vitem \begin{minipage}[t]{0.9\textwidth}\baselineskip18pt
$\ass(1)=\{\Tail^{K}_{1}\}$, $\ass(2)=\{\Tail^{K}_{2}\}$, $\ass(3)=\{(0,t,2{-}t) \mid 0 \!\leq\! t \!\leq\! 1\} \homeo [0,1]$ (homeo.) 
\end{minipage}
\end{enumerate*}
\end{rem}

Let us introduce a set of triple integers $A(n)$, $n \ge 1$, defined as follows:
$$
A(n) =\left\{\, (k,r,s) \mid 1 \le k \le r \le n, \, r\!+\!s = n\!+\!1\,\right\}.
$$

Then, for $(k,r,s) \in A(n)$, we have the following function $\theta_{k} : \real^{r} \times \real^{s} \to \real^{n}$.
$$
\theta_{k}(v_{1},\dots,v_{r};u_{1},\dots,u_{s}) 
= (v_1,\dots,v_{k-1},u_1,\dots,u_{s-1},u_{s}{+}v_{k},\dots,v_{r}).
$$
By definition, $\theta_{k}(\rho,\sigma)$ is affine in both $\rho$ and $\sigma$.

\begin{expl}\label{expl:theta-K}
Let $(k,r,s) \in A(n)$ and $(\rho,\sigma) \in \real^{r} \times \real^{s}$.
\begin{enumerate}
\item\label{expl:theta-K-1}
If $r\!=\!1$, then $k = 1$, $s = n$ and $\theta_{1}(\Tail^{K}_{1},\sigma)=\sigma$.
\item\label{expl:theta-K-2}
If $s\!=\!1$, then $1 \le k \le r = n$ and $\theta_{k}(\rho,\Tail^{K}_{1})=\rho$.
\end{enumerate}
\end{expl}

A direct calculation shows the following proposition.

\begin{prop}\label{prop:theta}
$\theta_{k}(\mlta(r) \times \ass(s)) \subset \mlta(n)$ if $(k,r,s) \in A(n)$.
\end{prop}

By putting $a=0$, we obtain $\theta_{k}(\ass(r) \times \ass(s)) \subset \ass(n)$ if $(k,r,s) \in A(n)$, which enables us to define face operators for Associahedra.

\begin{defn}
We define $\partial_{k} : \ass(s) \to \Map(\ass(r),\ass(n))$, $(k,r,s) \in A(n)$, by
\begin{align*}&
\partial_{k}(\sigma)(\rho) = \theta_{k}(\rho,\sigma) \in \ass(n),\quad (\rho,\sigma) \in \ass(r) \times \ass(s).
\end{align*}
\end{defn}
By Example \ref{expl:theta-K} \ref{expl:theta-K-1} and \ref{expl:theta-K-2}, we have $\partial_{k}(\Tail^{K}_{1})=\id : \ass(n) \to \ass(n)$, $1 \le k \le n$ and $\partial_{1}= \id : \ass(n) \to \Map(\{\Tail^{K}_{1}\},\ass(n))=\ass(n)$.
If $2 \le r < n$, then $\theta_{k}(\ass(r) \times \ass(s)) \subset \partial \ass(n)$, which gives a face of an Associahedron.
We denote 
\begin{align*}&
\ass_{k}(r,s)=\theta_{k}(\ass(r) \times \ass(s)) \subset \partial \ass(n),\quad (k,r,s) \in A(n) \ \text{with} \ 2 \le r < n,
\end{align*}
Then $\partial \ass(n)$ is decomposed as follows:
$$\textstyle
\partial \ass(n) = \underset{\substack{(k,r,s) \in A(n)\\2 \le k \le r < n}}\bigcup \ass_{k}(r,s) \cup \ass_{1}(n),\qquad \ass_{1}(n) = \underset{\substack{(1,r,s) \in A(n)\\2 \le r < n}}\bigcup \ass_{1}(r,s).
$$
\begin{prop}
$\ass_{1}(n) = \left\{\,(u_{1},\dots,u_{n}) \mid \underset{1 \le j < n}\Max\left\{\, \underset{i=1}{\overset{j}{\textstyle\sum}}(u_{i}\!-\!1) \,\right\} = -1 \,\right\}$
\end{prop}
Among boundary operators, a direct calculation shows the following. 

\begin{prop}\label{prop:boundary}
For $\tau \in \ass(t), \sigma \in \ass(s)$, the following holds. 
\begin{align*}&
\partial_{k}(\sigma){\comp}\partial_{j}(\tau) = 
\begin{cases}\,
\partial_{j+s-1}(\tau){\comp}\partial_{k}(\sigma),\quad k < j,
\\[1.5ex]\,
\partial_{j}(\partial_{k-j+1}(\tau,\sigma)),\quad j \leq k < j{+}t,
\\[1.5ex]\,
\partial_{j}(\tau){\comp}\partial_{k-t+1}(\sigma),\quad k \geq j{+}t.
\end{cases}
\end{align*}
\end{prop}

Proposition \ref{prop:boundary} tells us that each face of $\ass(n)$ meets on its face with one another, which give a piecewise-linear decomposition of $n{-}3$ sphere $\partial{\ass(n)}$, $n \!\geq\! 3$.

\subsection{Topological $A_{\infty}$-operad for morphisms}

\begin{defn}[A topological $A_{\infty}$-operad for morphisms in $\Top$]
$$\begin{array}{ll}
\mlta(\infty) = \{\,\mlta(n) \mid n \!\ge\! 1\,\},\quad 
\mlt(\infty) = \{\,\mlt(n) \mid n \!\ge\! 1\,\},\quad
&\mlt(n) = \mlt[\fracinlines{1}/{2}](n). 
\end{array}$$
We denote $\Tail^{J,a}_{n} = (a,1,\dots,1) \in \mlta(n)$ and $\Tail^{J}_{n} = \Tail^{J,\fracinlines1/2}_{n} \in \mlt(n)$.
For example, $\Tail^{J,a}_{1}=(a)$.
\end{defn}
For $0 \leq a \leq 1$, there is an embedding $\zeta^{a}_{n} : \mlta(n) \hookrightarrow \ass(n{+}1)$ given by
$$
\zeta^{a}_{n}(v_{1},\dots,v_{n}) = (0,v_{1},\dots,v_{n-1},v_{n}{+}1{-}a)
$$
so that $\mlt[1](n) = \ass(n{+}1)$.

\begin{rem}
\begin{enumerate*}
\item
For any $(v_{1},\dots,v_{n}) \in \mlta(n)$, we have $0 \le v_{1} \le a$.
\vitem
For any $(v_{1},\dots,v_{n}) \in \mlta(n)$ with $n \ge 2$, we have $v_{n} \geq 1$.
\vitem \begin{minipage}[t]{0.9\textwidth}\baselineskip18pt
$\mlta(1)=\{(\Tail^{J,a}_{1})\}$, $\mlta(2)=\{(t,1{+}a{-}t)\,\vert\,0 \!\le\! t \!\le\! a\} \homeo [0,a]$ (homeo.)
\end{minipage}
\end{enumerate*}
\end{rem}

\begin{expl}\label{expl:theta-J}
Let $(k,r,s) \in A(n)$ and $(\rho,\sigma) \in \real^{r} \times \real^{s}$.
\begin{enumerate}
\item\label{expl:theta-J-1}
If $r=1$, then $k=1$, $s\!=\!n$ and $\theta_{1}(\Tail^{J,a}_{1};\sigma)=(t_{1},\dots,t_{n-1},t_{n}{+}a)$, $\sigma=(t_{1},\dots,t_{n})$.
\item\label{expl:theta-J-2}
If $s=1$, then $1 \!\le\! k \!\le\! r \!=\! n$ and $\theta_{k}(\rho;\Tail^{K}_{1})=\rho$.
\end{enumerate}
\end{expl}

Proposition \ref{prop:theta} alows us to define face operators for Multiplihedra.

\begin{defn}
We define $\displaystyle\delta^{a}_{k} : \ass(s) \to \Map(\mlta(r),\mlta(n))$, $(k,r,s) \in A(n)$, by 
\par\vskip1ex\noindent\hfil$\displaystyle
\delta^{a}_{k}(\sigma)(\rho) = \theta_{k}(\rho,\sigma) \in \mlta(n),\quad (\rho,\sigma) \in \mlta(r) \times \ass(s). 
$\hfil%
\end{defn}
By Example \ref{expl:theta-J} \ref{expl:theta-J-1} and \ref{expl:theta-J-2}, we have $\delta_{k}(\Tail^{K}_{1})=\id : \mlta(n) \to \mlta(n)$, $1 \!\le\! k \!\le\! n$ and an inclusion $\delta_{1} : \ass(n) \hookrightarrow \Map(\{\Tail^{K}_{1}\},\mlta(n)) \homeo \mlta(n)$.
If $s \!\ge\! 2$ for $0 \!<\! a \!<\! 1$, then $\theta_{k}(\mlta(r) \times \ass(s))$ $\subset$ $\partial \mlta(n)$, which gives a face of a Multiplihedron.
We denote 
\begin{align*}
\mlta[k](r,s) &= \theta_{k}(\mlta(r) \times \ass(s))
= \left\{\, (u_{1},\dots,u_{n}) \midvert u_{k}=0, \ \underset{1 \le i \le s}{\textstyle\sum}u_{k-1+i} \ge s\!-\!1 \,\right\}
\end{align*}
where $(k,r,s) \in A(n)$ with $s \!\ge\! 2$ for $0 \!<\! a \!<\! 1$.
Then for $0 \!<\! a \!<\! 1$, $\displaystyle \partial \mlta(n) \supset \underset{\substack{(k,r,s) \in A(n)\\s \ge 2}}\bigcup \mlta[k](r,s)$, which is not equal when $n \!\ge\! 1$.

Let us introduce a set of tuples of integers $B(t,n)$, $t, \,n \ge 1$, defined as follows:
$$
B(t,n) = \{\, (n_{1},\dots,n_{t}) \in \ordinal^{t} \mid n_{1}+\cdots+n_{t} = n\,\}
$$
Then for $(r_{1},\dots,r_{t}) \in B(t,n)$, We introduce a map $\Theta^{a} : \real^{t} \times \real^{r_{1}} \times \cdots \times \real^{r_{t}} \to \real^{n}$ by
\begin{align*}&
\Theta^{a}(u_{1},\dots,u_{t};v^{(1)}_{1},\dots,v^{(1)}_{r_{1}};\dots;v^{(t)}_{1},\dots,v^{(t)}_{r_{t}}) = (w_{1},\dots,w_{n})
\\[1ex]&\quad
\iff w_{j} = \begin{cases}\,
v^{(i)}_{j-s_{i-1}},&s_{i-1}\!<\!j\!<\!s_{i},
\\[2ex]\,
v^{(i)}_{r_{i}}+(1{-}a){\cdot}u_{i},&j=s_{i}.
\end{cases}
\end{align*}
where $s_{i}=r_{1}+\cdots+r_{i}$, $1 \!\le\! i \!<\! t$.
By definition, $\Theta^{a}(\sigma;\rho_{1},\dots,\rho_{t})$ is affine in $\sigma$ and all $\rho_{i}$.

\begin{expl}\label{expl:Theta}
\begin{enumerate*}
\item\label{expl:Theta-J-1}
If $t \!=\! 1$, then $\Theta^{a}(\Tail^{K}_{1};\rho)=\rho$.
\vitem\label{expl:Theta-J-2}
If $t \!=\! n$, then $\gamma(\tau) = \Theta^{a}(\tau;\Tail^{J,a}_{1},\dots,\Tail^{J,a}_{1}) \in \mlta(n)$ for $\tau \in \ass(n)$ is expressed as
\\[.5ex]\text{\,\!}\noindent\hfil$\displaystyle
\gamma(u_{1},\dots,u_{n}) = a{\cdot}(1,\dots,1) + (1{-}a){\cdot}(u_{1},\dots,u_{n}) \in \mlta(n).$\hfil
\end{enumerate*}
\end{expl}

A concrete calculation shows the following for $(r_{1},\dots,r_{t}) \in B(t,n)$.
$$
\Theta^{a}(\ass(t) \times \mlta(r_{1}) \times \cdots \times \mlta(r_{t})) \subset \mlta(n),\quad 0 \le a \le 1.
$$

\begin{defn}[\cite{MR1000378}]
We define $\delta^{a} : \mlta(r_{1}) \times \cdots \times \mlta(r_{t}) \to \Map(\ass(t),\mlta(n))$ by
$$%
\delta^{a}(\rho_{1},\dots,\rho_{t})(\sigma) = \Theta^{a}(\sigma;\rho_{1},\dots,\rho_{t})
$$%
\end{defn}

\begin{rem}
When $a=0$, the following holds for $(\sigma;\rho_{1},\dots,\rho_{t}) \in \ass(t) \times \ass(r_{1}) \times \cdots \times \ass(r_{t})$.
\par\vskip1ex\noindent\hfil$
\delta^{0}(\rho_{1},\dots,\rho_{t})(\sigma) = \partial_{1}(\rho_{1}){\comp}\cdots{\comp}\partial_{t}(\rho_{t})(\sigma) \in \ass(n).
$\hfil\par\vskip1ex\noindent
In other words, $\delta^{0}$ is just an iteration of face operators. 
\end{rem}

By Example \ref{expl:Theta} \ref{expl:theta-J-1} and \ref{expl:theta-J-2}, we have $\delta^{a}=\id : \mlta(n) \to \Map(\{\Tail^{K}_{1}\},\mlta(n)) = \mlta(n)$ and $\delta^{a}(\Tail^{J,a}_{1},\dots,\Tail^{J,a}_{1})=\gamma$.
If $t \ge 2$ for $0 \!<\! a \!<\! 1$, then $\Theta^{a}(\ass(t) \times \mlta(r_{1}) \times \cdots \times \mlta(r_{t})) \subset \partial \mlta(n)$, a face of a Multiplihedron.
For $0 \!<\! a \!<\! 1$ and $(r_{1},\dots,r_{t}) \in B(t,n)$, $t \ge 2$, we denote
\begin{align*}
\mlta(t;r_{1},\dots,r_{t}) & = \Theta^{a}(\ass(t) \times \mlta(r_{1}) \times \cdots \times \mlta(r_{t}))
\\&
= \left\{\, (u_{1},\dots,u_{n}) \midvert \forall\,i \ \underset{1 \le j \le r_{i}}{\textstyle\sum}u_{s_{i-1}+j} \ge r_{i}{-}1{+}a \,\right\},
\end{align*}
where $s_{i}=r_{1}+\cdots+r_{i}$, $1 \!\le\! i \!<\! t$.
Then for $0 \!<\! a \!<\! 1$, $\partial \mlta(n)$ is decomposed as
\begin{align*}&
\partial \mlta(n) = \bigcup_{\substack{(k,r,s) \in A(n)\\s \ge 2}} \!\!\!\mlta[k](r,s) \ \cup \ \bigcup_{\substack{(r_{1},\dots,r_{t}) \in B(t,n)\\t \ge 2}} \!\!\!\mlta(t;n_{1},\dots,n_{t}).
\end{align*}

\begin{prop}\label{prop:boundary-J}
Let $\tau \in \ass(t), \sigma \in \ass(s)$.
Then the following holds:
\begin{align*}&
\delta^{a}_{k}(\sigma){\comp}\delta^{a}_{j}(\tau) = 
\begin{cases}\,
\delta^{a}_{j+s}(\tau){\comp}\delta^{a}_{k}(\sigma)),\quad k < j,
\\[1.5ex]\,
\delta^{a}_{j}(\partial_{k-j+1}(\sigma)(\tau)),\quad j \leq k < j{+}t,
\\[1.5ex]\,
\delta^{a}_{j}(\tau){\comp}\delta^{a}_{k-t+1}(\sigma)),\quad k \geq j{+}t.
\end{cases}
\\[1ex]&
\delta^{a}_{k}(\sigma){\comp}\delta^{a}(\rho_{1},\dots,\rho_{t}) = 
\delta^{a}(\rho_{1},\dots,\rho_{j-1},\delta^{a}_{k'}(\sigma)(\rho_{j}),\rho_{j+1},\dots,\rho_{t}),
\\[1ex]&
\delta^{a}(\rho_{1},\dots,\rho_{t}){\comp}\partial_{j}(\sigma) 
= \delta^{a}(\rho_{1},\dots,\rho_{j-1},\delta^{a}(\rho_{j},\dots,\rho_{j+s-1})(\sigma),\dots,\rho_{t}),
\end{align*}
where $k'$ is given by the formula $k'=k-(r_{1}{+}\cdots{+}r_{j-1})$, \,$1 \leq k' \leq r_{j}$.
\end{prop}
Proposition \ref{prop:boundary-J} tells us that each face of $\mlta(n)$ meets on its face with one another, which give a piecewise-linear decomposition of $n{-}2$ sphere $\partial{\mlta(n)}$, $n \!\geq\! 2$.

Let us assume $0 < a < 1$. 
Then we define a union of extra faces of $\mlta(n)$ as follows.
\begin{defn}
$\mlta[0](n) = \underset{\substack{(r_{1},\dots,r_{t}) \,\in\, B(t,n)\\t \ge 2}}{\textstyle\bigcup} \!\!\mlta(n_{1},\dots,n_{t})$, $n \ge 2$, and $\mlta[0](1) = \mlta(1)=\{\Tail^{J,a}_{1}\}$.
\end{defn}

\begin{prop}\label{prop:Multiplihedra-extra-faces}
$\mlta[0](n) = \left\{\,(u_{1},\dots,u_{n}) \mid \underset{1 \le j < n}\Max\left\{\, \underset{i=1}{\overset{j}{\textstyle\sum}}(u_{i}\!-\!1) \,\right\} = a\!-\!1 \,\right\}$. Further, if $a=1$, $\zeta^{1}_{n}|_{\mlt^{1}_{0}(n)} : \mlt^{1}_{0}(n) \homeo \ass_{1}(n{+}1)$ gives a natural homeomorphism.
\end{prop}

\begin{defn}[Another topological $A_{\infty}$-operad for objects in $\Top$]
\begin{align*}&
\mlt_{0}(\infty) = \{\,\mlt_{0}(n) \mid n \ge 1\,\},\qquad \mlt_{0}(n) = \mlt^{\fracinlines1/2}_{0}(n)
\end{align*}
\end{defn}

For $(k,r,s) \in A(n)$ with $r, \,s \ge 2$, let us denote 
$$
\mlta[k](r,s)_{0}=\delta^{a}_{k}(r,s)(\mlta[0](r) \times \ass(s)) \subset \partial{\mlta[0](n)}.
$$
so that we have $\textstyle \partial \mlta[0](n) = \underset{\substack{(k,r,s) \in A(n) \\ r, \,s \ge 2}}\bigcup \mlta[k](r,s)_{0}$.

\begin{rem}
When $a=\fracinline1/2$, we might drop the superscript $a$ from $J^{a}$ and $\delta^{a}$.
\end{rem}

\subsection{Canonical degeneracy operators}

We first introduce a shift $1$ map $\xi : \real_{+}^{n}$ $\to$ $\real_{+}^{n}$, $n \geq 1$, which is a core defining canonical degeneracy operators.
\begin{defn}
$\xi(t_{1},\dots,t_{n})=(t'_{1},\dots,t'_{n})$ is determined inductively as follows:
\begin{equation}\label{eq:degeneracy0}
\begin{aligned}&
\begin{cases}\,
t'_{1} = \Max\left\{\,0,t_{1}{-}1\,\right\},& k=1,
\\[.5ex]\,
t'_{k} = \Min\left\{\,t_{k},\underset{1 \leq j \leq k}\Max\left\{\begin{textstyle}\underset{i=1}{\overset{j}{\sum}}\end{textstyle}(t_{i}\!-\!1)\right\}-\begin{textstyle}\underset{i=1}{\overset{k-1}{\sum}}\end{textstyle}(t'_{i}\!-\!1)\,\right\},& k \ge 2.
\end{cases}
\end{aligned}
\end{equation}
\end{defn}

Some properties of $\xi$ are proved in Appendix \ref{sect:shift-one}.

\begin{defn}
\begin{enumerate}
\item
$d^{J,a}_{j} : \mlta(n) \to \mlta(n{-}1)$, $1 \!\leq\! j \!\leq\! n$ is given as follows: 
\begin{equation*}
\begin{cases}\,
d^{J,a}_{1}(t_{1},\dots,t_{n}) = (t'_{2},\dots,t'_{n}),& j\!=\!1, 
\\[.5ex]\,
d^{J,a}_{j}(t_{1},\dots,t_{n}) = (t_{1},\dots,t_{j-2},t_{j-1}{+}t'_{k},t'_{j+1},\dots,t'_{n}),& j \!\geq\! 2, 
\end{cases}
\end{equation*}
where $(t'_{j},\dots,t'_{n}) = \xi(t_{j},\dots,t_{n})$ for $(t_{1},\dots,t_{n}) \in \mlta(n)$ and $1 \!\leq\! j \!\leq\! n$.\vspace{1ex}
\item
$d^{K}_{j} = d^{J,0}_{j} : \ass(n) \to \ass(n{-}1)$, $1 \leq j \leq n$.
\item
$d^{J}_{j} = d^{J,\fracinlines{1}/{2}}_{j} : \mlt(n) \to \mlt(n{-}1)$, $1 \leq j \leq n$.
\end{enumerate}
\end{defn}

To discuss degeneracies on $\mlta(n)$, $\ass(n)$ and $\mlt(n)$ simultaneously, we employ the above notations.
By Lemma \ref{lem:degeneracy1}, direct calculations imply the following theorems.
\begin{thm}\label{thm:degeneracy-boundary-K}
Let $1 \!\le\! j \!\le\! n$ and $(k,r,s) \in A(n)$ with $r, \,s \!\ge\! 2$.
For any $\rho \in \ass(r)$ and $\sigma \in \ass(s)$, the following equation holds.
\begin{align*}&\,
d^{K}_{j}(\partial_{k}(\sigma)(\rho)) = \begin{cases}\,
\partial_{k-1}(\sigma)(d^{K}_{j}(\rho)), & 1 \leq j<k \ \& \ r>2,
\\[-.0ex]\,
\sigma, & j=1 \ \& \ k=r=2,
\\[.5ex]\,
\partial_{k}(d^{K}_{j-k+1}(\sigma))(\rho), & k \leq j < k{+}s \ \& \ s>2,
\\[-.0ex]\,
\rho, & j=k, k{+}1 \ \& \ s=2,
\\[.5ex]\,
\partial_{k}(\sigma)(d^{K}_{j-t}(\rho)), & k{+}s \leq j \leq n \ \& \ r>2,
\\[-.0ex]\,
\sigma, & j=n \ \& \ k=1 \ \& \ r=2,
\end{cases}
\end{align*}
\end{thm}
\begin{thm}\label{thm:degeneracy-boundary-I}
Let $1 \!\le\! j \!\le\! n$ and $(k,r,s) \in A(n)$ with $s \!\ge\! 2$.
For any $\rho \in \mlt(r)$ and $\sigma \in \ass(s)$, the following equation holds.
\begin{align*}&
d^{J,a}_{j}(\delta^{a}_{k}(\sigma)(\rho)) = \begin{cases}\,
\delta^{a}_{k-1}(\sigma)(d^{J,a}_{j}(\rho)), & 1 \leq j<k,
\\[.5ex]\,
\delta^{a}_{k}(d^{K}_{j-k+1}(\sigma))(\rho), & k \leq j < k{+}s \ \& \ s>2,
\\[-.0ex]\,
\rho, & j=k, \,k{+}1 \ \& \ s=2.
\\[.5ex]\,
\delta^{a}_{k}(\sigma)(d^{J,a}_{j-t}(\rho)), & k{+}s \leq j \leq n,
\end{cases}
\end{align*}
\end{thm}
\begin{thm}\label{thm:degeneracy-boundary-2}
Let $0 \!\le\! a \!\le\! 1$, $1 \!\le\! j \!\le\! n$ and $(r_{1},\dots,r_{t}) \!\in\! B(t,n)$ with $t \!\ge\!2$.
For any $\rho_{i} \in \mlt(r_{i})$ ($1 \!\leq\! i \!\leq\! t$) and $\tau \in \ass(t)$, the following equation holds.
\begin{align*}&
d^{J,a}_{j}(\delta^{a}(\rho_{1},\dots,\rho_{t})(\tau)) = 
\\&\qquad
\begin{cases}\,
\delta^{a}(\rho_{1},\dots,
d^{J,a}_{j-s_{k-1}}(\rho_{k}),\dots,\rho_{t})(\tau), & s_{k-1} < j \leq s_{k} \ \& \ r_{k}>1,
\\[.5ex]\,
\delta^{a}(\rho_{1},\dots,\rho_{k-1},\rho_{k+1},\dots,\rho_{t})(d^{K}_{k}(\tau)), & j = s_{k} \ \& \ r_{k}=1 \ \& \ t>2,
\\[.0ex]\,
\rho_{2}, & j=1 \ \& \ r_{1}=1 \ \& \ t=2,
\\[.5ex]\,
\rho_{1}, & j=n \ \& \ r_{2}=1 \ \& \ t=2,
\end{cases}
\end{align*}
where $s_{k}=r_{1}+\cdots+r_{k}$, $0 \le k \le t$.
\end{thm}

By induction using Theorems \ref{thm:degeneracy-boundaries}, \ref{thm:degeneracy-boundary-K} and \ref{thm:degeneracy-boundary-I}, we obtain the following theorem.

\begin{thm}\label{thm:degeneracy-relations}
Let $j \leq n$ and $k \leq r < n$.
Then the following equations hold.
\begin{align*}&
d^{K}_{j}d^{K}_{k} = 
\begin{cases}\,
d^{K}_{k-1}d^{K}_{j}, & 1 \leq j < k,
\\\,
d^{K}_{k}d^{K}_{j+1}, & k \leq j \leq n,
\end{cases}
\\&
d^{J}_{j}d^{J}_{k} = 
\begin{cases}\,
d^{J}_{k-1}d^{J}_{j}, & 1 \leq j < k,
\\\,
d^{J}_{k}d^{J}_{j+1}, & k \leq j \leq n.
\end{cases}
\end{align*}
\end{thm}

We denote $[n]=\{\,1,2,\dots,n\,\} \subset \ordinal$, and introduce the following set for later use:
$$
D(m,\ell) = \{\,(\{j_{1},\dots,j_{m-\ell}\})\mid \{j_{1},\dots,j_{m-\ell}\} \subset [m], \ j_{1}<\cdots<j_{m-\ell}\,\}.
$$
\begin{defn}\label{defn:iteration-of-degeneracy}
For $U = (\{j_{1},\ldots,j_{m-\ell}\}) \in D(m,\ell)$, $j_{1} \!<\! \cdots \!<\! j_{m-\ell}$, we define $d^{J,a}_{U}=d^{J,a}_{j_{1}}{\comp}\cdots{\comp}d^{J,a}_{j_{m-\ell}} : \mlta(m) \to \mlta(\ell)$. When $U = \{(\emptyset)\}$, we suppose $d^{J,a}_{U}=\id : \mlta(m) \to \mlta(m)$.
We also denote by $d^{J}_{U}=d^{J,\fracinlines1/2}_{U} : \mlt(m) \to \mlt(\ell)$ and $d^{K}_{U}=d^{J,0}_{U} : \ass(m) \to \ass(\ell)$.
\end{defn}

\section{Algebraic and minimal $A_{\infty}$ operads}\label{section:algeraic-operad}

For $R$ a commutative ring with a unit, we denote by $\DGM_{R}$ the category of chain complexes over $R$ (without augmentation) and chain maps over $R$ of degree $0$ and by $\DGCR$ the category of augmented chain complexes over $R$ with coalgebra structures and the chain maps over $R$ of degree $0$ preserving coalgebra structures.

Then the normalised chain complex $N_{*}(X;R)$ of a space $X$ admits a natural coalgebra structure over $R$ and a chain map induced from a continuous map preserves the coalgebra structures over $R$.
Thus they are in the category $\DGCR$.

\subsection{Algebraic and minimal $A_{\infty}$ operads for objects}

We give here an algebraic version of Associahedron and its minimal version.

\begin{defn}%
Let $\algebra{K}(n)=N_{*}(\ass(n);R)$ a DGC over $R$.
Then, the algebraic $A_{\infty}$ operad is the sequence of DGCs $\algebra{K}(\infty)=\{\,\algebra{K}(n) \mid n \ge 1 \,\}$ with the following data:
\begin{equation*}
\theta_{k *} : \algebra{K}(r){\otimes}\algebra{K}(s) \to \algebra{K}(n),\quad (k,r,s) \in A(n),\ r, \,s \geq 2.
\end{equation*}
\end{defn}
\begin{prop}\label{prop:orientation-obj}
There exists a chain $\mu(n) \in \algebra{K}(n)$ of $n{-}2$ dimension satisfying the following equation.
\begin{equation}\label{eq:orientation-obj}
\partial(\mu(n)) = \sum_{(k,r,s) \in A(n), \, r, \,s \geq 2}(-1)^{\varepsilon_{k}(r,s)} \theta_{k *}(\mu(r){\otimes}\mu(s))
\end{equation}
where we denote $\varepsilon_{k}(r,s) = {k{-}1{+}(s{-}2)(r{-}k)}$ (see \cite{MR0158400}).
\end{prop}
\begin{proof}
Without loss of generality, we may assume that $R = \integral$.
To show this proposition, we use the induction on $n$.
\par\noindent{Case} $n\!=\!2$.
We first choose a generator $\mu(2) \in \algebra{K}(2)$ represented by $\id : \{\ast\} \homeo \ass(2)$.
\par\noindent{Case} $n\!=\!k{+}1$, $k \geq 2$.
By induction hypothesis, we may assume that there exists elements $\mu(r)$ and $\mu(s)$ for $r, \,s \!<\! n$ appearing in the right-hand-side of the equation (\ref{eq:orientation-obj}).
Let 
$$
\displaystyle z = \sum_{(k,r,s) \in A(n), \, r, \,s \geq 2}(-1)^{\varepsilon_{k}(r,s)} \theta_{k *}(\mu(r){\otimes}\mu(s)).
$$
Then a concrete computation yields that $\partial z = 0$.
Since ${K}(n)$ is homeomorphic to an $n{-}2$ disk which is contractible, the cycle $z$ must be a boundary cycle, and hence we obtain a chain $\mu(n) \in \algebra{K}(n)$ of $n{-}2$ dimension for each $n \geq 2$ satisfying (\ref{eq:orientation-obj}).
\end{proof}
Then we may assume that $\mu(n)$ is a cycle in $N_{*}(\ass(n),\partial{\ass(n)};R)$ for any coefficient ring $R$ with a unit and we call $\mu(n)$ the orientation cycle of the pair $(\ass(n),\partial{\ass(n)})$.
\begin{defn}\label{defn:minimal-operad-obj}
Let an algebraic $A_{\infty}$ operad $\algebra{K}_{c}(\infty)=\{\algebra{K}_{c}(n)\}$ be the sequence of $\algebra{K}_{c}(n)$ the smallest subcomplex of $\algebra{K}(n)$ including the orientation cycle 
$$
\mu(n) \in N_{*}(\ass(n),\partial{\ass(n)};R) \subset N_{*}(\ass(n);R)
$$
and all the boundary images $\mu_{k}(r,s)=\theta_{k *}(\mu(r){\otimes}\mu(s)) \in N_{*}(\ass(n);R)$.
We call $\algebra{K}_{c}(\infty)$ the minimal $A_{\infty}$ operad for objects.\end{defn}

By Proposition \ref{prop:orientation-obj}, $\mu(n) \in \algebra{K}(n)$ satisfies the following equation.
\begin{align*}&
[\mu(n){:}\mu_{k}(r,s)] = (-1)^{\varepsilon_{k}(r,s)} = (-1)^{k{-}1{+}(s{-}2)(r{-}k)},
\end{align*}
where $[\alpha{:}\beta]$ denotes the incidence number of $\beta$ appearing in $\partial\alpha$.

\subsection{Algebraic and minimal $A_{\infty}$ operads for morphisms}

Similarly to the case of $\algebra{K}(n)$, we introduce the following
\begin{defn}%
Let an algebraic $A_{\infty}$ operad $\algebra{J}(\infty)=\{\algebra{J}(n)\}$ be the sequence of DGCs $\algebra{J}(n)=N_{*}(\mlt(n);R)$ over $R$ and DGC-maps over $R$:
\begin{align*}&
\theta_{k *} : \algebra{J}(r){\otimes}\algebra{K}(s) \to \algebra{J}(n),\qquad (k,r,s) \in A(n), \ s \geq 2,
\\&
\delta_{*} : \algebra{K}(t){\otimes}\algebra{J}(r_{1}){\otimes}\cdots{\otimes}\algebra{J}(r_{t}) \to \algebra{J}(n),\qquad (r_{1},\dots,r_{t}) \in B(t,n), \ t \geq 2.
\end{align*}
\end{defn}
\begin{prop}\label{prop:orientation-mor}
There is a chain $\nu(n) \in \algebra{J}(n)$ of $n{-}2$ dimension satisfying the following formula.
\begin{equation}\label{eq:orientation-mor}
\begin{aligned}&
\partial(\nu(n)) \ = \!\!\!\!\!\!\sum_{(k,r,s) \in A(n), \,s \ge 2}(-1)^{\epsilon_{k}(r,s)} \delta_{k *}(\nu(r){\otimes}\mu(s))
\\&\qquad\qquad + 
\!\!\!\!\!\!\sum_{(r_{1},\dots,r_{t}) \in B(t,n), \,t \ge 2}
(-1)^{\varepsilon(t;r_{1},{\cdots},r_{t})} \delta_{*}(\mu(t){\otimes}\nu(r_{1}){\otimes}\cdots{\otimes}\nu(r_{t})).
\end{aligned}
\end{equation}
\end{prop}
\begin{proof}
Without loss of generality, we may assume that $R = \integral$.
To show this proposition, we use the induction on $n$.
\par\noindent{Case} $n\!=\!1$.
We just choose a generator $\nu(1) \in \algebra{J}(1)$ represented by $\id : \{\ast\} \homeo \mlt(1)$.
\par\noindent{Case} $n\!=\!k{+}1$, $k \!\geq\! 1$.
By the induction hypothesis, we may assume that there exists elements $\nu(r)$ and $\nu(n_{1}),{\cdots},\nu(n_{t})$ appearing in the right-hand-side of (\ref{eq:orientation-mor}).
Let 
$$\displaystyle z \ = \!\!\sum_{\substack{(k,r,s) \in A(n)\\s \ge 2}}\!\!\!\!(-1)^{\varepsilon_{k}(r,s)} \delta_{k *}(\nu(r){\otimes}\mu(s)) \ + \!\!\!\!\!\sum_{\substack{(r_{1},\dots,r_{t}) \in B(t,n)\\t \ge 2}}\!\!\!\!\!\!
(-1)^{\varepsilon(t;r_{1},{\cdots},r_{t})} \delta_{*}(\mu(t){\otimes}\nu(r_{1}){\otimes}\cdots{\otimes}\nu(r_{t})).
$$
Then a concrete computation yields that $\partial z = 0$.
Since ${J}(n)$ is homeomorphic to an $n{-}1$ disk which is contractible, the cycle $z$ must be a boundary cycle, and hence we obtain a chain $\nu(n) \in \algebra{J}(n)$ of $n{-}1$ dimension for each $n \geq 1$ satisfying (\ref{eq:orientation-mor}).
\end{proof}
Then we may assume that $\nu(n)$ is a cycle in $N_{*}(\mlt(n),\partial{\mlt(n)};R)$ for any coefficient ring $R$ with a unit and we call $\nu(n)$ the orientation cycle of the pair $(\mlt(n),\partial{\mlt(n)})$.
\begin{defn}\label{defn:minimal-operad-mor}
Let an algebraic operad $\algebra{J}_{c}(\infty)=\{\algebra{J}_{c}(n)\}$ be the sequence of $\algebra{J}_{c}(n)$ the smallest sub-chain complex of $\algebra{J}(n)$ including the orientation cycle 
$$
\nu(n) \in N_{*}(\mlt(n),\partial{\mlt(n)};R) \subset \algebra{J}(n)
$$
and all the boundary images $\nu_{k}(r,s)=\delta_{k *}(\nu(r){\otimes}\mu(s)) \in \algebra{J}(n)$ and $\nu(t;r_{1},{\cdots},r_{t})=\delta_{*}(\mu(t){\otimes}\nu(r_{1}){\otimes}\cdots{\otimes}\nu(r_{t}))$.
We call $\algebra{J}_{c}(\infty)$ the minimal $A_{\infty}$ operad for morphisms.
\end{defn}
By Proposition \ref{prop:orientation-mor}, $\nu(n) \in \algebra{J}(n)$ satisfies the following equations.
\begin{align*}&
[\nu_{n}{:}\nu_{k}(r,s)] = (-1)^{\varepsilon_{k}(r,s)} = (-1)^{k{-}1{+}(s{-}2)(r{-}j)},
\\&
[\nu_{n}{:}\nu(t;r_{1},{\cdots},r_{t})] = (-1)^{\varepsilon(t;r_{1},{\cdots},r_{t})} = (-1)^{\underset{i=1}{\overset{t-1}{\sum}}(t{-}i)(n_{i}{-}1)}.
\end{align*}

Finally in this section, we define $\algebra{J}_{0}(n) = N_{*}(\mlt_{0}(n);R)$ a DGC over $R$.

\begin{prop}\label{prob:orientation-mor0}
The chain $\nu_{0}(n) = \!\underset{(r_{1},\dots,r_{t}) \in B(t,n), \,t \ge 2}\sum\!\!\!\!\!\!\!\!\!\!
(-1)^{\varepsilon(t;r_{1},{\cdots},r_{t})} \delta_{*}(\mu(t){\otimes}\nu(r_{1}){\otimes}\cdots{\otimes}\nu(r_{t})) \in \algebra{J}_{0}(n)$ of $n{-}2$ dimension to satisfy the following equation.
\begin{equation}\label{eq:orientation-mor0}
\partial(\nu_{0}(n)) \ = \!\!\!\!\!\!\sum_{\substack{(k,r,s) \in A(n)\\r, \,s \geq 2}}(-1)^{\varepsilon_{k}(r,s)} \theta_{k *}(\nu_{0}(r){\otimes}\mu(s))
\end{equation}
\end{prop}

\section{Internal $A_{\infty}$-category}\label{sect:internal-precategory}

In this section, we introduce our notion of an internal $A_{\infty}$-category, which was born from a conversation with Goro Nishida in a homotopy theory meeting at Fukuoka.
Our $A_{\infty}$-category might be a topological version of Fukaya's $A_{\infty}$-category (see \cite{MR1270931} or \cite{MR1898414}).
Honestly, we are trying to categorise the topological nature of $A_{\infty}$-structures to justify that our constructions are natural enough.

Let $\Category{C}$ be a monoidal category equipped with a associative tensor product $\otimes : \Category{C}\times\Category{C} \to \Category{C}$ with unit object $1$.
First, we introduce an assumption on our play ground $\Category{C}$.

\begin{defn}
An object $O$ in $\Category{C}$ is called `flat', if $\lim_{\lambda} (O{\otimes}A_{\lambda})$ = $O{\otimes}(\lim_{\lambda} A_{\lambda})$ and $\lim_{\lambda} (A_{\lambda}{\otimes}O)$ = $(\lim_{\lambda} A_{\lambda}){\otimes}O$.
Further we say that $\Category{C}$ is `regular', if every object $O$ in $\Category{C}$ is flat.
\end{defn}

From now on, we assume that $\Category{C}$ is a regular monoidal category with equalisers.

\subsection{Coalgebras and comodules}

First we introduce the notions of a comonoid and a coalgebra in $\Category{C}$.

\begin{defn}
\begin{enumerate}
\item
A triple $(O,\nu,\epsilon)$, or simply $O$, is called an `coalgebra', if $O$ is an object in $\Category{C}$, morphisms $\nu: O \to O{\otimes}O$ and $\epsilon : O \to 1$ in $\Category{C}$ satisfies that $(\nu{\otimes}1_{O}){\comp}\nu$ = $(1_{O}{\otimes}\nu){\comp}\nu$ and $(1_{O}{\otimes}\epsilon){\comp}\nu$ = $(\epsilon{\otimes}1_{O}){\comp}\nu$ = $1_{O}$ the identity morphism.
In that case, the morphisms $\nu$ and $\epsilon$ are often called a comultiplication and a counit, respectively, of a coalgebra $O$.
\item
Let $\Category{C}$ be symmetric regular monoidal.
Then for a coalgebra $O=(O,\nu,\epsilon)$, we define a coalgebra $O^{*}=(O,\nu^{*},\epsilon)$ by $\nu^{*}$ = $\Gamma{\comp}\nu$ where $\Gamma : O{\otimes}O \to O{\otimes}O$ is the symmetry isomorphism of $\Category{C}$.
\item
A coalgebra $O$ in a symmetric regular monoidal category $\Category{C}$ is called cocommutative, if $O$ = $O^{*}$.
\end{enumerate}
\end{defn}

Then we define a `bicomodules' under a coalgebra $O$:
\begin{defn}
Let $O=(O,\nu,\epsilon)$ be a coalgebra in $\Category{C}$ and $M$ be an object in $\Category{C}$.
\begin{enumerate}
\item
$M=(M,\tau)$ is right comodule under $O$, if $\tau : M \to M{\otimes}O$ satisfies $(\tau{\otimes}1_{O}){\comp}\tau=(1_{M}{\otimes}\nu){\comp}\tau$ and $(1_{M}{\otimes}\epsilon){\comp}\tau = 1_{M}$.
\item
$M=(M,\sigma)$ is left comodule under $O$, if $\sigma : M \to O{\otimes}M$ satisfies $(1_{O}{\otimes}\sigma){\comp}\sigma=(\nu{\otimes}1_{M}){\comp}\sigma$ and $(\epsilon{\otimes}1_{M}){\comp}\sigma = 1_{M}$.
\item
$M=(M;\tau,\sigma)$ is bicomodule under $O$, if $(M,\tau)$ is a right comodule and $(M,\sigma)$ is a left comodule.
\end{enumerate}
\end{defn}

For a morphism, we also introduce some more notions.

\begin{defn}
Let $O=(O,\nu,\epsilon)$ and $O=(O',\nu',\epsilon')$ be coalgebras in $\Category{C}$ and $M=(M;\tau,\sigma)$ and $M'=(M';\tau',\sigma')$ be objects in $\Category{C}$.
\begin{enumerate}
\item
A morphism $\phi : O \to O'$ of coalgebras in $\Category{C}$ is called a `homomorphism', if it satisfies $\nu'{\comp}\phi=(\phi{\otimes}\phi){\comp}\nu$ and $\epsilon'{\circ}\phi = \epsilon$.
\item
A pair $(f,\phi)$ of a morphism $f : M \to M'$ and a homomorphism $\phi : O \to O'$ where $M$ and $M'$ are right comodules by $\tau$ and $\tau'$ under $O$ and $O'$, respectively in $\Category{C}$ is called (right) `equivariant', if it satisfies $(f{\otimes}\phi){\comp}\tau=\tau'{\comp}f$.
\item
A pair $(f,\phi)$ of a morphism $f : M \to M'$ and a homomorphism $\phi : O \to O'$ where $M$ and $M'$ are left comodules by $\sigma$ and $\sigma'$ under $O$ and $O'$, respectively in $\Category{C}$ is called (left) `equivariant', if it satisfies $(\phi{\otimes}f){\comp}\sigma=\sigma'{\comp}f$.
\item
A pair $(f,\phi)$ of a morphism $f : M \to M'$ and a homomorphism $\phi : O \to O'$ where $M=(M;\tau,\sigma)$ and $M'=(M';\tau',\sigma')$ are bicomodules under $O$ and $O'$ respectively in $\Category{C}$ is called a `biequivariant' morphism or an `internal homomorphism', if it is both right and left equivariant.
\end{enumerate}
\end{defn}
We remark that, in a slightly different context, Tamaki \cite{Tamaki:2009aa} and Asashiba \cite{MR2975601} have given a similar idea which is used to generalize quiver for a generalization of the Grothendiek construction.

\subsection{Internal $A_{\infty}$-category and internal $A_{\infty}$-functor}

We introduce the notion of an internal precategory in a regular monoidal category $\Category{C}$.

\begin{defn}
\begin{enumerate}
\item
A pair $(M,O)$ of $O$ a coalgebra with a comultiplication $\nu$ and a counit $\epsilon$ and $M$ a bicomodule by $\tau : M \to M{\otimes}O$ and $\sigma : M \to O{\otimes}M$ under $O$ equipped with a morphism $\iota : O \rightarrow M$ is called an `internal precategory' in $\Category{C}$ and denoted by $(M,O;\sigma,\tau,\iota)$ if it satisfies the following two conditions:
\begin{align*}&
\sigma{\comp}\iota = (1_{O}{\otimes}\iota){\comp}\nu,\quad \tau{\comp}\iota = (\iota{\otimes}1_{O}){\comp}\nu.
\end{align*}
\item
Let $(M,O;\sigma{,}\tau{,}\iota)$ and $(M',O';\sigma'{,}\tau'{,}\iota')$ be internal precategories in $\Category{C}$.
A pair of morphisms $(f : M \to M',\phi : O \to O')$ in $\Category{C}$ is called an `internal prefunctor' in $\Category{C}$ and denoted by $(f,\phi)$ if it is an internal homomorphism satisfying the following condition:
$$
f{\comp}\iota=\iota'{\comp}\phi. 
$$
\end{enumerate}
\end{defn}
From now on, we often abbreviate $(M,O;\sigma,\tau,\iota)$ by $(M,O)$ or simply by $M$, and $(f,\phi)$ by $f$.
Let us denote by ${}^{ip}\Category{C}$ the category of internal precategories and internal prefunctors in $\Category{C}$.

\begin{expl}
\begin{enumerate}
\item
$O=(O,O;1_{O},\nu,\nu)$ is in ${}^{ip}\Category{C}$.
\item
For given two internal precategories $M$ = $(M,O;\sigma,\tau,\iota)$ and $M'$ = $(M',O;\sigma',\tau',\iota')$ in $\Category{C}$, the cotensor of a right comodule $M$ and a left comodule $M'$ gives an internal precategory $M\cotensor_{O}M' = (M\cotensor_{O}M',O;\sigma'',\tau'',\iota'')$ in $\Category{C}$ as the equalizer of $1{\otimes}\sigma'$ and $\tau{\otimes}1$:
\begin{align*}&
\begin{diagram}
\node{M\cotensor_{O}M'}
	\arrow{e,..}
\node{M{\otimes}M'}
	\arrow{e,tb,rar}{1{\otimes}\sigma'}{\tau{\otimes}1}
\node{M{\otimes}O{\otimes}M'.}
\end{diagram}
\end{align*}
\item
For any internal precategory $M=(M,O;\sigma,\tau,\iota)$, the equalizers $M\cotensor_{O}O \to M{\otimes}O$ and $O\cotensor_{O}M \to O{\otimes}M$ are naturally equivalent to $\tau : M \to M{\otimes}O$ and $\sigma : M \to O{\otimes}M$, respectively (see \cite{MR2696373}).
\item
For two internal prefunctors $f$ = $(f,\phi)$ : $(M_{1},O_{1};\sigma_{1},\tau_{1},\iota_{1})$ $\to$ $(M_{2},O_{2};\sigma_{2},\tau_{2},\iota_{2})$ and $f'$ = $(f',\phi)$ : $(M'_{1},O_{1};\sigma_{2},\tau_{2},\iota'_{1})$ $\to$ $(M'_{2},O_{2};\sigma'_{2},\tau'_{2},\iota'_{2})$ in $\Category{C}$, the cotensor of a right equivariant map $f$ and a left equivariant map $f'$ gives an internal prefunctor $f\cotensor_{\phi}f' : M_{1}\cotensor_{O_{1}}M'_{1}$ $\to$ $M_{2}\cotensor_{O_{2}}M'_{2}$ in $\Category{C}$: 
\begin{align*}&
\begin{diagram}
\node{M_{1}\cotensor_{O_{1}}M'_{1}}
	\arrow{e}
	\arrow{s,l}{f\cotensor_{\phi}f'}
\node{M_{1}{\otimes}M'_{1}}
	\arrow{e,tb,rar}{1{\otimes}\sigma'_{1}}{\tau_{1}{\otimes}1}
	\arrow{s,l}{f{\otimes}f'}
\node{M_{1}{\otimes}O_{1}{\otimes}M'_{1}.}
	\arrow{s,r}{f{\otimes}\phi{\otimes}f'}
\\
\node{M_{2}\cotensor_{O_{2}}M'_{2}}
	\arrow{e,..}
\node{M_{2}{\otimes}M'_{2}}
	\arrow{e,tb,rar}{1{\otimes}\sigma'_{2}}{\tau_{2}{\otimes}1}
\node{M_{2}{\otimes}O_{2}{\otimes}M'_{2}.}
\end{diagram}
\end{align*}
\item
For an internal precategory $M$ = $(M,O;\sigma,\tau,\iota)$ in $\Category{C}$, we obtain an internal precategory $\bigcotensor^{n}_{O}M$ = $(\bigcotensor^{n}_{O}M,O;\iota_{n},\sigma_{n},\tau_{n})$ with $\bigcotensor^{1}_{O}M = M$ in $\Category{C}$ by induction on $n$.
\item
For an internal prefunctor $f$ = $(f,\phi)$ : $(M,O)$ $\to$ $(M',O')$ in $\Category{C}$, we obtain an internal prefunctor $\bigcotensor^{n}_{\phi}f$ : $(\bigcotensor^{n}_{O}M,O)$ $\to$ $(\bigcotensor^{n}_{O'}M',O')$ in $\Category{C}$ with $\bigcotensor^{1}_{\phi}f = f$ by induction on $n$.
\end{enumerate}
\end{expl}

\subsection{Internal multiplication and internal $A_{\infty}$-action}

Let $X=(X,O)$ $=$ $(X,O;\sigma,\tau,\iota)$ be an internal precategory in $\Category{C}$ with an internal multiplication $\mu : X\cotensor_{O}X \to X$ in $\Category{C}$.

\begin{defn}\label{defn:hopf}
Let $X=(X,O;\sigma,\tau,\iota)$ be an internal precategory in $\Category{C}$.
\begin{enumerate}
\item\label{defn:internal-semi-precategory}
Let $X=(X,O;\sigma,\tau,\iota)$ be a bicomodule in $\Category{C}$.
If $X$ is equipped with an internal prefunctor $\mu : X\cotensor_{O}X \to X$, then $X=(X,\mu)$ is called an `internal semi-category' in $\Category{C}$, and the prefunctor $\mu$ is called an `internal multiplication' of $X$.
\item\label{defn:internal-h-precategory}
An internal semi-precategory $X=(X,\mu)$ is an `internal h-precategory' in $\Category{C}$, if an internal multiplication $\mu : \bigcotensor^{2}_{O}X \to X$ satisfies
$$
\mu{\comp}(\iota\cotensor_{O}1_{X})=1_{X}=\mu{\comp}(1_{X}\cotensor_{O}\iota),
$$
where we regard $O\cotensor_{O}X = X = X\cotensor_{O}O$.
We denote such an internal h-precategory by $(X,\mu)$ or simply by $X$.
\end{enumerate}
\end{defn}

We remark that an internal h-precategory $M=(M,\mu)$ is an internal category or a `monad' in the sense of Aguiar \cite{MR2696373}, if the internal multiplication $\mu$ satisfies the strict associativity condition:
$$
\mu{\comp}(1_{X}\cotensor_{O}\mu)=\mu{\comp}(\mu\cotensor_{O}1_{X}).
$$

It would be natural to extend these ideas slightly more:
for internal precategories $X=(X,O;\sigma_{X},\tau_{X},\iota_{X})$, $Y=(Y,O;\sigma_{Y},\tau_{Y},\iota_{Y})$ and $Z=(Z,O;\sigma_{Z},\tau_{Z},\iota_{Z})$ and internal homomorphisms $p : Y \to X$ and $q : Z \to X$ in $\Category{C}$, we call an internal homomorphism $\mu : Y{\cotensor_{O}}Z \to X$ with $\mu{\circ}(1_{Y}{\cotensor_{O}}\iota_{Z}) = p$ and $\mu{\circ}(\iota_{Y}{\cotensor_{O}}1_{Z}) = q$ an internal pairing with axes $(p,q)$, or left axis $p$ and right axis $q$ in $\Category{C}$.

We call an internal homomorphism $\mu' : Y{\cotensor_{O}}X \to Y$ with right axis $q : X \to Y$ an internal right pairing of $X$ on $Y$ along $q$, and an internal homomorphism $\mu'' : X{\cotensor_{O}}Z \to Z$ with left axis $p : X \to Z$ an internal left pairing of $X$ on $Z$ along $p$.
In either case, we don't care about the other axis in the above definition.

\begin{defn}\label{defn:act}
Let $X=(X,O;\sigma_{X},\tau_{X},\iota_{X})$, $Y=(Y,O;\sigma_{Y},\tau_{Y},\iota_{Y})$ and $Z=(Z,O;\sigma_{Z},\tau_{Z},\iota_{Z})$ be internal precategories in the category $\Category{C}$.
\begin{enumerate}
\item
For an internal prefunctor $q : X \to Y$ in $\Category{C}$, we call $(Y,\mu',X)$ an internal (right) action of $X$ on $Y$ along $q$ in $\Category{C}$, if $\mu' : Y{\otimes_{O}}X \to Y$ is an internal right pairing in $\Category{C}$ such that
\begin{align}&
\mu'{\comp}(1_{Y} \cotensor_{O} \iota_{X})=1_{Y},
\\&
\mu'{\comp}(\iota_{Y} \cotensor_{O} 1_{X})=q,
\end{align}
where we regard $Y \cotensor_{O} O = Y$ and $O \cotensor_{O} X=X$.
We denote such an internal action by $(Y,\mu',X)$ or simply by $(Y,X)$.
\item
For an internal prefunctor $p : X \to Z$ in $\Category{C}$, we call $(X,\mu'',Z)$ an internal (left) action of $X$ on $Z$ along $k$ in $\Category{C}$, if $\mu'' : X \cotensor_{O} Z$ $\to$ $Z$ is an internal left pairing in $\Category{C}$ such that
\begin{align}
\mu''{\comp}(\iota \cotensor_{O} 1_{Z})=1_{Z},
\\
\mu''{\comp}(1_{X} \cotensor_{O} \iota_{Z})=p,
\end{align}
where we regard $O \cotensor_{O} Z = Z$ and $X \cotensor_{O} O = X$.
We denote such an internal action by $(X,\mu'',Z)$ or simply by $(X,Z)$.
\end{enumerate}
\end{defn}
\begin{rem}
Even if we drop the condition on $q$ or $p$ to be an internal prefunctor, each becomes an internal prefunctor by the above definition, since so does $\mu'$ or $\mu''$.
\end{rem}

\begin{expl}
An internal h-precategory $(X,\mu)$ gives an internal right and left actions of $X$ on $X$ along the identity internal functor in $\Category{C}$.
\end{expl}

\section{$A_{m}$-forms for multiplications}\label{section:internal-multiplications}
\subsection{$A_{m}$-form for internal multiplication}

We introduce the notion of an $A_{m}$-form $(1 \!\leq\! m \!\leq\! \infty)$ for an internal multiplication in $\Top$.
Let $X=(X,\mu)$ be an internal precategory with an internal multiplication $\mu : X \times_{O} X \to X$ in $\Top$.

\begin{defn}\label{defn:form-top-obj}
We call $\{a(n);1 \!\leq\! n \!\leq\! m\}$ ($a(1)=1_{X}$) an ${A}_{m}$-form for $\mu$, if $\begin{textstyle}a(n) : \ass(n) \times \prod^n_{O}{X} \to X\end{textstyle}$ satisfies the following formulas for any $(\rho,\sigma) \in \ass(r) \times \ass(s)$, $n=r{+}s{-}1$ and $\mathbold{x}=(x_{1},\dots,x_{n}) \in \prod^{n}_{O} X$, $n {\geq} 2$.
\begin{align}\label{eq:topological-obj1}&
a(2) = \mu,
\\ 
\label{eq:topological-obj3}&
a(n)(\partial_{k}(\rho,\sigma);\mathbold{x})= a(r)(\rho;a_{k}(s)(\sigma;\mathbold{x})),
\end{align}
where $a_{k}(s)(\sigma;\mathbold{x})$ is given by 
$$
(x_{1},\dots,x_{k-1},a_{s}(\sigma;x_{k},\dots,x_{k+s-1}),x_{k+s},\dots,x_{n}).
$$
\end{defn}
An internal category with the above $A_{m}$-form for an internal multiplication might be called an internal $A_{m}$-category {\em without unit}.
When $O$ is the one-point set, an internal $A_{m}$-category $X$ {\em without unit} may be called an $A_{m}$-space {\em without unit}.

\subsection{Internal ${A}_{\infty}$-category with unit}

Now we define an internal $\free{A}_{m}$-category for $1 \leq m \leq \infty$ in $\Top$.
\begin{defn}\label{defn:form-unit-top-obj}
Let $\{a(n);1 \!\leq\! n \!\leq\! m\}$ ($a(1)=1_{X}$) be an $A_{\infty}$-form for $\mu : X \times_{O} X \to X$.
\begin{enumerate}
\item
We call an internal $A_{m}$-category $X=(X;\{a(n)\})$ an ``internal $A_{m}$-category with {hopf-unit}'', if $X$ further satisfies the following hopf-unit condition.
\begin{equation}\label{eq:topological-obj2-h}
a(2)(1_{X}{\times_{O}} \iota_{X})= 1_{X} = a(2)(\iota_{X}{\times_{O}} 1_{X})
\end{equation}
\item
We call an internal $A_{m}$-category $X=(X;\{a(n)\})$ an ``internal $A_{m}$-category with {strict-unit}'', if $X$ satisfies the following strict-unit condition.
\begin{equation}
\label{eq:topological-obj2}
\begin{array}[t]{l}\hskip-1ex
a(n)(1_{\ass(n)} \times ((\prod^{j-1}_{O} 1_{X}){\times_{O}}\iota_{X}{\times_{O}}(\prod^{n-j}_{O} 1_{X})))
\\[1ex]\qquad\qquad=
a(n{-}1)(d^{K}_{j} \times (\prod^{n-1}_{O} 1_{X})),\ 1 \leq j \leq n,
\end{array}
\end{equation}
\end{enumerate}
\end{defn}
If an internal precategory is an internal ${A}_{m}$-category with {\em strict-(or hopf-)unit} in $\Top$ for every $m \geq 2$, then it is called an internal ${A}_{\infty}$-category with {\em strict-(or hopf-)unit} in $\Top$.
When $m=2$, an internal $\free{A}_{2}$-category with {\em strict-(or hopf-)unit} in $\Top$ is an internal h-precategory in $\Top$.

When $O$ is the one-point set, then an internal ${A}_{m}$-category $X$ with {\em strict-(or hopf-)unit} in $\Top$ is called an ${A}_{m}$-space with {\em strict-(or hopf-)unit}.
When further $m=2$, an ${A}_{2}$-space with {\em strict-(or hopf-)unit} is nothing but an h-space.
Let us introduce one more definition of a unit: a space $X=(X,\{e\})$ with a based multiplication $\mu : X  \times  X \to X$ is called an h-space with {\em h-unit}, if $\mu(x,e) \sim x \sim \mu(e,x)$, in other words, the restriction of $\mu$ to $X \!\vee\! X \subset X \times X$ is (based) homotopic to the folding map $\nabla_{X} : X\!\vee\!X \to X$.
So we call $X$ an $A_{m}$-space with {\em h-unit}, if $X$ is an $A_{m}$-space {\em without unit} and, at the same time, $X$ is an h-space with {\em h-unit} with its $A_{2}$-form.

An internal $A_{1}$-category is nothing but an internal precategory and an internal $A_{1}$-space is nothing but a based space.

\begin{expl}
\begin{enumerate}
\item A topological group is an ${A}_{\infty}$-space with {strict-unit}.
\item A topological monoid homotopy equivalent to a CW complex is a loop-like ${A}_{\infty}$-space with {strict-unit}.
\item A space homotopy equivalent to an ${A}_{m}$-space with {strict-(or hopf-)unit} in the category of well-pointed spaces is also an ${A}_{m}$-space with {strict-(or hopf-)unit}.
\item The loop space $\Omega{X}$ of any simply-connected CW complex $X$ is not actually an $A_{\infty}$-space with {hopf-unit} but an ${A}_{\infty}$-space with {h-unit}.
\end{enumerate}
\end{expl}

\begin{thm}[Stasheff \cite{MR0158400}]
An ${A}_{\infty}$-space with {strict-unit} is homotopy equivalent to a topological monoid.
\end{thm}

\subsection{$A_{\infty}$-form for internal homomorphism}

We introduce the notion of an $A_{m}$-form ($1 \!\leq\! m \!\leq\! \infty$) for an internal homomorphism in $\Top$.
We assume that $(X,\{a(n)\})$ and $(X',\{b(n)\})$ be internal $\free{A}_{m}$-categories {\em without unit} in $\Top$, $1 \!\leq\! m \!\leq\! \infty$.
Let $f : X \to X'$ be an internal homomorphism in $\Top$.

\begin{defn}\label{defn:topological-hom-mor}
We call $\{h(n);1 \!\leq\! n \!\leq\! m\}$ an ${A}_{m}$-form for $f$, if internal homomorphisms $\begin{textstyle}h(n) : \mlt(n) \times \prod^n_{O}{X} \to X'\end{textstyle}$ satisfy the following formulas for any $(\rho,\sigma) \in \mlt(r) \times \ass(s)$, $n=r{+}s{-}1$, $(\tau;\rho_{1},\dots,\rho_{t}) \in \ass(t) \times \mlt(r_{1}) \times {\cdots} \times \mlt(r_{t})$ and $\mathbold{x}=(x_{1},\dots,x_{n}) \in \prod^{n}_{O} X$:
\begin{align}\label{eq:topological-hom-mor1}&
h(1) = f,
\\ 
\label{eq:topological-hom-mor3}&
h(n)(\delta_{k}(\rho,\sigma);\mathbold{x})= h(r)(\rho;a_{k}(s)(\sigma;\mathbold{x})),
\\
\label{eq:topological-hom-mor4}&
h(n)(\delta(\tau;\rho_{1},\dots,\rho_{t});\mathbold{x}) = b(t)(\tau;h(\rho_{1}),\dots,h(\rho_{t})),
\end{align}
where $h(\rho_{k})$, \,$1 \!\le\! k \!\le\! t$, are given by
\begin{align*}&
h(\rho_{k})=h(r_{k})(\rho_{k};x_{r_{1}+\cdots+r_{k-1}+1},\dots,x_{r_{1}+\cdots+r_{k}}),\quad 1 \!\leq\! k \!\le\! t.
\end{align*}
\end{defn}
A internal homomorphism with the above $A_{m}$-form might be called an internal $A_{m}$-functor {\em disregarding unit}.
When $O$ is the one-point set, then an internal $A_{m}$-functor {\em disregarding unit} may be called an $A_{m}$-map disregarding base-point.

\subsection{Internal ${A}_{\infty}$-functor}

We now define an internal ${A}_{m}$-functor in $\Top$ for $1 \leq m \leq \infty$.

Let $(X,\{a(n)\})$ and $(X',\{b(n)\})$ be internal $\free{A}_{m}$-categories with {\em hopf-units} and let $f : X \to X'$ be an internal homomorphism with an $\free{A}_{m}$-form $\{h(n)\}$ for $f$ {\em disregarding units}.
\begin{defn}\label{defn:hopf-form-top-mor}
We call $f=(f,\{h(n)\})$ an ``internal $A_{m}$-functor regarding {hopf-units}'', if $f$ is an internal prefunctor, i.e.,
\begin{equation}\label{eq:topological-mor}
f{\circ}\iota_{X} = \iota_{X'}{\circ}\phi
\end{equation}
\end{defn}
If an internal homomorphism is an internal ${A}_{m}$-functor regarding {\em hopf-unit} in $\Top$ for any $m \geq 1$, then it is an internal prefunctor and is called an internal ${A}_{\infty}$-functor regarding {\em hopf-units}.
When $m=1$, an internal ${A}_{1}$-functor regarding {\em hopf-units} is nothing but an internal prefunctor in $\Top$.

Let $(X,\{a(n)\})$ and $(X',\{b(n)\})$ be internal $\free{A}_{m}$-categories with {\em strict-units}, $m \geq 2$ and let $f=(f,\phi) : X=(X,O) \to (X',O')=X'$ be an internal homomorphism with an $\free{A}_{m}$-form $\{h(n)\}$ for $f$ {\em disregarding units}.
\begin{defn}\label{defn:form-top-mor}
We call $f=(f,\{h(n)\})$ an ``internal $A_{m}$-functor with {strict-unit}'', if $f$ satisfies the following condition with $h(0)=\iota_{X}{\circ}\phi$.
\begin{equation}
\label{eq:unital-r-act-form-mor2}
\begin{array}[t]{l}\hskip-1ex
h(n)(1_{\mlt(n)} \times ((\prod^{j-1}_{O} 1_{X}){\times_{O}}\iota_{X}{\times_{O}}(\prod^{n-j}_{O} 1_{X})))
\\[1ex]\qquad=
h(n{-}1)(d^{J}_{j} \times (\prod^{n-1}_{O} 1_{X})),\ 1 \!\leq\! j \!\leq\! n,
\end{array}
\end{equation}
\end{defn}
If an internal homomorphism is an internal ${A}_{m}$-functor regarding {\em strict-unit} in $\Top$ for any $m \geq 1$, then it is an internal prefunctor and is called an internal ${A}_{\infty}$-functor regarding {\em strict-unit} in $\Top$.
When $m=1$, an internal ${A}_{1}$-functor regarding {\em strict-unit} in $\Top$ is nothing but an internal prefunctor in $\Top$.

When both $O$ and $O'$ are one-point sets, then an internal ${A}_{m}$-functor regarding {\em strict-(or hopf-)unit} in $\Top$ is called an ${A}_{m}$-map regarding {\em strict-(or hopf-)unit}.
When further $m=1$, an $\free{A}_{1}$-map regarding {\em strict-(or hopf-)unit} is nothing but a based map.

\section{$A_{m}$-forms for internal actions}\label{section:internal-actions}
\subsection{${A}_{m}$-form for an internal action}

We introduce a notion of an $A_{m}$-form for an internal right (or left) pairing along an internal homomorphism in $\Top$, $1 \!\leq\! m \!\leq\! \infty$.

Let $(X,\{\,a(n) \mid 1 \!\leq\! n \!\leq\! m\!-\!1\,\})$ ($a(1)=1_{X}$) be an internal $\free{A}_{m-1}$-category {\em without unit} in $\Top$, and $\mu' : Y \times_{O} X \to Y$ be an internal right action along $p : X \to Y$ in $\Top$.
\begin{defn}\label{defn:r-act-form-top-obj}
We call $\{a'(n) \mid 1 \!\leq\! n \!\leq\! m\}$ an $\free{A}_{m}$-form for the internal right pairing $\mu'$ along $p$ in $\Top$, if they satisfy the following formulas for any $(k,r,s) \in A(n)$ with $r, \,s \!\ge\! 2$,  $(\rho,\sigma) \in \ass(r) \times \ass(s)$ and $\mathbold{x}=(y;x_{1},\dots,x_{n-1}) \in Y{\times_{O}}\prod^{n-1}_{O} X$. 
\begin{align}\label{eq:r-act-form-obj0}&
\begin{textstyle}a'(n) : \ass(n) \times (Y {\times_{O}} \prod^{n-1}_{O}{X}) \to Y\end{textstyle},
\\\label{eq:r-act-form-obj1}&
a'(1)=1_{Y}\quad\text{and}\quad a'(2) = \mu',
\\ 
\label{eq:r-act-form-obj3}&
a'(n)(\partial_{k}(\rho,\sigma);\mathbold{x})=
a'(r)(\rho;a'_{k}(s)(\sigma;\mathbold{x})),
\end{align}
where $a'_{k}(s)(\sigma;\mathbold{x})$ is given as follows. 
\begin{equation*}
a'_{k}(s)(\sigma;\mathbold{x}) = \left\{\!\begin{array}{ll}
(a'(s)(\sigma;y,x_{1},\dots,x_{s-1}),\dots,x_{n-1}),& k\!=\!1,
\\[1ex]
(y,x_{1},\dots,a(s)(\sigma;x_{k-1},\dots,x_{k+s-2}),\dots,x_{n-1}),& k\!>\!1.
\end{array}\right.
\end{equation*}
\end{defn}
A pair of internal precategories with the above $\free{A}_{m}$-form for a right pairing might be called an internal (right) $\free{A}_{m}$-action {\em without unit}.
When $O$ is the one-point set, an internal (right) $\free{A}_{m}$-action {\em without unit} may be called a (right) $\free{A}_{m}$-action {\em without unit}.

Let $\mu'' : X \times_{O} Z \to Z$ be an internal left pairing along $q : X \to Z$ in $\Top$.
\begin{defn}\label{defn:l-act-form-top-obj}
We call $\{a''(n);1 \!\leq\! n \!\leq\! m\}$ an $\free{A}_{m}$-form for the internal left pairing $\mu''$ along $q$ in $\Top$, if they satisfy the following formulas for any $(k,r,s) \in A(n)$ with $r, \,s \!\ge\! 2$,  $(\rho,\sigma) \in \ass(r) \times \ass(s)$ and $\mathbold{x}=(x_{1},\dots,x_{n-1};z) \in \prod^{n}_{O} X{\times_{O}}Z$. 
\begin{align}\label{eq:l-act-form-top-obj0}&
\begin{textstyle}a''(n) : \ass(n) \times (\prod^{n-1}_{O}{X} {\times_{O}} Z) \to Z\end{textstyle}
\\\label{eq:l-act-form-top-obj1}&
a''(1)=1_{Z}\quad\text{and}\quad a''(2) = \mu'',
\\ 
\label{eq:l-act-form-top-obj3}&
a''(n)(\partial_{k}(\rho,\sigma);\mathbold{x})=
a''(r)(\rho;a''_{k}(s)(\sigma;\mathbold{x})),
\end{align}
where $a''_{k}(s)(\sigma;\mathbold{x})$ is given by 
\begin{equation*}
a''_{k}(s)(\sigma;\mathbold{x}) = \left\{\!\begin{array}{ll}
(x_{1},\dots,a(s)(\sigma;x_{k},\dots,x_{k+s-1}),\dots,x_{n-1},z),& k\!<\!r,
\\[1ex]
(x_{1},\dots,a''(s)(\sigma;x_{n-s+1},\dots,x_{n-1},z)),& k\!=\!r.
\end{array}\right.
\end{equation*}
\end{defn}
A pair of internal precategories with the above $\free{A}_{m}$-form for a left pairing might be called an internal (left) $\free{A}_{m}$-action {\em without unit}.
When $O$ is the one-point set, an internal (left) $\free{A}_{m}$-action {\em without unit} may be called a (left) $\free{A}_{m}$-action {\em without unit}.

\subsection{${A}_{\infty}$-action of an internal ${A}_{\infty}$-category}

Now, we define an internal ${A}_{m}$-action of an internal $\free{A}_{m}$-category in $\Top$, $2 \leq m \leq \infty$.
We assume that $X$, $Y$ and $Z$ are internal precategories with the same `object' $O$ in $\Top$, and $p : X \to Y$ and $q : X \to Z$ be internal prefunctors.
We also assume that $(X,\{a(n)\})$ is an internal ${A}_{m-1}$-category with {\em hopf-unit} in $\Top$.

Let $(Y,X)=(Y,X;\{a'(n)\})$ be an internal right $A_{m}$-action of $X$ on $Y$ along an internal prefunctor $p$ {\em without unit} and let $(X,Z)=(X,Z;\{a''(n)\})$ be an internal left $A_{m}$-action of $X$ on $Z$ along an internal prefunctor $q$ {\em without unit}.
\begin{defn}\label{defn:hopf-act-form-top-obj}
\begin{enumerate}
\item
We call $(Y,X)$ an ``internal right ${A}_{m}$-action with {hopf-unit}'', if $(Y,X)$ satisfies the following condition.
\begin{align}\label{eq:r-hopf-act-form-top-obj1}&
\text{$a'(2)$ has axes $(1_{Y},p)$}
\end{align}
\item
We call $(X,Z)$ an ``internal left $\free{A}_{m}$-action with {hopf-unit}'', if $(X,Z)$ satisfies the following condition.
\begin{align}\label{eq:l-hopf-act-form-top-obj1}&
\text{$a''(2)$ has axes $(q,1_{Z})$}
\end{align}
\end{enumerate}
\end{defn}
If an action of an internal ${A}_{m}$-category in $\Top$ is an internal ${A}_{m}$-action with {\em hopf-unit} for any $m \geq 2$, then it is called an internal ${A}_{m}$-action with {\em hopf-unit} of an internal ${A}_{m}$-category in $\Top$.

\begin{defn}\label{defn:strict-act-form-top-obj}
Let $(Y,X)=(Y,X;\{a'(n)\})$ be an internal right $A_{m}$-action of $X$ on $Y$ along an internal prefunctor $p$ {without unit}, and let $(X,Z)=(X,Z;\{a''(n)\})$ be an internal left $A_{m}$-action of $X$ on $Z$ along an internal prefunctor $q$ {without unit}.
\begin{enumerate}
\item
We call $(Y,X)$ an ``internal right ${A}_{m}$-action with {strict-unit}'', if $(Y,X)$ satisfies the following strict-unit condition.
\begin{equation}
\label{eq:strict-r-act-form-top-obj2}
\begin{array}[t]{l}\hskip-1ex
a'(n)(1_{\ass(n)} \times (1_{Y} {\times_{O}} (\prod^{j-2}_{O} 1_{X}){\times_{O}}\iota_{X}{\times_{O}}(\prod^{n-j}_{O} 1_{X})))
\\[1ex]\qquad=
a'(n{-}1)(d^{K}_{j} \times (1_{Y} {\times_{O}} (\prod^{n-1}_{O} 1_{X}))),\ 1 {<} j \leq n,
\end{array}
\end{equation}
\item
We call $(X,Z)$ an ``internal left ${A}_{m}$-action with {strict-unit}'', if $(X,Z)$ satisfies the following strict-unit condition.
\begin{equation}
\label{eq:unital-r-act-form-obj2}
\begin{array}[t]{l}\hskip-1ex
a''(n)(1_{\ass(n)} \times ((\prod^{j-1}_{O} 1_{X}){\times_{O}}\iota_{X}{\times_{O}}(\prod^{n-j-1}_{O} 1_{X}) {\times_{O}}1_{Z}))
\\[1ex]\qquad=
a''(n{-}1)(d^{K}_{j} \times ((\prod^{n-1}_{O} 1_{X}) {\times_{O}} 1_{Z})),\ 1 \leq j {<} n,
\end{array}
\end{equation}
\end{enumerate}
\end{defn}
If an action of an internal ${A}_{m}$-category in $\Top$ is an internal ${A}_{m}$-action with {\em strict-unit} for any $m \geq 2$, then it is called an internal ${A}_{m}$-action with {\em strict-unit} of an internal ${A}_{m}$-category in $\Top$.
When $m=2$, an internal ${A}_{2}$-action with {\em hopf-unit} or with {\em strict-unit} of an internal ${A}_{\infty}$-category in $\Top$ is nothing but an internal action of an internal precategory in $\Top$.

\subsection{${A}_{\infty}$-equivariant form for internal homomorphism}

We introduce the notion of an $A_{m}$-equivariant form for an internal homomorphism between internal $A_{m}$-actions {\em without units}, $2 \!\leq\! m \!\leq\! \infty$.

We assume that $X$, $Y$, $Z$ are internal precategories with an `object' $O$, $X'$, $Y'$ and $Z'$ be internal precategories with another `object' $O'$, and $q : X \to Y$, $p : X \to Z$, $q' : X' \to Y'$, $p' : X' \to Z'$, $f=(f,\phi) : X=(X,O) \to (X',O')=X'$ be internal prefunctor in $\Top$.

We also assume that $(f,\{h(n)\}) : (X,\{a(n)\}) \to (X',\{b(n)\})$ is an internal $\free{A}_{m-1}$ functor {\em disregarding units} between internal $\free{A}_{m-1}$-categories {\em without units}.

Let $(Y,X;\{a'(n)\})$ and $(Y',X';\{b'(n)\})$ be internal right ${A}_{m}$-pairings {\em without units} and $g : Y \to Y'$ be an internal homomorphism in $\Top$.
\begin{defn}\label{defn:r-act-form-top-mor}
We call $\{h'(n);1 \!\leq\! n \!\leq\! m\}$ an ${A}_{m}$-equivariant form for $(g,f)$, if $\begin{textstyle}h'(n) : \mlt(n) \times (Y {\times_{O}}\prod^{n-1}_{O}{X}) \to Y'\end{textstyle}$ satisfies the following formulas for any $(\rho,\sigma) \in \mlt(r) \times \ass(s)$, $n=r{+}s{-}1$, $(\tau;\rho_{1},\dots,\rho_{t}) \in \ass(t) \times \mlt(r_{1}) \times {\cdots} \times \mlt(r_{t})$ and $\mathbold{x}=(y,x_{1},\dots,x_{n-1}) \in Y {\times_{O}}\prod^{n-1}_{O} X$:
\begin{align*}
&
h'(1) = g,
\\ 
&
h'(n)(\delta_{k}(\rho,\sigma);\mathbold{x})= h'(r)(\rho;a'_{k}(s)(\sigma;\mathbold{x})),
\\
&
h'(n)(\delta(\tau;\rho_{1},\dots,\rho_{t});\mathbold{x})= b'(t)(\tau;h'(\rho_{1}),h(\rho_{2}),\dots,h(\rho_{t})),
\end{align*}
where $h'(\rho_{1})$ and $h(\rho_{k})$, \,$1 \!<\! k \!\le\! t$, are given by 
\begin{align*}&
h'(\rho_{1})=h'(r_{1})(\rho_{1};y,x_{1},\dots,x_{r_{1}-1}),\quad k\!=\!1,
\\&
h(\rho_{k})=h(r_{k})(\rho_{k};x_{r_{1}+\cdots+r_{k-1}},\dots,x_{r_{1}+\cdots+r_{k}-1}),\quad 1 \!<\! k \!\leq\! t.
\end{align*}
\end{defn}
A pair of internal homomorphisms with the above $A_{m}$-equivariant form might be called an internal (right) $A_{m}$-equivariant functor {\em disregarding units}.
When $O$ and $O'$ are one-point sets, an internal (right) $A_{m}$-equivariant functor {\em disregarding units} may be called a (right) $A_{m}$-equivariant map {\em disregarding units}.

Let $(X,Z;\{a''(n)\})$ and $(X',Z';\{b''(n)\})$ be internal left ${A}_{m}$-pairings {\em without units} and $\ell : Z \to Z'$ be an internal homomorphism in $\Top$.
\begin{defn}\label{defn:l-act-form-top-mor}
We call $\{h''(n);1 \!\leq\! n \!\leq\! m\}$ an ${A}_{m}$-equivariant form for $(\ell,f)$, if $\begin{textstyle}h''(n) : \mlt(n) \times (\prod^{n-1}_{O}{X}{\times_{O}}Z) \to Z'\end{textstyle}$ satisfies the following formulas for any $(\rho,\sigma) \in \mlt(r) \times \ass(s)$, $n=r{+}s{-}1$, $(\tau;\rho_{1},\dots,\rho_{t}) \in \ass(t) \times \mlt(r_{1}) \times {\cdots} \times \mlt(r_{t})$ and $\mathbold{x}=(x_{1},\dots,x_{n-1},z) \in \prod^{n-1}_{O} X {\times_{O}}Z$:
\begin{align*}%
&
h''(1) = \ell,
\\ 
&
h''(n)(\delta_{k}(\rho,\sigma);\mathbold{x})= h''(r)(\rho;a''_{k}(s)(\sigma;\mathbold{x})),
\\
&
h''(n)(\delta(\tau;\rho_{1},\dots,\rho_{t});\mathbold{x})= b''(t)(\tau;h(\rho_{1}),\dots,h(\rho_{t-1}),h''(\rho_{t})),
\end{align*}
where $h(\rho_{k})$, \,$1 \!\le\! k \!<\! t$, and $h''(\rho_{t})$ are given by 
\begin{align*}&
h(\rho_{k})=h(r_{k})(\rho_{k};x_{r_{1}+\cdots+r_{k-1}+1},\dots,x_{r_{1}+\cdots+r_{k}}),\quad 1 \!\leq\! k \!<\! t,
\\&
h''(\rho_{t})=h''(r_{t})(\rho_{t};x_{n-r_{t}+1},\dots,x_{n-1},z),\quad k\!=\!t.
\end{align*}
\end{defn}

A pair of internal homomorphisms with the above $A_{m}$-equivariant form might be called an internal $A_{m}$-equivariant functor {\em disregarding units}.
When $O$ and $O'$ are one-point sets, an internal $A_{m}$-equivariant functor {\em disregarding units} may be called a $A_{m}$-equivariant map {\em disregarding units}.

\subsection{Internal ${A}_{\infty}$-equivariant functor}

Now we define an internal ${A}_{m}$-equivariant functor between internal $\free{A}_{m}$-actions in $\Top$, $2 \!\leq\! m \!\leq\! \infty$.

We assume that $X$, $Y$, $Z$ are internal precategories with an `object' $O$, $X'$, $Y'$ and $Z'$ be internal precategories with another `object' $O'$, and $q : X \to Y$, $p : X \to Z$, $q' : X' \to Y'$, $p' : X' \to Z'$, $f=(f,\phi) : X=(X,O) \to (X',O')=X'$ be internal prefunctor in $\Top$.

Let $(f,\{h(n)\}) : (X,\{a(n)\})$ $\to$ $(X',\{b(n)\})$ be an internal $\free{A}_{m-1}$ functor regarding {\em hopf-units} between internal $\free{A}_{m-1}$-categories with {\em hopf-units}.
Let $(g,f;\{h'(n)\}) : (Y,X;\{a'(n)\})$ $\to$ $(Y',X';\{b'(n)\})$ be an internal ${A}_{m}$-equivariant functor {\em disregarding units} between internal right ${A}_{m}$-actions with {\em hopf-units} and let $(f,\ell;\{h''(n)\}) : (X,Z;\{a''(n)\}) \to (X',Z';\{b''(n)\})$ be an internal ${A}_{m}$-equivariant functor {\em disregarding units} between internal left ${A}_{m}$-actions with {\em hopf-units}.
\begin{defn}\label{defn:act-form-top-mor}
\begin{enumerate}
\item
We call $(g,f)=(g,f;\{h'(n)\})$ an ``internal right $\free{A}_{m}$-equivariant functor regarding {hopf-units}'' in $\Top$, if $g$ and $f$ are internal prefunctors, i.e.,
\begin{align*}%
f{\circ}\iota_{X}=\iota_{X'}{\circ}\phi,\quad g{\circ}\iota_{Y}=\iota_{Y'}{\circ}\phi
\end{align*}
\item
We call $(f,\ell)=(f,\ell;\{h''(n)\})$ an ``internal left $\free{A}_{m}$-equivariant functor regarding {hopf-units}'' in $\Top$, if $\ell$ and $f$ are internal prefunctors, i.e.,
\begin{align*}%
f{\circ}\iota_{X}=\iota_{X'}{\circ}\phi,\quad \ell{\circ}\iota_{Z}=\iota_{Z'}{\circ}\phi
\end{align*}
\end{enumerate}
\end{defn}
If an internal prefunctor is a right (or left) $\free{A}_{m}$-equivariant functor regarding {\em hopf-units} for any $m \geq 1$, then it is called a right (or left) $\free{A}_{\infty}$-equivariant regarding {\em hopf-units}.

Let $(g,f;\{h'(n)\}) : (Y,X;\{a'(n)\}) \to (Y',X';\{b'(n)\})$ be an internal ${A}_{m}$-equivariant functor {\em disregarding units} between internal right ${A}_{m}$-actions with {\em strict-units} and let $(f,\ell;\{h''(n)\}) : (X,Z;\{a''(n)\}) \to (X',Z';\{b''(n)\})$ be an internal ${A}_{m}$-equivariant functor {\em disregarding units} between internal left ${A}_{m}$-actions with {\em strict-units}.
\begin{defn}%
\begin{enumerate}
\item
We call $(g,f)=(g,f;\{h'(n)\})$ an ``internal right $\free{A}_{m}$-equivariant functor regarding {strict-units}'', if $g$ and $f$ are internal prefunctors, i.e.,
\begin{equation*}
\begin{array}[t]{l}\hskip-1ex
h'(n)(1_{\mlt(n)} \times (1_{Y}{\times_{O}}(\prod^{j-2}_{O} 1_{X}){\times_{O}}\iota_{X}{\times_{O}}(\prod^{n-j}_{O} 1_{X})))
\\[1ex]\qquad=
h'(n{-}1)(d^{J}_{j} \times (1_{Y} {\times_{O}} (\prod^{n-1}_{O} 1_{X}))),\ 1 {<} j \leq n,
\end{array}
\end{equation*}
\item
We call $(f,\ell)=(f,\ell;\{h''(n)\})$ an ``internal left $\free{A}_{m}$-equivariant functor regarding {strict-units}'', if $\ell$ and $f$ are internal prefunctors, i.e.,
\begin{equation*}
\begin{array}[t]{l}\hskip-1ex
h''(n)(1_{\mlt(n)} \times ((\prod^{j-1}_{O} 1_{X}){\times_{O}}\iota_{X}{\times_{O}}(\prod^{n-j-1}_{O} 1_{X}) {\times_{O}}1_{Z}))
\\[1ex]\qquad=
h''(n{-}1)(d^{J}_{j} \times ((\prod^{n-1}_{O} 1_{X}) {\times_{O}} 1_{Z})),\ 1 \leq j {<} n,
\end{array}
\end{equation*}
\end{enumerate}
\end{defn}

If an internal prefunctor is a right (or left) $\free{A}_{m}$-equivariant functor regarding {\em strict-units} for any $m \geq 1$, then we call it a right (or left) $\free{A}_{\infty}$-equivariant regarding {\em strict-units}.

\section{Topological $A_{\infty}$-operad categories}

To follow the original definitions given by Jim Stasheff, we introduce here some small categories made from the topological ${A}_{\infty}$-operads and faces with or without degeneracies for an object or a morphism.

\subsection{Topological ${A}_{\infty}$-operad category without degeneracies for an object}

We define a topological small category $\Category{\free{BK}}$:
\begin{defn}\label{defn:ext-operadic-object}
Let $\Category{\free{BK}}$ be the category consisting of $\Object{\Category{\free{BK}}}$ the set of objects and $\Morphism{\Category{\free{BK}}}$ the set of morphisms given as follows:
\begin{description}
\item[(objects)]
$\Object{\Category{\free{BK}}} = \{\,\underline{0},\underline{1},\underline{2},\dots\,\} \homeo \cardinal$,
\item[(morphisms)]
For any two non-negative integers $m$ and $n$, \par
$\Morphism{\Category{\free{BK}}}(\underline{m},\underline{n})$ $=$ $\underset{(a_{0},\dots,a_{m+1}) \in B(m{+}2,n{+}2)}{\textstyle\coprod} \ass(a_{0}) \times \cdots \times \ass(a_{m+1})$.
\\
For any $(\rho_{0},\dots,\rho_{\ell+1}) : \underline{\ell} \to \underline{m}$, $\rho_i \in \ass(r_{i})$, $(r_{0},\dots,r_{\ell+1}) \in B(\ell{+}2,m{+}2)$ and $(\sigma_{0},\dots,\sigma_{m+1}) : \underline{m} \to \underline{n}$, $\sigma_{j} \in \ass(a_{j})$, $(a_{0},\dots,a_{m+1}) \in B(m{+}2,n{+}2)$, the composition $(\tau_{0},\dots,\tau_{\ell+1})=(\sigma_{0},\dots,\sigma_{m+1}){\comp}(\rho_{0},\dots,\rho_{\ell+1})$ 
is given by 
\par\vskip1ex\noindent\hfil
$\tau_i = \delta^{0}(\sigma^{(i)}_{0},\dots,\sigma^{(i)}_{r_{i}-1})(\rho_{i}) 
= \partial_{1}(\sigma^{(i)}_{0}){\comp}\partial_{2}(\sigma^{(i)}_{1}){\comp}\cdots{\comp}\partial_{r_{i}}(\sigma^{(i)}_{r_{i}-1})(\rho_{i}),$
\hfil\par\vskip1ex\noindent
where $\sigma^{(i)}_{j} = \sigma_{r_{0}+\cdots+r_{i-1}+j}$ for $0 \leq j < r_{i}$.
\end{description}
\end{defn}

Truncating $\Category{\free{BK}}$, we obtain a series of topological small categories as follows.
\begin{defn}
For each $m \geq 0$, we define a full-subcategory $\Category{\free{BK}}_{m}$ of $\Category{\free{BK}}$, whose object set is $\{\underline{0},\underline{1},\dots,\underline{m}\} \homeo \{0,1,\dots,m\} \subset \cardinal$.
\end{defn}

\subsection{Topological ${A}_{\infty}$-operad category without degeneracies for a morphism}

We define a topological small category $\Category{\free{BJ}}$:
\begin{defn}\label{defn:ext-operadic-morphism}
Let $\Category{\free{BJ}}$ be the category consisting of $\Object{\Category{\free{BJ}}}$ the set of objects and $\Morphism{\Category{\free{BJ}}}$ the set of morphisms given as follows:
\begin{description}
\item[(objects)]
$\Object{\Category{\free{BJ}}} = \{\underline{0},\underline{0}',\underline{1},\underline{1}',\underline{2},\underline{2}',\dots\} \homeo \cardinal\times{\cyclic_2}$,
\item[(morphisms)]
For any two non-negative integers $m$ and $n$,
\par\vskip.5ex
$\Morphism{\Category{\free{BJ}}}(\underline{m},\underline{n}) = 
\Morphism{\Category{\free{BK}}}(\underline{m},\underline{n})$,
\par\vskip1ex
$\Morphism{\Category{\free{BJ}}}(\underline{m}',\underline{n}') = 
\Morphism{\Category{\free{BK}}}(\underline{m}',\underline{n}')$,
\par\vskip1ex
$\Morphism{\Category{\free{BJ}}}(\underline{m}',\underline{n}) = \emptyset$ and 
\par\vskip1ex
$\Morphism{\Category{\free{BJ}}}(\underline{m},\underline{n}') = 
\underset{(a_{0},\dots,a_{m+1}) \in B(m{+}2,n{+}2)}{\textstyle\coprod} \mlt(a_{0}) \times \cdots \times \mlt(a_{m{+}1})$ 
\\[1ex]\noindent
with the following relations:
\begin{enumerate}
\item
For any $(\rho_{0},\dots,\rho_{\ell+1}) : \underline{\ell} \to \underline{m}'$, $\rho_{i} \in \mlt(r_{i})$, \,$(r_{0},\dots,r_{\ell+1}) \in B(\ell{+}2,m{+}2)$ and $(\sigma_{0},\dots,\sigma_{m+1}) : \underline{m}' \to \underline{n}'$ with $\sigma_{j} \in \ass(a_{j})$, \,$(a_{0},\dots,a_{m+1}) \in B(m{+}2,n{+}2)$, the composition $(\tau_{0},\dots,\tau_{\ell+1})=(\sigma_{0},\dots,\sigma_{m+1}){\comp}(\rho_{0},\dots,\rho_{\ell+1})$ is given by 
\par\vskip1ex\noindent\hfil
$\tau_i=\delta_{1}(\sigma^{(i)}_{0}){\comp}\cdots{\comp}\delta_{r_{i}+1}(\sigma^{(i)}_{r_{i}-1})(\rho_{i}),$
\hfil\par\vskip1ex\noindent
where $\sigma^{(i)}_{j} = \sigma_{r_{0}+\cdots+r_{i-1}+j}$ for $0\!\leq\!j\!<\!r_{i}$.\vspace{1ex}
\item
For any $(\rho_{0},\dots,\rho_{\ell+1}) : \underline{\ell} \to \underline{m}$, $\rho_i \in \ass(r_{i})$, \,$(r_{0},\dots,r_{\ell+1}) \in B(\ell{+}2,m{+}2)$ and $(\sigma_{0},\dots,\sigma_{m+1}) : \underline{m} \to \underline{n}'$ with $\sigma_{j} \in \mlt(a_{j})$, \,$(a_{0},\dots,a_{m+1}) \in B(m{+}2,n{+}2)$, the composition \par\quad$(\tau_{0},\dots,\tau_{\ell+1})=(\sigma_{0},\dots,\sigma_{m+1}){\comp}(\rho_{0},\dots,\rho_{\ell+1})$ 
is given by 
\par\vskip1ex\noindent\hfil
$\tau_i=\delta(\sigma^{(i)}_{0},\dots,\sigma^{(i)}_{r_{i}-1})(\rho_{i})$,
\hfil\par\vskip1ex\noindent
where $\sigma^{(i)}_{j} = \sigma_{r_{0}+\cdots+r_{i-1}+j}$ for $0\!\leq\!j\!<\!r_{i}$.
\end{enumerate}
\end{description}
\end{defn}
\begin{rem}
There are two inclusion functors $Bj, Bj' : \Category{\free{BK}} \to \Category{\free{BJ}}$ between topological small categories determined by
\begin{align*}&
Bj(\underline{n}) = \underline{n}, \ \text{and} \ Bj'(\underline{n}) = \underline{n}',
\end{align*}
\end{rem}

Truncating $\Category{\free{BJ}}$, we obtain a series of topological small categories as follows.
\begin{defn}
For each $m \geq 0$, we define a full-subcategory $\Category{\free{BJ}}_{m}$ of $\Category{\free{BJ}}$, whose object set is $\{\underline{0},\underline{0}',\underline{1},\underline{1}',\dots,\underline{m},\underline{m}'\} \homeo \{0,1,\dots,m\} \times \cyclic_2 \subset \cardinal \times \cyclic_2$.
\end{defn}
\begin{rem}
By restricting $Bj$ and $Bj'$, we obtain two inclusion functors $Bj_{m}, Bj'_{m} : \Category{\free{BK}}_{m}$ $\to$ $\Category{\free{BJ}}_{m}$ are obtained.
\end{rem}

\subsection{Topological ${A}_{\infty}$-operad category with degeneracies for an object}

We introduce a small topological category $\Category{\unital{BK}}$:
let us recall that $D(m,m)=\{(\emptyset)\}$, and $D(m,\ell)=\{\}$ if $m<\ell$.

\begin{defn}\label{defn:ext-operadic-object-deg}
Let $\Category{\unital{BK}}$ be the category consisting of $\Object{\Category{\unital{BK}}}$ the set of objects and $\Morphism{\Category{\unital{BK}}}$ the set of morphisms given as follows:
\begin{description}
\item[(objects)]
$\Object{\Category{\unital{BK}}} = \{\,\underline{0},\underline{1},\underline{2},\underline{3},\dots\,\} = \cardinal$ the set of all non-negative integers,
\item[(morphisms)]
For any two non-negative integers $m$ and $n$, \par
$\Morphism{\Category{\unital{BK}}}(\underline{m},\underline{n})$ $=$ $\underset{1 \leq \ell \leq m}{\coprod} \Morphism{\Category{\free{BK}}}(\underline{\ell},\underline{n}) \times D(m,\ell)$
\\
which is the set of all formal compositions $(\sigma_{0},\dots,\sigma_{\ell+1}){\comp}(\{i_{1},\dots,i_{m{-}\ell}\})$ made of $(\{i_{1},\dots,i_{m{-}\ell}\}) \in D(m,\ell)$ and $(\sigma_{0},\dots,\sigma_{\ell+1}) \in \ass(a_{0}) \times \cdots \times \ass(a_{\ell+1}) \subset \Morphism{\Category{\free{BK}}}(\underline{\ell},\underline{n})$ 
for some $\ell$, with the following composition formulas.
\begin{enumerate}
\item
For any $(\{i\}) : \underline{m} \to \underline{m{-}1}$ and $(\{j_{1},\dots,j_{n-m}\}) : \underline{n} \to \underline{m}$, the composition $(\{k_{1},\dots,k_{n-m+1}\})=(\{i\}){\comp}(\{j_{1},\dots,j_{n-m}\})$ is given by 
\par\vskip1ex\noindent\hfil
$\{k_{1},\dots,k_{n-m+1}\}=\{j_{1},\dots,j_{n-m}\} \cup \{i{+}b\}$
\hfil\par\vskip1ex\noindent
where $b$ is determined by $j_{b}{-}b \leq i{-}1 < j_{b+1}{-}(b{+}1)$ assuming $j_{0}=0$, since $j_{k}-k$ is an increasing sequence.
For example, we have $(\{i_{1},\dots,i_{m-\ell}\})=(\{i_{1}\}){\comp}\cdots{\comp}(\{i_{m-\ell}\})$ if $1 \le i_{1} < \cdots < i_{m-\ell} \le m$.
\item
For any $(\{i\}) : \underline{n} \to \underline{n{-}1}$, $1 \!\leq\! i \!\leq\! n$ and $(\tau_{0},\dots,\tau_{m+1}) : \underline{m} \to \underline{n}$ with $\tau_{j} \in \ass(r_{j})$, $(r_{0},\dots,r_{m+1}) \in B(m{+}2,n{+}2)$, the composition $(\{i\}){\circ}(\tau_{0},\dots,\tau_{m+1})$ is given by
\par\vskip1ex\noindent\hfil
$\begin{array}{l}
(\{i\}){\circ}(\tau_{0},\dots,\tau_{m+1})
= \begin{cases}\,
(\tau_{0},\dots,d^{K}_{i'}(\tau_{j}),\dots,\tau_{m+1}), &r_{j}\!>\!1,
\\[.5ex]\,
(\tau_{0},\dots,\tau_{j-1},\tau_{j+1},\dots,\tau_{m+1}){\circ}(\{j\}), &r_{j}\!=\!1,
\end{cases}
\end{array}$
\hfil\par\vskip1ex\noindent
where $i'$ and $j$ are determined by $i{+}1=r_{0}{+}\cdots{+}r_{j-1}{+}i'$, $1 \leq i'\leq r_{j}$.
\end{enumerate}
\end{description}
\end{defn}

Truncating $\Category{\unital{BK}}$, we obtain a series of topological small categories as follows.
\begin{defn}
For each $m \geq 0$, we define a full-subcategory $\Category{\unital{BK}}_{m}$ of $\Category{\unital{BK}}$, whose object set is $\{\underline{0},\underline{1},\dots,\underline{m}\} \homeo \{0,1,\dots,m\} \subset \cardinal$.
\end{defn}

\subsection{Topological ${A}_{\infty}$-operad category with degeneracies for a morphism}

We define a topological small category $\Category{\unital{BJ}}$.
\begin{defn}\label{defn:ext-operadic-morphism-deg}
Let $\Category{\unital{BJ}}$ be the category consisting of $\Object{\Category{\unital{BJ}}}$ the set of objects and $\Morphism{\Category{\unital{BJ}}}$ the set of morphisms given as follows:
\begin{description}
\item[(objects)]
$\Object{\Category{\unital{BJ}}} = \{\underline{0},\underline{0}',\underline{1},\underline{1}',\underline{2},\underline{2}',\dots\} \homeo \cardinal\times{\cyclic_2}$,
\item[(morphisms)]
For any two non-negative integers $m$ and $n$,
\par\vskip1ex
$\Morphism{\Category{\unital{BJ}}}(\underline{m},\underline{n}) = 
\Morphism{\Category{\unital{BK}}}(\underline{m},\underline{n})$,
\par\vskip1ex
$\Morphism{\Category{\unital{BJ}}}(\underline{m}',\underline{n}') = 
\Morphism{\Category{\unital{BK}}}(\underline{m}',\underline{n}')$,
\par\vskip1ex
$\Morphism{\Category{\unital{BJ}}}(\underline{m}',\underline{n}) = \emptyset$ and 
\par\vskip1ex
$\Morphism{\Category{\unital{BJ}}}(\underline{m},\underline{n}') = 
\underset{1 \leq \ell \leq m}{\textstyle\coprod} \Morphism{\Category{\free{BJ}}}(\underline{\ell},\underline{n}') \times D(m,\ell)$
\par\vskip1ex\noindent
which is the set of all formal compositions $(\sigma_{0},\dots,\sigma_{\ell+1}){\comp}(\{i_{1},\dots,i_{m{-}\ell}\})$ made of $(\{i_{1},\dots,i_{m{-}\ell}\}) \in D(m,\ell)$ and $(\sigma_{0},\dots,\sigma_{\ell+1}) \in \mlt(a_{0}) \times \cdots \times \mlt(a_{\ell+1}) \subset \Morphism{\Category{\free{BJ}}}(\underline{\ell},\underline{n}')$ for some $\ell$, with the following composition formulas.
\begin{enumerate}
\item
For any $(\{i\}) : \underline{n}' \to \underline{n{-}1}'$, $1 \!\le\! i \!\le\! n$ and $(\rho_{0},\dots,\rho_{m+1})$ $:$ $\underline{m} \to \underline{n}'$ with $\rho_{j} \in \mlt(a_{j})$, $(a_{0},\dots,a_{m+1}) \in B(m{+}2,n{+}2)$, the composition $(\{i\}){\circ}(\rho_{0},\dots,\rho_{m+1})$ is given by
\\[.5ex]\text{\,\!}\quad
$(\{i\}){\circ}(\rho_{0},\dots,\rho_{m+1})$ 
= $\begin{cases}\, (\rho_{0},\dots,d^{J}_{i'}(\rho_{j}),\dots,\rho_{m+1}), & a_{j}\!>\!1,
\\[.5ex]\,
(\rho_{0},\dots,\rho_{j-1},\rho_{j+1},\dots,\rho_{m+1}){\circ}(\{j\}), & a_{j}\!=\!1,
\end{cases}$
\\[1ex]
where $i'$ and $j$ are determined by $i{+}1=a_{0}{+}\cdots{+}a_{j-1}{+}i'$, $1 \!\leq\! i' \!\leq\! a_{j}$.
\end{enumerate}
\end{description}
\end{defn}

Truncating $\Category{\unital{BJ}}$, we obtain a series of topological small categories as follows.
\begin{defn}
For each $m \geq 0$, we define a full-subcategory $\Category{\unital{BJ}}_{m}$ of $\Category{\unital{BJ}}$, whose object set is $\{\underline{0},\underline{0}',\underline{1},\underline{1}',\dots,\underline{m},\underline{m}'\} \homeo \{0,1,\dots,m\} \times \cyclic_2 \subset \cardinal \times \cyclic_2$.
\end{defn}

\section{Bar construction of an internal $A_{\infty}$-category}

To make sure that our construction is {\em natural enough}, we temporarily move into a categorical setting along the line of Aguiar's definition of an internal category \cite{MR2696373}, in this section.  In the next section, we shall return our original interest on $A_{m}$-spaces.

\subsection{Representations of an enriched category}

A category $\Category{D}$ is called a $\Category{C}$-enriched category, if the set of morphisms $\Category{D}(A,B)$ is an object of $\Category{C}$ for any two objects $A, B \in \Object{\Category{D}}$ such that $\Category{D}(-,-)$ gives a functor from $\Category{D}^{op}\times\Category{D}$ to $\Category{C}$.
If a small category $\Category{D}$ is a $\Category{C}$-enriched category, then $\Category{D}$ is called a $\Category{C}$-enriched small category.
Let us paraphrase the word `$\Category{C}$-enriched' by `topological', if $\Category{C}=\Top$.

Because it is technically difficult to treat a functor between enriched categories, we introduce here the notion of a representation of a $\Category{C}$-enriched small category in $\Category{C}$.
For a $\Category{C}$-enriched category $\Category{D}$, a left representation $\Phi$ of $\Category{D}$ in $\Category{C}$ is a pair $(\Object{\Phi},\Morphism{\Phi})$ of correspondences satisfying the following conditions and is denoted by $\Phi : \Category{D} \rightharpoonup \Category{C}$:
\begin{enumerate}
\item
$\Object{\Phi} : \Object{\Category{D}} \to \Object{\Category{C}}$ and 
\item
$\Morphism{\Phi} : \Object{\Category{D}}\otimes\Object{\Category{D}}$ $\to$ $\Morphism{\Category{C}}$ such that 
\begin{enumerate}
\item
for any $\underline{a}, \underline{b}$ in $\Object{\Category{D}}$, 
$\Morphism{\Phi}(\underline{a},\underline{b}) : \Morphism{\Category{D}}(\underline{a},\underline{b})\otimes\Object{\Phi}(\underline{a})$ $\to$ $\Object{\Phi}(\underline{b})$, 
\item
for any $\underline{a}$ in $\Object{\Category{D}}$ and $x \in \Object{\Phi}(\underline{a})$, 
$\Morphism{\Phi}(\underline{a},\underline{a})(1_{a},x)$ $=$ $x$ and 
\item
for any $\underline{a}$, $\underline{b}$, $\underline{c}$ in $\Object{\Category{D}}$ and any $x \in \Object{\Phi}(\underline{a})$, \par\quad
$\Morphism{\Phi}(\underline{b},\underline{c})(\beta,\Morphism{\Phi}(\underline{a},\underline{b})(\alpha,x)) = \Morphism{\Phi}(\underline{a},\underline{c})(\beta{\circ}\alpha,x)$.
\end{enumerate}
\end{enumerate}
A right representation $\Psi$ of $\Category{D}$ in $\Category{C}$ is a pair $(\Object{\Psi},\Morphism{\Psi})$ of correspondences satisfying the following conditions and is denoted by $\Psi : \Category{D} \rightharpoondown \Category{C}$:
\begin{enumerate}
\item
$\Object{\Psi} : \Object{\Category{D}}$ $\to$ $\Object{\Category{C}}$ and 
\item
$\Morphism{\Psi} : \Object{\Category{D}}\otimes\Object{\Category{D}}$ $\to$ $\Morphism{\Category{C}}$ such that 
\begin{enumerate}
\item
for any $\underline{a}, \underline{b}$ in $\Object{\Category{D}}$, 
$\Morphism{\Psi}(\underline{a},\underline{b}) : \Object{\Psi}(\underline{b})\otimes\Morphism{\Category{D}}(\underline{a},\underline{b})$ $\to$ $\Object{\Psi}(\underline{a})$,
\item
for any $\underline{a}$ in $\Object{\Category{D}}$ and $x \in \Object{\Psi}(\underline{a})$, 
$\Morphism{\Psi}(\underline{a},\underline{a})(x,1_{a})$ $=$ $x$ and 
\item
for any $\underline{a}$, $\underline{b}$, $\underline{c}$ in $\Object{\Category{D}}$ and any $x \in \Object{\Psi}(\underline{a})$, 
\par\quad
$\Morphism{\Psi}(\underline{b},\underline{c})(\Morphism{\Psi}(\underline{a},\underline{b})(x,\alpha),\beta) = \Morphism{\Psi}(\underline{a},\underline{c})(x,\alpha{\circ}\beta)$.
\end{enumerate}
\end{enumerate}
\begin{rem}
If the regular bimonoidal category $\Category{C}$ is self-enriched to be a closed monoidal category by tensor product, then a left (or right) representation is nothing but a covariant (or contravariant) functor.
\end{rem}
\begin{expls}
We have the following two canonical left representations of $\free{A}_{\infty}$-operad categories in $\Top$.
\begin{enumerate}
\item
Let $\overline{\free{K}}$ be the following representation of $\Category{\free{BK}}$:
\begin{align*}&
\overline{\free{K}}(\underline{n}) = \ass(n{+}2),\quad
\overline{\free{K}}(\underline{m},\underline{n})(\tau_{0},\dots,\tau_{m+1};\sigma) = \partial_{1}(\tau_{0}){\comp}\cdots{\comp}\partial_{m+2}(\tau_{m+1})(\sigma).
\end{align*}
\item
Let $\overline{\free{J}}_{0}$ be the following representation of $\Category{\free{BK}}$:
\begin{align*}&
\overline{\free{J}}_{0}(\underline{n}) = \mlt_{0}(n{+}2),\quad
\overline{\free{J}}_{0}(\underline{m},\underline{n})(\tau_{0},\dots,\tau_{m+1};\sigma) = \delta_{1}(\tau_{0}){\comp}\cdots{\comp}\delta_{m+2}(\tau_{m+1})(\sigma).
\end{align*}
\item
Let $\overline{\free{J}}$ be the following representation of $\Category{\free{BJ}}$:
\begin{align*}&
\begin{array}{l}
\overline{\free{J}}(\underline{n}) = \ass(n{+}2),\quad
\overline{\free{J}}(\underline{n}') = \mlt_{0}(n{+}2),\quad
\end{array}
\\&\qquad
\begin{cases}\,
\overline{\free{J}}(\underline{m},\underline{n})(\tau_{0},\dots,\tau_{m+1};\sigma) = \partial_{1}(\tau_{0}){\comp}\cdots{\comp}\partial_{m+2}(\tau_{m+1})(\sigma),
\\[1ex]\,
\overline{\free{J}}(\underline{m},\underline{n}')(\rho_{0},\dots,\rho_{m+1};\tau) = \delta(\tau;\rho_{0},\dots,\rho_{m+1}),
\\[1ex]\,
\overline{\free{J}}(\underline{m}',\underline{n}')(\tau_{0},\dots,\tau_{m+1};\rho) = \delta_{1}(\tau_{0}){\comp}\cdots{\comp}\delta_{m+2}(\tau_{m+1})(\rho).
\end{cases}
\end{align*}
\end{enumerate}
where $\partial_{k}(\tau)(\rho)=\partial_{k}(\rho,\tau)$ and $\delta_{k}(\tau)(\rho)=\delta_{k}(\rho,\tau)$.
\end{expls}
Similarly, we obtain the following examples.
\begin{expls}
The followings are canonical left representations of topological ${A}_{\infty}$-operad categories with degeneracies in $\Top$.
\begin{enumerate}
\item
Let $\overline{\unital{K}}$ be the following representation of $\Category{\unital{BK}}$:
\begin{align*}&
\overline{\unital{K}}(\underline{n}) = \ass(n{+}2),\quad
\\&\qquad
\begin{cases}\,
\overline{\unital{K}}(\underline{m},\underline{n})(\tau_{0},\dots,\tau_{m+1};\sigma) = \partial_{1}(\tau_{0}){\comp}\cdots{\comp}\partial_{m+2}(\tau_{m+1})(\sigma),
\\[2ex]\,
\overline{\unital{K}}(\underline{m},\underline{\ell})((\{j_{1},\dots,j_{m-\ell}\});\sigma) = d^{K}_{j_{1}+1}{\comp}\cdots{\comp}d^{K}_{j_{m-\ell}+1}(\sigma),
\end{cases}
\end{align*}
for  $\ell \leq m \leq n$.
\item
Let $\overline{\unital{J}}_{0}$ be the following representation of $\Category{\unital{BK}}$:
\begin{align*}&
\overline{\unital{J}}_{0}(\underline{n}) = \mlt_{0}(n{+}2),\quad
\\&\qquad
\begin{cases}\,
\overline{\unital{J}}_{0}(\underline{m},\underline{n})(\tau_{0},\dots,\tau_{m+1};\sigma) = \delta_{1}(\tau_{0}){\comp}\cdots{\comp}\delta_{m+2}(\tau_{m+1})(\sigma),
\\[2ex]\,
\overline{\unital{J}}_{0}(\underline{m},\underline{\ell})((\{j_{1},\dots,j_{m-\ell}\});\sigma) = d^{J}_{j_{1}+1}{\comp}\cdots{\comp}d^{J}_{j_{m-\ell}+1}(\sigma),
\end{cases}
\end{align*}
for  $\ell \leq m \leq n$.
\item
Let $\overline{\unital{J}}$ be the following representation of $\Category{\unital{BJ}}$:
\begin{align*}&
\begin{array}{l}
\overline{\unital{J}}(\underline{n}) = \ass(n{+}2),\quad
\overline{\unital{J}}(\underline{n}') = \mlt_{0}(n{+}2),\quad
\end{array}
\\&\qquad
\begin{cases}\,
\overline{\unital{J}}(\underline{m},\underline{n})(\tau_{0},\dots,\tau_{m+1};\sigma) = \partial_{1}(\tau_{0}){\comp}\cdots{\comp}\partial_{m+2}(\tau_{m+1})(\sigma),
\\[1ex]\,
\overline{\unital{J}}(\underline{m},\underline{n}')(\rho_{0},\dots,\rho_{m+1};\tau) = \delta(\tau;\rho_{0},\dots,\rho_{m+1}),
\\[1ex]\,
\overline{\unital{J}}(\underline{m}',\underline{n}')(\tau_{0},\dots,\tau_{m+1};\rho) = \delta_{1}(\tau_{0}){\comp}\cdots{\comp}\delta_{m+2}(\tau_{m+1})(\rho),
\\[1ex]\,
\overline{\unital{J}}(\underline{m},\underline{\ell})((\{j_{1},\dots,j_{m-\ell}\});\sigma) = d^{K}_{j_{1}+1}{\comp}\cdots{\comp}d^{K}_{j_{m-\ell}+1}(\sigma),
\\[1ex]\,
\overline{\unital{J}}(\underline{m}',\underline{\ell}')((\{j_{1},\dots,j_{m-\ell}\});\rho) = d^{J}_{j_{1}+1}{\comp}\cdots{\comp}d^{J}_{j_{m-\ell}+1}(\rho),
\end{cases}
\end{align*}
for  $\ell \leq m \leq n$.
\end{enumerate}
\end{expls}

\subsection{Hom and tensor of representations}

In this section, we introduce two natural constructions of an object from two representations.

Firstly we introduce a natural ``hom'' set of two left representations.
Let $\Phi, \Phi' : \Category{D} \rightharpoonup \Category{C}$ be two left representations of a $\Category{C}$-enriched small category $\Category{D}$ in a regular monoidal category $\Category{C}$.
\begin{defn}
We define $\Hom_{\Category{D}}(\Phi,\Phi')$ the set of natural homomorphisms between two left representations $\Phi$ and $\Phi'$ over $\Category{D}$, which consists of a family of maps $\{\,f_{X} \,\midvert\, X \in \Object{\Category{D}}\,\}$ in the category $\Category{C}$, such that the following diagram is commutative:
\begin{equation*}
\begin{diagram}
\node{\Morphism{\Category{D}}(A,B){\otimes}\Object{\Phi}(B)}
\arrow{e,t}{\Morphism{\Phi}}
\arrow{s,l}{1{\otimes}f_{B}}
\node{\Object{\Phi}(A)}
\arrow{s,r}{f_{A}}
\\
\node{\Morphism{\Category{D}}(A,B){\otimes}\Object{\Phi'}(B)}
\arrow{e,t}{\Morphism{\Phi'}}
\node{\Object{\Phi'}(A),}
\end{diagram}
\end{equation*}
where $A$ and $B$ run over the set of objects $\mathcal{O}(\Category{D})$ of the category $\Category{D}$.
\end{defn}
A natural homomorphism between two right representations $\Psi$ and $\Psi'$ is defined similarly, and we also obtain $\Hom_{\Category{D}}(\Psi,\Psi')$ the set of natural homomorphisms between $\Psi$ and $\Psi'$.

\smallskip

Secondly we introduce a tensor product of right and left representations.
Since the ordinary bar construction of a group can be regarded by co-equalizer, we also use a co-equalizer here.
Let $\Phi : \Category{D} \rightharpoonup \Category{C}$ and $\Psi : \Category{D} \rightharpoondown \Category{C}$ be right and left representations of a $\Category{C}$-enriched small category $\Category{D}$ in a regular monoidal category $\Category{C}$.
\begin{defn}\label{defn:co-equalizer}
Let the tensor product $\Psi{\otimes}_{\Category{D}}\Phi$ of two representations $\Psi$ and $\Phi$ over $\Category{D}$ be the co-equalizer of the following morphisms in $\Category{C}$:
\begin{description}
\item[\quad1]
$\displaystyle\bigoplus_{A,B}\Object{\Psi}(A){\otimes}\Morphism{\Category{D}}(A,B){\otimes}\Object{\Phi}(B)$ 
$\displaystyle\xrightarrow{\overset{A,B}{\textstyle\bigoplus}\,1{\otimes}\Morphism{\Phi}(A,B)} \bigoplus_{A}\Object{\Psi}(A){\otimes}\Object{\Phi}(A)$,\smallskip
\item[\quad2]
$\displaystyle\bigoplus_{A,B}\Object{\Psi}(A){\otimes}\Morphism{\Category{D}}(A,B){\otimes}\Object{\Phi}(B)$ 
$\displaystyle\xrightarrow{\overset{A,B}{\textstyle\bigoplus}\Morphism{\Psi}(A,B){\otimes}1} \bigoplus_{B}\Object{\Psi}(B){\otimes}\Object{\Phi}(B)$,
\end{description}
where $A$ and $B$ run over the set of objects $\mathcal{O}(\Category{D})$ of the category $\Category{D}$.
\end{defn}

\subsection{Two-sided bar construction of an internal ${A}_{\infty}$-category with {\em strict-unit}}

To construct an $A_{m}$-structure for an $A_{m}$-map $f : X \to X'$ regarding {\em strict-units} between $A_{m}$-spaces $X$ and $X'$ with {\em strict-units}, we must use homeomorphic but different constructions of $A_{m}$-structures for $X$, say $\unital{B}'(Y,X,Z)$ and $\unital{B}(Y,X,Z)$. 
More seriously, $\unital{B}'(Y,X,Z)$ is used to construct an $A_{m}$-form of an $A_{m}$-map whose $A_{m}$-form has {\em strict-unit} as in \cite{Iwase:1983} or \cite{MR1000378}.
So, in this section, we define both of them.

Let $X=(X,\{a(n)\})$ be an internal ${A}_{\infty}$-category with {\em strict-unit}, with an internal right ${A}_{\infty}$-action $(Y,X;\{a'(n)\})$ with {\em strict-unit} of $X$ on $Y$ and an internal left ${A}_{\infty}$-action $(X,Z;\{a''(n)\})$ with {\em strict-unit} of $X$ on $Z$ in $\Top$.
\begin{defn}\label{defn:strict-right-representation}
The internal right and left ${A}_{\infty}$-actions of $X$ on $Y$ and $Z$ induces a right representation $\underline{\unital{B}(Y,X,Z)}$ of the topological ${A}_{\infty}$-operad category $\Category{\unital{BK}}$ defined for $n \geq 0$, $\tau_{k} \in \ass(t_{k})$, $0 \leq k \leq m{+}1$ and $(y,x_{1},\dots,x_{m},z) \in Y{\times_{O}}(\prod_{O}^{m}X){\times_{O}}Z$ as follows:
\begin{align*}&
\underline{\unital{B}(Y,X,Z)}(\underline{n}) = Y{\times_{O}}(\prod_{O}^{n}X){\times_{O}}Z,
\\[-2mm]&
\underline{\unital{B}(Y,X,Z)}(\underline{m},\underline{n})(\tau_{0},\dots,\tau_{m+1};y,x_{1},\dots,x_{n},z) = (a'(\tau_{0}),a(\tau_{1}),\dots,a(\tau_{m}),a''(\tau_{m+1})),
\\[1ex]&
\underline{\unital{B}(Y,X,Z)}(\underline{m},\underline{\ell})((\{i_{1},\dots,i_{m-n}\});y,x_{1},\dots,x_{n},z) 
= (y,x'_{1},\dots,x'_{m},z),
\end{align*}
where $x'_{j} = 
\begin{cases}\,
x_{j-b},\quad b = \#\{\,i_{k} \!<\! j \,\},&j\not\in\{i_{1},\dots,i_{m-n}\}
\\\,
{\ast},&j\in\{i_{1},\dots,i_{m-n}\}
\end{cases}, \ 1 \!\le\! j \!\le\! m$ and $a'(\tau_{0})$, $a(\tau_{k})$ ($1 \leq k \leq m$), $a''(\tau_{m+1})$ are given by the following formulae.
\begin{align*}&
a'(\tau_{0})=a'(t_{0})(\tau_{0};y,x_{1},\dots,x_{t_{0}-1}),
\\&
a(\tau_{k})=a(t_{k})(\tau_{k};x_{t_{0}+\cdots+t_{k-1}},\dots,x_{t_{0}+\cdots+t_{k}-1}),\quad 1 \leq k \leq m,
\\&
a''(\tau_{m+1})=a''(t_{m+1})(\tau_{m+1};x_{t_{0}+\cdots+t_{m}},\dots,x_{t_{0}+\cdots+t_{m+1}-2},z).
\end{align*}
\end{defn}

\begin{defn}\label{defn:strict-projective-spaces}
$\unital{B}(Y,X,Z)$ and $\unital{B}'(Y,X,Z)$ the two-sided Bar constructions of $(Y,X,Z)$ in $\Top$ of $(Y,X,Z)$ is defined as follows:
\begin{align*}&
\unital{B}(Y,X,Z) = \underline{\unital{B}(Y,X,Z)} \otimes \overline{\unital{K}},
\quad%
\unital{B}'(Y,X,Z) = \underline{\unital{B}(Y,X,Z)} \otimes \overline{\unital{J}_{0}},
\end{align*}
where $\overline{\unital{K}}$ denotes the canonical left representation of $\Category{\unital{BK}}$.
\end{defn}

Then we can show that the two-sided bar constructions $\unital{B}(Y,X,Z)$ and $\unital{B}'(Y,X,Z)$ are always well-defined:

\begin{thm}
$\unital{B}(Y,X,Z)$ and $\unital{B}'(Y,X,Z)$ are well-defined and homeomorphic to each other in $\Top$.
\end{thm}
\begin{proof}
Let $\underline{\unital{B}_{n}(Y,X,Z)}$ and $\overline{\unital{K}_{n}}$ be the restriction of $\underline{\unital{B}(Y,X,Z)}$ and $\overline{\unital{K}}$ to $\Category{\unital{BK}}_{n}$, respectively.
Then $\unital{B}_{n} = \underline{\unital{B}_{n}(Y,X,Z)} \otimes \overline{\unital{K}_{n}}$, $n \geq 0$ gives a filtration of $\unital{B}(Y,X,Z)$.
By the definition of a two-sided Bar construction, $\unital{B}_{n}$ can be described as the following form:
\begin{align*}&
\unital{B}_{0} = Y \times_{O} Z
\\[-2mm]&
\unital{B}_{n} = \unital{B}_{n-1} \cup \ass(n{+}2) \times (Y \times_{O} \prod_{O}^{n}X \times_{O} Z),\quad n \geq 1,
\end{align*}
where $\ass(n{+}2) \times (Y {\times_{O}} \prod_{O}^{n}X {\times_{O}} Z)$ is attached to $B_{n-1}$ by a co-equalizer of two maps given as follows, by Definition \ref{defn:co-equalizer}.
\begin{description}
\item[\quad1]
$\displaystyle \ass(r) \times \Morphism{\Category{\unital{BK}}_{n}}(\underline{r},\underline{n}) \times (Y \times_{O} \prod_{O}^{n}X \times_{O} Z)$ 
$\displaystyle \xrightarrow{\id \times \underline{\unital{B}(Y,X,Z)}(\underline{r},\underline{n})} \ass(r) \times (Y \times_{O} (\prod_{O}^{r-2}X) \times_{O} Z)$,\smallskip
\item[\quad2]
$\displaystyle \ass(r) \times \Morphism{\Category{\unital{BK}}_{n}}(\underline{r},\underline{n}) \times (Y {\times_{O}} \prod_{O}^{n}X {\times_{O}} Z)$ 
$\displaystyle\xrightarrow{\overline{\unital{K}}(\underline{r},\underline{n}) \otimes \id} \ass(n) \times (Y \times_{O} \prod_{O}^{n}X \times_{O} Z)$,
\end{description}
Since $\Morphism{\Category{\unital{BK}}_{n}}$ is generated by $(\{i\}) \in \Morphism{\Category{\unital{BK}}_{n}}(\underline{m},\underline{m{-}1})$ for $1 \le i \le m \le n$ and $\ass(s)_{k} = \ass(1)^{k-1} \times \ass(s) \times \ass(1)^{r-k} \subset \Morphism{\Category{\unital{BK}}_{n}}(\underline{r},\underline{m})$ for $1 \le k \le r$, $2 \le r \le m \le n$ and $r\!+\!s=m\!+\!3$, we may only consider $(\{i\}) \in \Morphism{\Category{\unital{BK}}_{n}}(\underline{n},\underline{n{-}1})$ and $\ass(s)_{k} \subset \Morphism{\Category{\unital{BK}}_{n}}(\underline{r},\underline{n})$ to construct $\unital{B}_{n}$.
Thus, $\unital{B}_{n}$ is a pushout of the following maps:
\begin{align*}&
\coprod_{\substack{1 \leq k \leq r, \ 2 \leq r \leq n+1\\r+s=n+3}} \!\!\!\! \ass(r) \times \ass(s)_{k} \times (Y \times_{O} \prod_{O}^{n}X \times_{O} Z)
\\[-5ex]&
\qquad\hskip17em\qquad \underset{\amalg\incl}{\xrightarrow{\theta_{k} \times \id}} \ \ass(n{+}2) \times (Y \times_{O} \prod_{O}^{n}X \times_{O} Z),
\\[-4ex]&
\quad \amalg \ \coprod_{i=1}^{n}\ass(n{+}2) \times (Y {\times_{O}} (\prod_{O}^{i-1}X {\times_{O}} O {\times_{O}} \prod_{O}^{n-i}X) {\times_{O}} Z)
\\[0ex]&
\qquad\hskip5em\qquad\downarrow \id \times a_{k}(s) \ \amalg \ d^{K}_{i} \times \id
\\[-0ex]&
\coprod_{\substack{2 \leq r \leq n+1}} \!\! \ass(r) \times  (Y \times_{O} \prod_{O}^{r-2}X \times_{O} Z)
\\[-2.5ex]&
\qquad\hskip17em\quad \longrightarrow \ \free{B}_{r-2} \cup \free{B}_{n-1} = \free{B}_{n-1},
\\[-2.5ex]&
\quad \amalg \ \coprod_{i=1}^{n}\ass(n{+}1) \times (Y {\times_{O}} (\prod_{O}^{i-1}X {\times_{O}} \prod_{O}^{n-i}X) {\times_{O}} Z)
\end{align*}
where $a_{k}(s) : \ass(s) \times Y \times_{O} \prod_{O}^{n}X \times_{O} Z \to Y \times_{O} \prod_{O}^{r-2}X \times_{O} Z$ is given by
$$
a_{k}(s)(\sigma;y,x_{1},\dots,x_{s};z)=
\begin{cases}\,
(a'(s)(\sigma;y,x_{1},\dots,x_{s-1}),\dots,x_{n};z),&k \!=\! 1,
\\\,
(y;x_{1},\dots,a(s)(\sigma;x_{k-1},\dots,x_{k+s-2}),\dots,x_{n};z),&1 \!<\! k \!<\! r,
\\\,
(y;x_{1},\dots,a''(s)(\sigma;x_{r-1},\dots,x_{r+s-3},z)),&k \!=\! r.
\end{cases}
$$

Since a wedge-sum is compatible with a colimit, $\unital{B}(Y,X,Z)$ is also a colimit of $\unital{B}_{n}$'s, and hence is well-defined.
By replacing $\ass(n{+}2)$ by $\mlt_{0}(n{+}2)$ and $\ass[k](r,s)$ by $\mlt(r,s)_{0}$, we obtain that $\unital{B}'(Y,X,Z)$ is also well-defined.
Then the homeomorphism $\omega_{n} : \ass(n) \to \mlt_{0}(n)$ introduced in Definition \ref{defn:omega} of \S\ref{sect:homeo} gives a homeomorphism between $\unital{B}(Y,X,Z)$ and $\unital{B}'(Y,X,Z)$.
\end{proof}

Let $(f,\{h(n)\}) : (X,\{a(n)\}) \to (X',\{b(n)\})$ be an internal $\unital{A}_{\infty}$-functor regarding {\em strict-units} between internal $\unital{A}_{\infty}$-categories with {\em strict-units}.
Let $(g,f;\{h'(n)\}) : (Y,X;\{a'(n)\})$ $\to$ $(Y',X';\{b'(n)\})$ be an internal $\unital{A}_{\infty}$-equivariant functor regarding {\em strict-units} between internal right $\unital{A}_{\infty}$-actions with {\em strict-units} and let $(f,\ell;\{h''(n)\}) : (X,Z;\{a''(n)\})$ $\to$ $(X',Z';\{b''(n)\})$ be an internal $\unital{A}_{m}$-equivariant functor regarding {\em strict-units} between internal left $\unital{A}_{m}$-actions with {\em strict-units}.
Then $(g,f,\ell)$ induces a map $\unital{B}(g,f,\ell) : \unital{B}(Y,X,Z) \homeo \unital{B}'(Y,X,Z) \to \unital{B}(Y',X',Z')$ by
\begin{align*}&
\unital{B}(g,f,\ell)([\sigma;y,x_{1},\dots,x_{n},z])=[\tau;h'(\rho_{0}),h(\rho_{1}){\cdots},h(\rho_{t}),h''(\rho_{t+1})],
\end{align*}
where $h'(\rho_{0})$, $h(\rho_{k})$, \,$1 \!\le\! k \!\le\! t$ and $h''(\rho_{t+1})$ are the same as before.

\subsection{Two-sided bar construction of an internal ${A}_{\infty}$-category with {\em hopf-unit}}

To construct an $A_{m}$-structure for an $A_{m}$-map $f : X' \to X$ regarding {\em h-units} between $A_{m}$-spaces $X'$ and $X$ with {\em h-units}, we must use homeomorphic but different constructions of $A_{m}$-structures for $X'$ and $X$, say $\free{B}'(Y,X,Z)$ and $\free{B}(Y,X,Z)$, as for the case in which we are working with {\em strict-units}. 
So, in this section, we define both of them.

Let $X=(X,\{a(n)\})$ be an internal $\free{A}_{\infty}$-category with {\em hopf-unit} in $\Top$ with an internal right $\free{A}_{\infty}$-action $(Y,X;\{a'(n)\})$ with {\em hopf-unit} of $X$ on $Y$ and a left $\free{A}_{\infty}$-action $(X,Z;\{a''(n)\})$ with {\em hopf-unit} of $X$ on $Z$ in $\Top$.
\begin{defn}\label{defn:right-representation}
The internal right and left $\free{A}_{\infty}$-actions of $X$ on $Y$ and $Z$ induces a right representation $\underline{\free{B}(Y,X,Z)}$ of the $\free{A}_{\infty}$-operad category $\Category{\free{BK}}$ defined for $n \geq 0$, $\tau_{k} \in \ass(t_{k})$, $0 \leq k \leq m{+}1$ and $(y,x_{1},\dots,x_{m},z) \in Y{\times_{O}}(\prod_{O}^{m}X){\times_{O}}Z$ as follows:
\begin{align*}&
\underline{\free{B}(Y,X,Z)}(\underline{n}) = Y{\times_{O}}(\prod_{O}^{n}X){\times_{O}}Z,
\\&
\underline{\free{B}(Y,X,Z)}(\underline{m},\underline{n})(\tau_{0},\dots,\tau_{m+1};y,x_{1},\dots,x_{n},z) = (a'(\tau_{0}),a(\tau_{1}),\dots,a(\tau_{m}),a''(\tau_{m+1})),
\end{align*}
where $a'(\tau_{0})$, $a(\tau_{k})$ ($1 \leq k \leq m$), $a''(\tau_{m+1})$ are given by the following formulae.
\begin{align*}&
a'(\tau_{0})=a'(t_{0})(\tau_{0};y,x_{1},\dots,x_{t_{0}-1}),
\\&
a(\tau_{k})=a(t_{k})(\tau_{k};x_{t_{0}+\cdots+t_{k-1}},\dots,x_{t_{0}+\cdots+t_{k}-1}),\quad 1 \leq k \leq m,
\\&
a''(\tau_{m+1})=a''(t_{m+1})(\tau_{m+1};x_{t_{0}+\cdots+t_{m}},\dots,x_{t_{0}+\cdots+t_{m+1}-2},z).
\end{align*}
\end{defn}

\begin{defn}\label{defn:projective-spaces}
$\free{B}(Y,X,Z)$ and $\free{B}'(Y,X,Z)$ the two-sided Bar constructions of $(Y,X,Z)$ in $\Top$ are defined as follows:
\begin{align*}&
\free{B}(Y,X,Z) = \underline{\free{B}(Y,X,Z)} \otimes \overline{\free{K}},
\\&
\free{B}'(Y,X,Z) = \underline{\free{B}(Y,X,Z)} \otimes \overline{\free{J}_{0}},
\end{align*}
where $\overline{\free{K}}$ and $\overline{\free{J}_{0}}$ denote the canonical left representations of $\Category{\free{BK}}$.
\end{defn}

Then we can show that the two-sided bar constructions $\free{B}(Y,X,Z)$ and $\free{B}'(Y,X,Z)$ are always well-defined.

\begin{thm}\label{thm:attaching-cells-on-projective-space}
$\free{B}(Y,X,Z)$ and $\free{B}'(Y,X,Z)$ are well-defined and homeomorphic to each other in $\Top$.
\end{thm}
\begin{proof}
Let $\underline{\free{B}_{n}(Y,X,Z)}$ and $\overline{\free{K}_{n}}$ be the restriction of $\underline{\free{B}(Y,X,Z)}$ and $\overline{\free{K}}$ to $\Category{\free{BK}}_{n}$, respectively.
Then $\free{B}_{n} = \underline{\free{B}_{n}(Y,X,Z)} \otimes \overline{\free{K}_{n}}$, $n \geq 0$ gives a filtration of $\free{B}(Y,X,Z)$.
By the definition of a two-sided Bar construction, $\free{B}_{n}$ can be described as the following form:
\begin{align*}&
\free{B}_{0} = Y{\times_{O}}Z
\\[-3mm]&
\free{B}_{n} = \free{B}_{n-1} \cup \ass(n{+}2)  \times  (Y {\times_{O}} (\prod_{O}^{n}X) {\times_{O}} Z),\quad n \geq 1,
\end{align*}
\vskip-2mm
where $\ass(n{+}2)  \times  (Y {\times_{O}} (\prod_{O}^{n}X) {\times_{O}} Z)$ is attached to $\free{B}_{n-1}$ by a co-equalizer of two maps given as follows, by Definition \ref{defn:co-equalizer}.
\begin{description}
\item[\quad1]
$\displaystyle \ass(r) \times \Morphism{\Category{\free{BK}}_{n}}(\underline{r},\underline{n}) \times (Y \times_{O} \prod_{O}^{n}X \times_{O} Z)$ 
$\displaystyle \xrightarrow{\id \times \underline{\free{B}(Y,X,Z)}(\underline{r},\underline{n})} \ass(r) \times (Y \times_{O} (\prod_{O}^{r-2}X) \times_{O} Z)$,\smallskip
\item[\quad2]
$\displaystyle \ass(r) \times \Morphism{\Category{\free{BK}}_{n}}(\underline{r},\underline{n}) \times (Y {\times_{O}} \prod_{O}^{n}X {\times_{O}} Z)$ 
$\displaystyle\xrightarrow{\overline{\free{K}}(\underline{r},\underline{n}) \otimes \id} \ass(n) \times (Y \times_{O} \prod_{O}^{n}X \times_{O} Z)$,
\end{description}
Since $\Morphism{\Category{\free{BK}}_{n}}$ is generated by $\ass(s)_{k} = \ass(1)^{k-1} \times \ass(s) \times \ass(1)^{r-k} \subset \Morphism{\Category{\free{BK}}_{n}}(\underline{r},\underline{m})$ for $1 \le k \le r$, $2 \le r \le m \le n$ and $r\!+\!s=m\!+\!3$, we may only consider $\ass(s)_{k} \subset \Morphism{\Category{\free{BK}}_{n}}(\underline{r},\underline{n})$ to construct $\free{B}_{n}$.
Thus, $\free{B}_{n}$ is a pushout of the following maps:
\begin{align*}&
\coprod_{\substack{1 \leq k \leq r, \ 2 \leq r \leq n+1\\r+s=n+3}} \!\!\!\! \ass(r) \times \ass(s)_{k} \times (Y \times_{O} \prod_{O}^{n}X \times_{O} Z)
\underset{\amalg\incl}{\xrightarrow{\theta_{k} \times \id}} \ \ass(n{+}2) \times (Y \times_{O} \prod_{O}^{n}X \times_{O} Z),
\\[0ex]&
\qquad\hskip5em\qquad\downarrow \id \times a_{k}(s) 
\\[-0ex]&\qquad\quad
\coprod_{\substack{2 \leq r \leq n+1}} \!\! \ass(r) \times  (Y \times_{O} \prod_{O}^{r-2}X \times_{O} Z)
\ \longrightarrow \ \free{B}_{r-2} \cup \free{B}_{n-1} = \free{B}_{n-1},
\end{align*}
where $a_{k}(s) : \ass(s) \times Y \times_{O} \prod_{O}^{n}X \times_{O} Z \to Y \times_{O} \prod_{O}^{r-2}X \times_{O} Z$ is given by
$$
a_{k}(s)(\sigma;y,x_{1},\dots,x_{s};z)=
\begin{cases}\,
(a'(s)(\sigma;y,x_{1},\dots,x_{s-1}),\dots,x_{n};z),&k \!=\! 1,
\\\,
(y;x_{1},\dots,a(s)(\sigma;x_{k-1},\dots,x_{k+s-2}),\dots,x_{n};z),&1 \!<\! k \!<\! r,
\\\,
(y;x_{1},\dots,a''(s)(\sigma;x_{r-1},\dots,x_{r+s-3},z)),&k \!=\! r.
\end{cases}
$$

Since a wedge-sum is compatible with a colimit, $\free{B}(Y,X,Z)$ is also a colimit of $\free{B}_{n}$'s, and hence is well-defined.
By replacing $\ass(n{+}2)$ by $\mlt_{0}(n{+}2)$ and $\ass[k](r,s)$ by $\mlt(r,s)_{0}$, we obtain that $\free{B}'(Y,X,Z)$ is also well-defined.
Then the homeomorphism $\omega_{n} : \ass(n) \to \mlt_{0}(n)$ introduced in Definition \ref{defn:omega} of \S\ref{sect:homeo} gives a homeomorphism between $\free{B}(Y,X,Z)$ and $\free{B}'(Y,X,Z)$.
\end{proof}

Let $(f,\{h(n)\}) : (X,\{a(n)\})$ $\to$ $(X',\{b(n)\})$ be an internal $\free{A}_{\infty}$-functor regarding {\em hopf-units} between internal $\free{A}_{\infty}$-categories with {\em hopf-units}.
Let $(g,f;\{h'(n)\}) : (Y,X;\{a'(n)\})$ $\to$ $(Y',X';\{b'(n)\})$ be an internal $\free{A}_{\infty}$-equivariant prefunctor regarding {\em hopf-units} between internal right $\free{A}_{\infty}$-actions with {\em hopf-units} and let $(f,\ell;\{h''(n)\}) : (X,Z;\{a''(n)\}) \to (X',Z';\{b''(n)\})$ be an internal $\free{A}_{m}$-equivariant prefunctor regarding {\em hopf-units} between internal left $\free{A}_{m}$-actions with {\em hopf-units}.
Then $(g,f,\ell)$ induces a map $\free{B}(g,f,\ell) : \free{B}(Y,X,Z) \homeo \free{B}'(Y,X,Z) \to \free{B}(Y',X',Z')$ by
\begin{align*}&
\free{B}(g,f,\ell)([\sigma;y,x_{1},\dots,x_{n},z])=[\tau;h'(\rho_{0}),h(\rho_{1}){\cdots},h(\rho_{t}),h''(\rho_{t+1})],
\end{align*}
where $h'(\rho_{0})$, $h(\rho_{k})$, \,$1 \!\le\! k \!\le\! t$, and $h''(\rho_{t+1})$ are given by 
\begin{align*}&
h'(\rho_{0})=h'(\rho_{0};y,x_{1},\dots,x_{r_{0}-1})
\\&
h(\rho_{k})=h(\rho_{k};x_{r_{0}+ \cdots +r_{k-1}},\dots,x_{r_{0}+ \cdots r_{k}-1}), \quad 1 \!\le\! k \!\le\! t,
\\&
h''(\rho_{t+1})=h''(\rho_{t+1};x_{r_{0}+\cdots+r_{t}},\dots,x_{r_{0}+\cdots+r_{t+1}-2},z)
\end{align*}
where $\omega_{n+2}(\sigma)=\delta(t;r_{0},\dots,r_{t+1})(\tau;\rho_{0},\dots,\rho_{t+1})$, \,$(r_{0},\dots,r_{t+1}) \in B(t{+}2,n{+}2)$.

\section{Proof of Theorem \ref{thm:ASIM}}\label{section:units}

We adopt here a completely different approach from the original one to consider about the equivalence of two definitions given in \cite{MR0158400}.
In this section, we give a bar construction for an $A_{\infty}$-space with {\em h-unit}.

\subsection{${A}_{\infty}$-space with {\em h-unit}}

First we define a slightly weaker version of an $\free{A}_{m}$-forms $(m \leq \infty)$ for a topological space.
\begin{defn}\label{defn:form-top}
We call $(X;\{a(n);1 \!\leq\! n \!\leq\! m\})$ ($a(1)=1_{X}$) an `${A}_{m}$-space with {h-unit}', if based maps $\begin{textstyle}a(n) : \ass(n) \times \prod^n{X} \to X\end{textstyle}$ satisfy the following formulas for any $n \leq m$ and $\mathbold{x}=(x_{1},\dots,x_{n}) \in \prod^{n} X$:
\begin{align}
\label{eq:topological-obj2-2}&
a(2)\vert_{\{\ast\} \times ( X \vee X )} \sim \nabla_{X} \ (\text{based homotopic}),
\\ \label{eq:topological-obj3-2}&
a(n)(\partial_{k}(\rho,\sigma);\mathbold{x})= a(r)(\rho;a_{k}(s)(\sigma;\mathbold{x})),
\end{align}
where $a_{k}(s)(\sigma;\mathbold{x})$ is given by 
\begin{align*}&
(x_{1},\dots,x_{k-1},a(s)(\sigma;x_{k},\dots,x_{k+s-1}),x_{k+s},\dots,x_{n}).
\end{align*}
\end{defn}

If a space is an ${A}_{m}$-space with {\em h-unit} for any $m \geq 1$, then it is called an ${A}_{\infty}$-space with {\em h-unit}.
It is easy to see that an ${A}_{2}$-space with {\em h-unit} is just an h-space.
Then an ${A}_{m}$-space with {\em hopf-unit} is an ${A}_{m}$-space with {\em h-unit} and the converse is also true by using homotopy extension property (HEP) of $(X,e)$, the space with non-degenerate base point.

\subsection{${A}_{\infty}$-map regarding {\em h-units}}

First we define a version of an $\free{A}_{m}$-forms $(m \leq \infty)$ for a map between ${A}_{m}$-spaces with {\em h-units}.
\begin{defn}\label{defn:form-map}
We call $(f: X \to Y,\{h(n);1 \!\leq\! n \!\leq\! m\})$ ($h(1)=f$) an ``${A}_{m}$-map regarding h-units'', if $\begin{textstyle}h(n) : \mlt(n) \times \prod^n{X} \to Y\end{textstyle}$ satisfies the following formulas for any $n \leq m$ and $\mathbold{x}=(x_{1},\dots,x_{n}) \in \prod^{n} X$:
\begin{align}
\label{eq:tilde-topological-mor2}&
\text{$f(e_{X})$ and $e_{Y}$ lie in the same connected component,}
\\ \label{eq:topological-mor3}&
h(n)(\delta_{k}(\rho,\sigma);\mathbold{x})= h(r)(\rho;a_{k}(s)(\sigma;\mathbold{x})),
\\ \label{eq:topological-mor4}&
h(n)(\delta(\tau;\rho_{1},\dots,\rho_{t});\mathbold{x}) = b(t)(\tau;h(\rho_{1}),\dots,h(\rho_{t})),
\end{align}
where $h(\rho_{k})$, \,$1 \!\le\! k \!\le\! t$, are given by
\begin{align*}&
h(\rho_{k})=h(r_{k})(\rho_{k};x_{r_{1}+\cdots+r_{k-1}+1},\dots,x_{r_{1}+\cdots+r_{k}}),\quad 1 \!\leq\! k \!\le\! t.
\end{align*}
\end{defn}
If a map is an ${A}_{m}$-map regarding {\em h-units} for any $m \geq 1$, then it is called an ${A}_{\infty}$-map regarding {\em h-units}.

\subsection{Projective spaces of an $A_{\infty}$-space with {\em h-unit}}

Let $X$ be an $\free{A}_{m}$-space, i.e, $X$ has an $\free{A}_{m}$-form $\left\{a(r) : \ass(r) \times X^{r} \to X \,\midvert\, 1 \leq r \leq m\right\}$.

\begin{defn}
Let $Y$ and $Z$ be either $X$ or \,${\ast}=\{\ast\}$.
\begin{enumerate}
\item
We define a map $a'(s) : \ass(s) \times Y \times X^{s-1} \to Y$ by 
\begin{align*}&
\begin{cases}\,
a'(s)(\sigma;y,\chi)=a(s)(\sigma;y,\chi),
&Y=X, \ s \leq m,
\\[.5ex]\,
a'(s)(\sigma;\ast,\chi)= {\ast},& Y= {\ast}, \ s \leq m{+}1,
\end{cases}
\end{align*}
where $\sigma \in \ass(s)$, $y \in Y$ and $\chi \in X^{s-1}$.
\item
We define a map $a''(s) : \ass(s) \times X^{s-1} \times Z \to Z$ by
\begin{align*}&
\begin{cases}\,
a''(s)(\sigma;\chi,z)=a(s)(\sigma;\chi,z),
&Z=X, \ s \leq m,
\\[.5ex]\,
a''(s)(\sigma;\chi,{\ast})= {\ast},
& Z= {\ast}, \ s \leq m{+}1,
\end{cases}
\end{align*}
where $\sigma \in \ass(s)$, $z \in Z$ and $\chi \in X^{s-1}$.
\item
For any $\sigma \in \ass(s)$, $r, \,s \geq 2$ and $1 \leq k \leq r$, we define a map $\bar{a}_{k}(\sigma) : Y \times X^{r+s-3} \times Z \to Y \times X^{r-2} \times Z$ by the following formula.
\begin{align*}&
\bar{a}_{k}(\sigma)(y,\chi,z)
\\[.5ex]&\qquad=
\begin{cases}\,
(a'(s)(\sigma;y,x_{2},\dots,x_{s}),\dots,x_{r+s-2},z),&k{=}1,
\\[.25ex]\,
(y,x_{2},\dots,a(s)(\sigma;x_{k},\dots,x_{k+s-1}),\dots,x_{r+s-2},z),&1 {<} k {<} r,
\\[.25ex]\,
(y,x_{2},\dots,a''(s)(\sigma;x_{r},\dots,x_{r+s-2},z)),&k{=}r,
\end{cases}
\end{align*}
where $y \in Y$, $z \in Z$ and $\chi=(x_{2},\dots,x_{r+s-2}) \in X^{r+s-3}$.
\end{enumerate}
\end{defn}

We give here the following definition of $\free{A}_{m}$-structure for $X$:

\begin{defn}
Using the above internal actions of $X$ on $Y$ and $Z$, we define $E^{n}$, $P^{n}$ and $D^{n}$ for $0 \leq n \leq m$ inductively as follows.
\begin{enumerate}
\item
$E^{0}=\emptyset$, %
$P^{0}=\ass(2) \times {\ast} \times {\ast}$ $\homeo$ $\ast$ and 
$D^{0}=\ass(2) \times \{e\} \times {\ast}$ $\homeo$ $\ast$.
\vspace{1ex}\item
$\displaystyle E^{n+1} = \free{B}_{n}(X,X,{\ast}) = \left( E^{n} \ {\textstyle\coprod} \ \left(\ass(n{+}2) \times X \times X^{n} \times {\ast}\right) \right)/\sim$,
\\
where $(\partial_{k}(\sigma)(\rho),x,\chi,{\ast}) \sim (\rho,\bar{a}_{k}(\sigma)(x,\chi,{\ast}))$ for $n \geq 1$.
\vspace{1ex}\item
$\displaystyle P^{n} = \free{B}_{n}({\ast},X,{\ast}) = \left( P^{n-1} \ {\textstyle\coprod} \ \left(\ass(n{+}2) \times {\ast} \times X^{n} \times {\ast}\right) \right)/{\sim}$,
\\
where $(\partial_{k}(\sigma)(\rho),{\ast},\chi,{\ast}) \sim (\rho,\bar{a}_{k}(\sigma)(\ast,\chi,{\ast}))$  for $n \geq 1$,
\vspace{1ex}\item
$\displaystyle D^{n} = \left( E^{n} \ {\textstyle\coprod} \ \left(\ass(n{+}2) \times \{e\} \times X^{n} \times {\ast}\right)\right)/{\sim}$,\\
where $(\partial_{k}(\sigma)(\rho),e,\chi,{\ast}) \sim [\rho,\bar{a}_{k}(\sigma)(e,\chi,{\ast})] \in E^{r-1} \subset E^{n}$ for $n \geq 1$, in which $(e,\chi,{\ast})$ is regarded as an element in $X \times X^{n} \times {\ast}$.
\end{enumerate}
As for (2), by Theorem \ref{thm:attaching-cells-on-projective-space}, $\free{B}_{n}(X,X,{\ast})$ is a pushout of two maps 
\begin{align*}&
\coprod_{\substack{1 \leq k \leq r, \ 2 \leq r \leq n+1\\r+s=n+3}} \!\! \ass(r) \times \ass(s)  \times  (X \times X^{n} \times {\ast}) \ \xrightarrow{\theta_{k} \times \id} \ \ass(n{+}2) \times (X \times X^{n} \times {\ast}),
\\[-2ex]&
\qquad\hskip5em\qquad\downarrow \id \times a_{k}(s)
\\[-0ex]&
\coprod_{\substack{1 \leq k \leq r, \ 2 \leq r \leq n+1\\r+s=n+3}} \!\! \ass(r) \times  (X \times X^{r-2} \times {\ast}) \ \to \ \free{B}_{r-2}(X,X,{\ast}) \subset \free{B}_{n-1}(X,X,{\ast}),
\end{align*}
where $a_{k}(s) = \ad(\bar{a}_{k}) : \ass(s) \times X^{n} \to X^{r}$.
As for (3) and (4), it is obtained by similar arguments and left to the readers.
\end{defn}

\begin{rem}
Since $E^{0}=\emptyset$, we have $E^{1}=\ass(2) \times X \times {\ast} \homeo X$.
\end{rem}
We also have obvious projections $p^{X}_{n} : E^{n} \to P^{n-1}$ and $q^{X}_{n} : (D^{n},E^{n}) \to (P^{n},P^{n-1})$ with $q^{X}_{n}\vert_{E^{n}} = p^{X}_{n-1}$ and $p^{X}_{n}\vert_{D^{n}}=q^{X}_{n}$, $1 \leq n \leq m$ given by
$$
p^{X}_{n}([\tau,x,\chi,\ast])=[\tau,\ast,\chi,\ast]
\quad \text{and} \quad
q^{X}_{n}([\tau,e,\chi,\ast])=[\tau,\ast,\chi,\ast].
$$

From now on, we assume that $(X,a_{2},e)$ is a CW complex loop-like h-space with {\em h-unit}.
We show the following proposition.

\begin{prop}\label{prop:structure-without-unit}
$p^{X}_{n}$ is a quasi-fibration for each $n$.
\end{prop}
\begin{proof}
Let $\alpha(n{+}2)=(0,\fracinline{1}/{2},\dots,\fracinline{1}/{2},\frac{n{+}2}{2})$.
Then $\ass(n{+}2)$ can be described as a union of subsets $A=\ass[2](2,n{+}1)*\{\alpha(n{+}2)\} \subset \ass(n{+}2)$ (join construction) and $B=\overline{\left(\ass(n{+}2) \smallsetminus A\right)} \subset \ass(n{+}2)$, 
where $(A,\Img \partial_{2}(2,n{+}1))$ and $(B,\overline{\left(\partial \ass(n{+}2) \smallsetminus \Img \partial_{2}(2,n{+}1)\right)})$ are DR-pairs in the sense of Whitehead \cite{MR516508}.
Thus $P^{n}$ can be described as a union of images $V \subset P^{n}$ and $W \subset P^{n}$ of $A \times {\ast} \times X^{n} \times {\ast}$ and $B \times {\ast} \times X^{n} \times {\ast}$, respectively, on which the projection $p^{X}_{n}$ is a quasi-fibration.
We can also observe that, on the intersection of $V$ and $W$, $p^{X}_{n}$ is also a quasi-fibration and a fibre on $V \cap W$ mapped to a corresponding fibre on $V$ or $W$ by a right or left translation.
Since $X$ is loop-like, left or right translation is a homotopy equivalence and hence we can deduce that $p^{X}_{n}$ is a quasi-fibration (see also Proof for Theorem 5 of \cite{MR0158400} or the arguments in Chapter 7 of Mimura \cite{Mimura86hopf} using Corollary 1.8 and Lemma 1.13 in Chapter 5 of Mimura-Toda \cite{MR516508} for more detailed arguments).
\end{proof}

The next lemma is the key to obtain the $A_{m}$-structure for $X$.
\begin{lem}\label{lem:structure-null-homotopic}
The inclusion $E^{n} \hookrightarrow D^{n}$ is null-homotopic.
\end{lem}
\begin{proof}
First, we introduce a series of spaces $\widehat{D}^{n}$ inductively on $n$:
\begin{enumerate}
\item
$\widehat{D}^{0} = \ass(2) \times \{e\} \times {\ast} \homeo \ast$.
\vspace{1ex}\item
$\displaystyle\widehat{D}^{n} = \left(\widehat{D}^{n-1} \ {\textstyle\coprod} \ \left(\ass(n{+}2) \times \{e\} \times X^{n} \times {\ast}\right) \right)/{\sim}$, $n \geq 1$,
\end{enumerate}
where $(\partial_{k}(\sigma)(\rho),e,\chi,{\ast})$ $\sim$ $[\rho,\bar{a}_{k}(\sigma)(e,\chi,{\ast})]$ $\in$ $\widehat{D}^{r}$ $\subset$ $\widehat{D}^{n-1}$ for $k>1$, and $(\partial_{1}(\sigma)(\rho),e,\chi,{\ast})$ $\sim$ $[\rho,\bar{a}_{1}(\sigma)(e,\chi,{\ast})]$ $\in$ $E^{r-1}$ $\subset$ $E^{n}$ for $k=1$ and $r \leq n$.
Then the only free-face $\partial_{1}(n{+}1,2) : \ass(n{+}1) \hookrightarrow \ass(n{+}2)$ induces a canonical map $\hat{i}_{n} : E^{n} \rightarrow \widehat{D}^{n}$.
We can further obtain that $\hat{i}_{n}$ is an inclusion by induction on $n$.
Apparently, the inclusion $E^{n} \hookrightarrow D^{n}$ is a composition of $\hat{i}_{n} : E^{n} \hookrightarrow \widehat{D}^{n}$ with the identification map $\widehat{D}^{n} \twoheadrightarrow D^{n}$.

Let us remark here that $\displaystyle\widehat{D}^{1} = \left(D^{0} \ {\textstyle\coprod} \ \ass(3) \times \{e\} \times X \times {\ast}\right)/{\sim} \  \homeo \widehat{C}X$ the unreduced cone of $X$ and $D^{1}=\left(X \ \coprod \ (\ass(3) \times \{e\} \times X \times {\ast}\right)/{\sim} \  \homeo \widehat{C}X/(e \sim \ast) \simeq S^{1}$.
Thus $\widehat{D}^{1}$ is contractible while $D^{1}$ is not.
So, we are left to show that $\widehat{D}^{n}$ is contractible for any $n \geq 2$.

Let $L'(n{+}2) = \overline{\partial \ass(n{+}2) \smallsetminus \ass[1](n{+}1,2)}$ to obtain $(\ass(n{+}2),L'(n{+}2))$ a DR-pair.
Since $e$ is a non-degenerate base point of a CW complex $X$, the pair $(X,\{e\})$ is a NDR-pair in the sense of \cite{MR516508}.
Thus the pair $(K,L)=(\ass(n{+}2),L'(n{+}2)) \times (X,\{e\}) \times X^{n} \times {\ast}$ is a DR-pair.

Since the identification map $(K,L) \to (\widehat{D}^{n},\widehat{D}^{n-1})$ gives a relative homeomorphism, the pair $(\widehat{D}^{n},\widehat{D}^{n-1})$ is also a DR-pair, and hence so is $(\widehat{D}^{n},\widehat{D}^{0})$, for any $n\geq1$.
On the other hand by definition, $\widehat{D}^{0}$ is nothing but a one-point-set which is contractible.
Thus $\widehat{D}^{n}$ is contractible for every $n \geq 0$.
It completes the proof of Proposition \ref{prop:structure-without-unit}.
\end{proof}

This implies that $E^{n}$, $P^{n}$, $D^{n}$ and quasi-fibrations $p^{X}_{n}$ for $0 \!\leq\! n \!\leq\! m$ give an $A_{m}$-structure for $X$ in the sense of \cite{MR0158400}.

Then, by using Theorem \ref{thm:AS}, we obtain the following.
\begin{prop}\label{prop:main}
Let $X$ be a CW complex loop-like h-space {\em h-unit}.
If $X$ has an $A_{m}$-form $\{\,a(n),n\!\geq\!2\,\}$ with $a(2)=\mu$, then there is a homotopy-equivalence $A_{m}$-map inclusion $j : X \hookrightarrow \widetilde{X}$ such that $\widetilde{X}$ has an $A_{m}$-form with {strict-unit}.
\end{prop}
Let $X$ be a connected CW complex h-space with multiplication $\mu : X \times X \to X$, which has an $A_{m}$-form $\{a(n),n \!\leq\! m\}$ with $a(2)=\mu$.
So we have a homotopy-equivalence ${A}_{m}$-map inclusion $j : X \hookrightarrow \widetilde{X}$ in the sense of \cite{Iwase:1983} or \cite{MR1000378}.

Let us assume that there is a deformation of $A_{\ell-1}$-form $\{\hat{h}_{t}(n),n\!<\!\ell\}$ with {\em strict-unit}, $\ell \leq m$, where $\hat{h}_{t}(n)$ is given by maps $\hat{h}_{t}(n) : \ass(n) \times X^{n}$ $\to$ $\widetilde{X}$ obtained by taking adjoint of $h(n) : [0,1] \times \ass(n) \times X^{n}$ $\to$ $\widetilde{X}$ which satisfies the following.
\begin{align}&
\hat{h}_{0}(n) = \hat{a}(n), \ \text{and}
\\&
\hat{h}_{1}(n)(\ass(n) \times X^{n}) \subset X, \quad n < \ell.
\end{align}
Then, by using $\{\hat{h}_{t}(n),n \!<\! \ell\}$, we obtain a map
\begin{align*}&
{h'}(\ell) : \{0\} \times \ass(\ell) \times X^{\ell} 
\cup [0,1] \times \partial{\ass(\ell)} \times X^{\ell} \cup \ass(\ell) \times \fatvee{\ell}X \to \widetilde{X}
\end{align*}
given as follows: for $(\rho,\sigma) \in \ass(r) \times \ass(s)$ ($1 \!\leq\! k \!\leq\! r$, $2 \!\leq\! r, \,s \!<\! \ell$, $r\!+\!s=\ell\!+\!1$), $\tau \in \ass(\ell)$ and $(x_{1},\dots,x_{\ell}) \in X^{\ell}$, we define\vspace{1ex}
\begin{enumerate}
\item
$\displaystyle
{h'}(\ell)(0;\tau;x_{1},\dots,x_{\ell})
=\hat{a}(\ell)(\tau;x_{1},\dots,x_{\ell})
$\vspace{1ex}
\item
\hspace{-1.6mm}$\displaystyle
\begin{array}[t]{l}
{h'}(\ell)(t;\partial_{k}(\sigma)(\rho);x_{1},\dots,x_{\ell})
\\[.5ex]\quad
=\hat{h}_{t}(r)(\rho;x_{1},\dots,x_{k-1},\hat{h}_{t}(s)(\sigma;x_{k},\dots,x_{k+s-1}),x_{k+s},\dots,x_{\ell}),
\end{array}
$\vspace{1ex}
\item
\hspace{-1.6mm}$\displaystyle
\begin{array}[t]{l}
{h'}(\ell)(t;\tau;x_{1},\dots,x_{j-1},e,x_{j+1}{\cdots},x_{\ell})
\\[.5ex]\quad
=\hat{h}_{t}(\ell{-}1)(d^{K}_{k}(\tau);x_{1},\dots,x_{j-1},x_{j+1},\dots,x_{\ell}),
\end{array}
$\vspace{1ex}
\end{enumerate}
which coincide on their intersection with each other, by the relations given in (\ref{eq:topological-obj1}), 
(\ref{eq:topological-obj3}) and (\ref{eq:topological-obj2}).

Since $([0,1],\{0\}) \times (\ass(\ell),\partial{\ass(\ell)}) \times (X^{\ell},\fatvee{\ell}X)$ is a DR-pair, we can extend ${h'}(\ell)$ to a homotopy ${h'}(\ell) : [0,1] \times \ass(\ell) \times X^{\ell}$ $\to$ $\widetilde{X}$ and thus obtain a map $\hat{h'}_{1}(\ell)$ $=$ ${h'}(\ell)\vert_{\{1\} \times \ass(\ell) \times X^{\ell}}$ $:$ $\ass(\ell) \times X^{\ell}$ $\to$ $\widetilde{X}$ which satisfies the following: for $(\rho,\sigma) \in \ass(r) \times \ass(s)$, \,$(k,r,s) \in A(\ell{+}1), \ r, s \!\ge\! 2$, \,$\tau \in \ass(\ell)$ and $(x_{1},\dots,x_{\ell}) \in X^{\ell}$, we have
\begin{enumerate}\item$\displaystyle
\hat{h'}_{1}(\ell) : (\ass(\ell),\partial{\ass(\ell)}) \times (X^{\ell},\fatvee{\ell}X) \to (\widetilde{X},X),
$\vspace{1ex}\item\hspace{-1.6mm}$\displaystyle
\begin{array}[t]{l}
\hat{h'}_{1}(\ell)(\partial(\sigma)(\rho);x_{1},\dots,x_{\ell})
\\[.5ex]\quad
=\hat{h}_{1}(r)(\rho;x_{1},\dots,\hat{h}_{1}(s)(\sigma;x_{k},\dots,x_{k+s-1}),x_{k+s},\dots,x_{\ell}),
\end{array}
$\vspace{1ex}\item\hspace{-1.6mm}$\displaystyle
\begin{array}[t]{l}
\hat{h'}_{1}(\ell)(\tau;x_{1},\dots,x_{j-1},e,x_{j+1}{\cdots},x_{\ell})
\\[.5ex]\quad
=\hat{h}_{1}(\ell{-}1)(d^{K}_{k}(\tau);x_{1},\dots,x_{j-1},x_{j+1},\dots,x_{\ell}).
\end{array}
$\end{enumerate}
Since $(\widetilde{X},X)$ is a DR-pair, we can further compress $\hat{h'}_{1}(n{+}1)$ into $X$, and hence we get a deformation ${h}(n{+}1) : [0,1] \times \ass(n{+}1) \times X^{n+1} \to \widetilde{X}$ which satisfies the following: for $(\rho,\sigma) \in \ass(r) \times \ass(s)$, \,$(k,r,s) \in A(\ell{+}1), \ r, s \!\ge\! 2$, \,$\tau \in \ass(\ell)$ and $(x_{1},\dots,x_{\ell}) \in X^{\ell}$, we have
\begin{enumerate}
\item\hspace{-1.6mm}$\displaystyle
\begin{array}[t]{l}
{h}(\ell)(0;\tau;x_{1},\dots,x_{n+1})
=\hat{a}(n{+}1)(\tau;x_{1},\dots,x_{n+1}) \ \text{and}
\\[.5ex]%
{h}(\ell)(1;\tau;x_{1},\dots,x_{n+1})
\in X
\end{array}
$\vspace{1ex}\item\hspace{-1.6mm}$\displaystyle
\begin{array}[t]{l}
{h}(\ell)(t;\partial(\sigma)(\rho);x_{1},\dots,x_{n+1})
\\[.5ex]\quad
=\hat{h}_{t}(r)(\rho;x_{1},\dots,\hat{h}_{t}(s)(\sigma;x_{k},\dots,x_{k+s-1}),x_{k+s},\dots,x_{n+1}),
\end{array}
$\vspace{1ex}\item\hspace{-1.6mm}$\displaystyle
\begin{array}[t]{l}
\hat{h}(\ell)(t;\tau;x_{1},\dots,x_{j-1},e,x_{j+1}{\cdots},x_{n})
\\[.5ex]\quad
=\hat{h}_{t}(\ell{-}1)(d^{K}_{k}(\tau);x_{1},\dots,x_{j-1},x_{j+1},\dots,x_{\ell}).
\end{array}
$\end{enumerate}

Let $h_{t}(\ell) : \ass(\ell) \times X^{\ell} \to \widetilde{X}$ be the map obtained by taking adjoint of $\hat{h}(\ell) : [0,1] \times \ass(\ell) \times X^{\ell} \to \widetilde{X}$:
$$
h_{t}(\ell)(\tau;x_{1},\dots,x_{\ell}) = \hat{h}(\ell)(t;\tau;x_{1},\dots,x_{\ell}).
$$
Then the $A_{\ell-1}$-form $\{h_{t}(n),n \!<\! \ell\}$ together with $h_{t}(\ell)$ gives a deformation of $A_{\ell}$-form with {\em strict-unit}.
Then by induction, we obtain a deformation of $A_{m}$-form $\{h_{t}(n),n \!\leq\! m\}$ with {\em strict-unit}.
Thus we obtain an $A_{m}$-form with {\em strict-unit} for $X$.

We remark that a similar argument gives an $A_{m}$-form regarding {\em strict units} for $j$.

\appendix

\section{Properties of Degeneracies}\label{sect:degeneracies}

\subsection{Shift $1$ map}\label{sect:shift-one}

Let $\xi(t_{1},\dots,t_{n})=(t'_{1},\dots,t'_{n})$.
By definition, we obtain the following.
\begin{equation}\label{eq:degeneracy1}
\begin{textstyle}\underset{i=1}{\overset{k}{\sum}}\end{textstyle}(t'_{i}\!-\!1) = \Min\left\{\,\begin{textstyle}\underset{i=1}{\overset{k-1}{\sum}}\end{textstyle}(t'_{i}\!-\!1)+(t_{k}\!-\!1),\underset{1 \leq j \leq k}\Max\left\{\begin{textstyle}\underset{i=1}{\overset{j}{\sum}}\end{textstyle}(t_{i}\!-\!1)\right\}-1\,\right\},\quad k \ge 2.
\end{equation}
This equation turns further into the following one.
\begin{equation}\label{eq:degeneracy2}
\begin{aligned}&
\sum^{k}_{i=1}(t_{i}{-}t'_{i}) 
= \Max\left\{\,\begin{textstyle}\underset{i=1}{\overset{k-1}{\sum}}\end{textstyle}(t_{i}{-}t'_{i}),\begin{textstyle}\underset{i=1}{\overset{k}{\sum}}\end{textstyle}(t_{i}\!-\!1)-\underset{1 \leq j \leq k}\Max\left\{\begin{textstyle}\underset{i=1}{\overset{j}{\sum}}\end{textstyle}(t_{i}\!-\!1)\right\}+1\,\right\},\quad k \ge 2.
\end{aligned}
\end{equation}
\par
Firstly, the following proposition implies that $\xi : \real_{+}^{n} \to \real_{+}^{n}$ is well-defined for $n \geq 1$.
\begin{prop}\label{prop:degeneracy-1}
$0 \leq t'_{k} \leq t_{k}$ for $1 \!\le\! k \!\le\! n$.
\end{prop}
\begin{proof}
We show this by induction on $k$.
\par\noindent(Case: $k=1$)
We have $0 \leq \Max\{0,t_{1}{-}1\}$ and $\Max\{0,t_{1}{-}1\} \leq t_{1}$.
\par\noindent(Case: $k \geq 2$)
By induction hypothesis, we may assume that 
$$
0 \leq t'_{i} \leq t_{i},\quad \text{for any $i < k$.}
$$
Since $\Min\{t_{k},x\} \leq t_{k}$ for any $x \in \real$, we have $t'_{k} \leq t_{k}$.
On the other hand, we have
$\underset{1 \leq j \leq k}\Max\left\{\begin{textstyle}\underset{i=1}{\overset{j}{\sum}}\end{textstyle}(t_{i}\!-\!1)\right\} \geq \begin{textstyle}\underset{i=1}{\overset{k-1}{\sum}}\end{textstyle}(t_{i}\!-\!1)$, and hence by (\ref{eq:degeneracy0}) with $k > 1$, we obtain
\begin{align*}
t'_{k} \geq& \Min\left\{\,t_{k},\begin{textstyle}\underset{i=1}{\overset{k-1}{\sum}}\end{textstyle}(t_{i}\!-\!1)-\begin{textstyle}\underset{i=1}{\overset{k-1}{\sum}}\end{textstyle}(t'_{i}\!-\!1)\,\right\} 
= \Min\left\{\,t_{k},\begin{textstyle}\underset{i=1}{\overset{k-1}{\sum}}\end{textstyle}(t_{i}{-}t'_{i})\,\right\}.
\end{align*}
Since $t_{i}-t'_{i} \geq 0$ for each $i < k$, we obtain $t'_{k} \geq \Min\{t_{k},0\} = 0$.
It completes the proof of the proposition.
\end{proof}
\par
Secondly, we show some other properties of $\xi$.
\begin{prop}\label{prop:degeneracy-2}
For any $k \geq 1$, we have $\underset{i=1}{\overset{k}{\sum}}(t_{i}{-}t'_{i}) \leq 1$.
\end{prop}
\begin{proof}
We show this by induction on $k$.
\par\noindent(Case: $k=1$)
By the definition of $t'_{1}$, we have $t'_{1} \geq t_{1}{-}1$ and hence 
$$\underset{i=1}{\overset{1}{\sum}}(t_{i}{-}t'_{i}) = t_{1}{-}t'_{1} \leq t_{1}{-}(t_{1}{-}1) = 1.$$
\par\noindent(Case: $k \geq 2$)
By the induction hypothesis, we may assume 
 $$\begin{textstyle}\underset{i=1}{\overset{k-1}{\sum}}\end{textstyle}(t_{i}{-}t'_{i}) \leq 1.$$
Again by induction hypothesis together with (\ref{eq:degeneracy2}), we proceed as
\begin{align*}&
\sum^{k}_{i=1}(t_{i}{-}t'_{i}) 
\leq \Max\left\{\,1,\begin{textstyle}\underset{i=1}{\overset{k}{\sum}}\end{textstyle}(t_{i}\!-\!1)-\underset{1 \leq j \leq k}\Max\left\{\begin{textstyle}\underset{i=1}{\overset{j}{\sum}}\end{textstyle}(t_{i}\!-\!1)\right\}+1\,\right\}.
\end{align*}
Since $\begin{textstyle}\underset{i=1}{\overset{k}{\sum}}\end{textstyle}(t_{i}\!-\!1) \leq \underset{1 \leq j \leq k}\Max\left\{\begin{textstyle}\underset{i=1}{\overset{j}{\sum}}\end{textstyle}(t_{i}\!-\!1)\right\}$, we have
$$
\begin{textstyle}\underset{i=1}{\overset{k}{\sum}}\end{textstyle}(t_{i}\!-\!1)-\underset{1 \leq j \leq k}\Max\left\{\begin{textstyle}\underset{i=1}{\overset{j}{\sum}}\end{textstyle}(t_{i}\!-\!1)\right\}+1 \leq 1.
$$
Thus we obtain $\underset{i=1}{\overset{k}{\sum}}(t_{i}{-}t'_{i}) \leq \Max\{1,1\} = 1$.
\end{proof}
\begin{prop}\label{prop:degeneracy-4}
If $\underset{i=1}{\overset{k}{\sum}}(t_{i}{-}t'_{i}) = 1$ for some positive integer $k$, then we have $t'_{k'}=t_{k'}$ for any $k'>k$. 
\end{prop}
\begin{proof}
We have 
\begin{align*}&
\begin{textstyle}\underset{i=1}{\overset{k+1}{\sum}}\end{textstyle}(t_{i}\!-\!1)-\underset{1 \leq j \leq k+1}\Max\left\{\begin{textstyle}\underset{i=1}{\overset{j}{\sum}}\end{textstyle}(t_{i}\!-\!1)\right\}+1 
\leq \begin{textstyle}\underset{i=1}{\overset{k+1}{\sum}}\end{textstyle}(t_{i}\!-\!1)-\begin{textstyle}\underset{i=1}{\overset{k+1}{\sum}}\end{textstyle}(t_{i}\!-\!1)+1 = 1.
\end{align*}
By the assumption, we have
\begin{align*}&
(t'_{k+1}{-}t_{k+1})+1 = (t'_{k+1}{-}t_{k+1})+\underset{i=1}{\overset{k}{\sum}}(t_{i}{-}t'_{i})=\underset{i=1}{\overset{k+1}{\sum}}(t_{i}{-}t'_{i}) 
\\&\qquad
= \Max\left\{\,\underset{i=1}{\overset{k}{\sum}}(t_{i}{-}t'_{i}),\begin{textstyle}\underset{i=1}{\overset{k+1}{\sum}}\end{textstyle}(t_{i}\!-\!1)-\underset{1 \leq j \leq k+1}\Max\left\{\begin{textstyle}\underset{i=1}{\overset{j}{\sum}}\end{textstyle}(t_{i}\!-\!1)\right\}{+}1\,\right\} 
\\[1ex]&\qquad
\leq \Max\{1,1\} = 1,
\end{align*}
and hence $t'_{k+1}=t_{k+1}$.
This immediately implies that $t'_{k'}=t_{k'}$ for any $k'>k$.
\end{proof}
\begin{prop}\label{prop:degeneracy-3-0}
If $t_{1} \geq 1$, then $t_{1}{-}t'_{1} = 1$.
\end{prop}
\begin{proof}
If $t_{1} \geq 1$, then $t'_{1}=\Max\{0,t_{1}{-}1\}=t_{1}{-}1$, and hence we obtain $t_{1}{-}t'_{1} = 1$.
\end{proof}
\begin{prop}\label{prop:degeneracy-3-1}
Let $s \!\geq\! 2$ and $0 \!\leq\! b \!\leq\! 1$.\vspace{-2ex}
If $\underset{i=1}{\overset{j}{\sum}}(t_{i}\!-\!1) \leq b\!-\!1$ for any $j \leq s$, then $t'_{1}=0$ and $\underset{i=2}{\overset{k}{\sum}}(t'_{i}\!-\!1) \leq b\!-\!1$ for any $k$, $2 \leq k \leq s$.
\end{prop}
\begin{proof}
Since $\underset{i=1}{\overset{j}{\sum}}t_{i} \leq j{-}1{+}b$ for any $j \leq s$, we have $t_{1}{-}1 \leq b{-}1 \leq 0$ and $\underset{1 \leq j \leq k}\Max\left\{\begin{textstyle}\underset{i=1}{\overset{j}{\sum}}\end{textstyle}(t_{i}\!-\!1)\right\} \leq b{-}1$.
Thus by (\ref{eq:degeneracy0}) and (\ref{eq:degeneracy1}), we have $t'_{1}=\Max\{0,t_{1}{-}b\}=0$ and 
$$
\begin{textstyle}\underset{i=1}{\overset{k}{\sum}}\end{textstyle}(t'_{i}\!-\!1) = \Min\left\{\,\begin{textstyle}\underset{i=1}{\overset{k-1}{\sum}}\end{textstyle}(t'_{i}\!-\!1)+(t_{k}\!-\!1),b\!-\!2\,\right\} \leq b\!-\!2.
$$
Thus we obtain $\underset{i=2}{\overset{k}{\sum}}(t'_{i}\!-\!1) \leq b\!-\!1$ for any $k$, $2 \leq k \leq s$.
\end{proof}
\begin{prop}\label{prop:degeneracy-3-2}
Let $s \!\geq\! 2$ and $0 \!\leq\! b \!\leq\! 1$.
If $\underset{i=1}{\overset{s}{\sum}}(t_{i}\!-\!1) \geq b\!-\!1$ and $\underset{i=1}{\overset{\ell}{\sum}}(t_{i}\!-\!1) \leq b\!-\!1$ for any $\ell$, $1 \leq \ell < s$, then $\underset{i=1}{\overset{s}{\sum}}(t_{i}{-}t'_{i}) = 1$ and there exist real numbers $\hat{t}_{s}, \bar{t}_{s} \geq 0$ such that 
$t_{s} = \hat{t}_{s} + \bar{t}_{s},\quad \underset{i=1}{\overset{s-1}{\sum}}t_{i}+\bar{t}_{s} = s{-}1{+}b,\quad
t'_{\ell}=t_{\ell},\ \ell > s,\ t'_{s} = \hat{t}_{s} + \bar{t}'_{s}$ and 
$$(t'_{1},\dots,t'_{s-1},\bar{t}'_{s}) = \xi(t_{1},\dots,t_{s-1},\bar{t}_{s}).$$
\end{prop}
\begin{proof}
If $\underset{i=1}{\overset{s}{\sum}}(t_{i}\!-\!1) \geq b\!-\!1$ and $\underset{i=1}{\overset{\ell}{\sum}}(t_{i}\!-\!1) \leq b\!-\!1$ for any $\ell<s$, then it follows
$$
\underset{1 \leq j \leq s}\Max\left\{\begin{textstyle}\underset{i=1}{\overset{j}{\sum}}\end{textstyle}(t_{i}\!-\!1)\right\} = \begin{textstyle}\underset{i=1}{\overset{s}{\sum}}\end{textstyle}(t_{i}\!-\!1) \ge b.
$$
Hence by using (\ref{eq:degeneracy2}) and Proposition \ref{prop:degeneracy-2}, we obtain 
$$
\underset{i=1}{\overset{s}{\sum}}(t_{i}{-}t'_{i}) = \Max\left\{\,\begin{textstyle}\underset{i=1}{\overset{s-1}{\sum}}\end{textstyle}(t_{i}{-}t'_{i}), 1\,\right\} = 1.
$$
Hence $t'_{\ell}=t_{\ell}$, $\ell > s$ by Proposition \ref{prop:degeneracy-4}.
Let $\bar{t}_{s}=b - \underset{i=1}{\overset{s-1}{\sum}}(t_{i}\!-\!1)$ and $\hat{t}_{s}=t_{s}{-}\bar{t}_{s}$.
We know 
$t_{s}{-}t'_{s}=1 - \underset{i=1}{\overset{s-1}{\sum}}(t_{i}{-}t'_{i}).$
The same argument also implies $\bar{t}_{s}-\bar{t}'_{s}=1 - \underset{i=1}{\overset{s-1}{\sum}}(t_{i}{-}t'_{i})$ where $(t'_{1},\dots,t'_{s-1},\bar{t}'_{s}) = \xi(t_{1},\dots,t_{s-1},\bar{t}_{s})$.
Thus we obtain $t_{s}-t'_{s}=\bar{t}_{s}-\bar{t}'_{s}$ and hence $t'_{s}-\bar{t}'_{s}=t_{s}-\bar{t}_{s}=\hat{t}_{s}$.
\end{proof}
\begin{prop}\label{prop:degeneracy-5}
Let $\xi(t_{1},\dots,t_{n})=(t'_{1},\dots,t'_{n})$.
If $\underset{i=1}{\overset{s}{\sum}}(t_{k+i-1}\!-\!1) \geq -1$ and $\underset{i=1}{\overset{\ell}{\sum}}(t_{k+i-1}\!-\!1) \leq -1$ for $1 \leq \ell < s$, then there exist $\hat{t}_{k+s-1}, \bar{t}_{k+s-1} \in \real_{+}$ such that
\begin{align*}&
\begin{textstyle}
t_{k+s-1} = \hat{t}_{k+s-1} + \bar{t}_{k+s-1},\quad \underset{i=1}{\overset{s-1}{\sum}}(t_{k+i-1}\!-\!1)+\bar{t}_{k+s-1} = 0.
\end{textstyle}
\\&
t'_{k+\ell-1}=t_{k+\ell-1},\  1 \leq \ell < s,\quad \text{and}\quad
t'_{k+s-1} = \hat{t}'_{k+s-1} + \bar{t}_{k+s-1},
\\[1ex]&
\xi(t_{1},\dots,t_{k-1},\hat{t}_{k+s-1},t_{k+s},\dots,t_{n}) = (t'_{1},\dots,t'_{k-1},\hat{t}'_{k+s-1},t'_{k+s},\dots,t'_{n}).
\end{align*}
\end{prop}
\begin{proof}
Firstly, the assumption implies that \vspace{-2ex}$\bar{t}_{k+s-1} = -\begin{textstyle}\underset{i=1}{\overset{s-1}{\sum}}\end{textstyle}(t_{k+i-1}\!-\!1) \ge 1 \ge 0$ and that $\hat{t}_{k+s-1} = t_{k+s-1}-\bar{t}_{k+s-1}=\begin{textstyle}\underset{i=1}{\overset{s}{\sum}}\end{textstyle}(t_{k+i-1}\!-\!1)+1 \ge -1 + 1 = 0$.

Secondly for $1 \leq \ell < s$, we obtain  
\begin{equation}\label{eqn:degeneracy-skip-middle}
\underset{1 \leq j \leq k{+}\ell{-}1}\Max\left\{\begin{textstyle}\underset{i=1}{\overset{j}{\sum}}\end{textstyle}(t_{i}\!-\!1)\right\} = \underset{1 \leq j \leq k{-}1}\Max\left\{\begin{textstyle}\underset{i=1}{\overset{j}{\sum}}\end{textstyle}(t_{i}\!-\!1)\right\}
\end{equation}
and by Proposition \ref{prop:degeneracy-1}, we obtain
\begin{align*}&
\underset{1 \leq j \leq k{-}1}\Max\left\{\begin{textstyle}\underset{i=1}{\overset{j}{\sum}}\end{textstyle}(t_{i}\!-\!1)\right\} - \begin{textstyle}\underset{i=1}{\overset{k-1}{\sum}}\end{textstyle}(t'_{i}\!-\!1) \geq \begin{textstyle}\underset{i=1}{\overset{k-1}{\sum}}\end{textstyle}(t_{i}\!-\!1) - \begin{textstyle}\underset{i=1}{\overset{k-1}{\sum}}\end{textstyle}(t'_{i}\!-\!1) = \begin{textstyle}\underset{i=1}{\overset{k-1}{\sum}}\end{textstyle}(t_{i}\!-\!t'_{i}) \ge 0\quad \text{and}
\\&
\underset{0 \leq j \leq k{+}\ell{-}1}\Max\left\{\begin{textstyle}\underset{i=1}{\overset{j}{\sum}}\end{textstyle}(t_{i}\!-\!1)\right\} - \!\begin{textstyle}\underset{i=1}{\overset{k+\ell-2}{\sum}}\end{textstyle}\!\!\!(t'_{i}\!-\!1) 
\geq \underset{1 \leq j \leq k{-}1}\Max\left\{\begin{textstyle}\underset{i=1}{\overset{j}{\sum}}\end{textstyle}(t_{i}\!-\!1)\right\} - \begin{textstyle}\underset{i=1}{\overset{k-1}{\sum}}\end{textstyle}(t'_{i}\!-\!1) - \begin{textstyle}\underset{i=1}{\overset{\ell-1}{\sum}}\end{textstyle}(t_{k+i-1}\!-\!1).
\end{align*}
Assuming that $t'_{k+i-1}=t_{k+i-1}$ for $1 \leq i < \ell$, we proceed as 
\begin{align*}&
\underset{1 \leq j \leq k{+}\ell{-}1}\Max\left\{\begin{textstyle}\underset{i=1}{\overset{j}{\sum}}\end{textstyle}(t_{i}\!-\!1)\right\} - \begin{textstyle}\underset{i=1}{\overset{k+\ell-2}{\sum}}\end{textstyle}\!(t'_{i}\!-\!1) 
\\&\qquad\qquad\qquad
\geq - \begin{textstyle}\underset{i=1}{\overset{\ell-1}{\sum}}\end{textstyle}(t_{k+i-1}\!-\!1) = - \begin{textstyle}\underset{i=1}{\overset{\ell}{\sum}}\end{textstyle}(t_{k+i-1}\!-\!1) + t_{k+\ell-1} \geq t_{k+\ell-1},
\end{align*}
which implies $t'_{k+\ell-1}=t_{k+\ell-1}$ by definition.
Thus, by using induction, we obtain 
\begin{equation}\label{eqn:degeneracy-skip-middle1}
t'_{k+\ell-1}=t_{k+\ell-1},\quad 1 \leq \ell < s.
\end{equation}
\par
Thirdly 
by (\ref{eqn:degeneracy-skip-middle1}) and the definition of $\bar{t}_{k+s-1}$, we have 
$$
\begin{textstyle}\underset{i=1}{\overset{k+s-2}{\sum}}\end{textstyle}\!(t'_{i}\!-\!1) = \begin{textstyle}\underset{i=1}{\overset{k-1}{\sum}}\end{textstyle}(t'_{i}\!-\!1) + \begin{textstyle}\underset{i=1}{\overset{s-1}{\sum}}\end{textstyle}(t_{k+i-1}\!-\!1) = \begin{textstyle}\underset{i=1}{\overset{k-1}{\sum}}\end{textstyle}(t'_{i}\!-\!1) - \bar{t}_{k+s-1}.
$$
On the other hand 
by (\ref{eqn:degeneracy-skip-middle}) and the definition of $\hat{t}_{k+s-1}$, we have 
\begin{align*}
\underset{1 \leq j \leq k{+}s{-}1}\Max\left\{\textstyle \underset{i=1}{\overset{j}{\sum}}(t_{i}\!-\!1)\right\} 
&= \Max\left\{\,\textstyle\underset{i=1}{\overset{j}{\sum}}(t_{i}\!-\!1), \begin{textstyle}\underset{i=1}{\overset{k+s-1}{\sum}}\end{textstyle}\!(t_{i}\!-\!1) \,;\, 1 \!\leq\! j \!<\! k\,\right\}
\\&
= \Max\left\{\,\begin{textstyle}\underset{i=1}{\overset{j}{\sum}}\end{textstyle}(t_{i}\!-\!1), \begin{textstyle}\underset{i=1}{\overset{k-1}{\sum}}\end{textstyle}(t_{i}\!-\!1)+(\hat{t}_{k+s-1}\!-\!1) \,;\, 1 \!\leq\! j \!<\! k\,\right\}.
\end{align*}
By putting $\hat{t}'_{k+s-1} = t'_{k+s-1} {-} \bar{t}_{k+s-1}$, we obtain 
\begin{equation}\label{eqn:degeneracy-skip-middle2}
\begin{aligned}
\hat{t}'_{k+s-1} %
& = \Min\left\{\,t_{k+s-1}-\bar{t}_{k+s-1},\underset{1 \leq j \leq k{+}s{-}1}\Max\left\{\,\begin{textstyle}\underset{i=1}{\overset{j}{\sum}}\end{textstyle}(t_{i}\!-\!1)\,\right\}-\begin{textstyle}\underset{i=1}{\overset{k+s-1}{\sum}}\end{textstyle}(t'_{i}\!-\!1)-\bar{t}_{k+s-1}\,\right\}.
\\&
= \Min\left\{\,\hat{t}_{k+s-1},\Max\left\{\,\begin{textstyle}\underset{i=1}{\overset{j}{\sum}}\end{textstyle}(t_{i}\!-\!1),\right.\!\right.\!
\left.\!\left.\!\begin{textstyle}\underset{i=1}{\overset{k-1}{\sum}}\end{textstyle}(t_{i}\!-\!1)+(\hat{t}_{k+s-1}\!-\!1) \,;\, 1 \!\leq\! j \!<\! k\,\right\}-\begin{textstyle}\underset{i=1}{\overset{k-1}{\sum}}\end{textstyle}(t'_{i}\!-\!1)\,\right\}.
\end{aligned}
\end{equation}
Finally, for $1 \leq \ell \leq n\!-\!k\!-\!s\!+\!1$, we obtain 
\begin{align*}&
\underset{1 \leq j \leq k+s+\ell-1}\Max\left\{\begin{textstyle}\underset{i=1}{\overset{j}{\sum}}\end{textstyle}(t_{i}\!-\!1)\right\} 
= \Max\left\{\,\begin{textstyle}\underset{i=1}{\overset{j}{\sum}}\end{textstyle}(t_{i}\!-\!1), 
\right.\\&\qquad\qquad\qquad\left.
\begin{textstyle}\underset{i=1}{\overset{k-1}{\sum}}\end{textstyle}(t_{i}\!-\!1)+(\hat{t}_{k+s-1}\!-\!1)+\begin{textstyle}\underset{i= 1}{\overset{j'}{\sum}}\end{textstyle}(t_{k+s+i-1}\!-\!1)
\,;\, 1 \!\leq\! j \!<\! k, 0 \!\leq\! j' \!\leq\! \ell\,\right\}.
\end{align*}
Since $\begin{textstyle}\underset{i=1}{\overset{k+s-2}{\sum}}\end{textstyle}\!(t'_{i}\!-\!1) = \begin{textstyle}\underset{i=1}{\overset{k-1}{\sum}}\end{textstyle}(t'_{i}\!-\!1) - \bar{t}_{k+s-1}$, we obtain 
\begin{align*}
\begin{textstyle}\underset{i=1}{\overset{k+s+\ell-2}{\sum}}\end{textstyle}\!\!(t'_{i}\!-\!1) 
&= \begin{textstyle}\underset{i=1}{\overset{k+s-2}{\sum}}\end{textstyle}\!\!(t'_{i}\!-\!1) + (t'_{k+s-1}\!-\!1)  + \begin{textstyle}\underset{i=1}{\overset{\ell-1}{\sum}}\end{textstyle}(t'_{k+s+i-1}\!-\!1) 
\\&= \begin{textstyle}\underset{i=1}{\overset{k-1}{\sum}}\end{textstyle}(t'_{i}\!-\!1)- \bar{t}_{k+s-1} + (t'_{k+s-1}\!-\!1) + \begin{textstyle}\underset{i=1}{\overset{\ell-1}{\sum}}\end{textstyle}(t'_{k+s+i-1}\!-\!1)
\\&= \begin{textstyle}\underset{i=1}{\overset{k-1}{\sum}}\end{textstyle}(t'_{i}\!-\!1) + \hat{t}'_{k+s-1}-1 + \begin{textstyle}\underset{i=1}{\overset{\ell-1}{\sum}}\end{textstyle}(t'_{k+s+i-1}\!-\!1).
\end{align*}
Hence we obtain the following equation.
\begin{equation}\label{eqn:degeneracy-skip-middle3}
\begin{aligned}&
t'_{k+s+\ell-1} = \Min\left\{\,t_{k+s+\ell-1},\Max\left\{\,\begin{textstyle}\underset{i=1}{\overset{j}{\sum}}\end{textstyle}(t_{i}\!-\!1), 
\right.\right.
\\&\qquad\qquad\qquad\left.\left.
\begin{textstyle}\underset{i=1}{\overset{k-1}{\sum}}\end{textstyle}(t_{i}\!-\!1)+(\hat{t}_{k+s-1}\!-\!1)+\begin{textstyle}\underset{i=1}{\overset{j'}{\sum}}\end{textstyle}(t_{k+s+i-1}\!-\!1) \,;\, 1 \!\leq\! j \!<\! k, 0 \!\leq\! j' \!\leq\! \ell\,\right\}
\right.
\\&\qquad\qquad\qquad\qquad\qquad\qquad\qquad\quad\left.
-\begin{textstyle}\underset{i=1}{\overset{k-1}{\sum}}\end{textstyle}(t'_{i}\!-\!1)-(\hat{t}'_{k+s-1}\!-\!1)-\begin{textstyle}\underset{i=1}{\overset{\ell-1}{\sum}}\end{textstyle}(t'_{k+s+i-1}\!-\!1)\,\right\}.
\end{aligned}
\end{equation}
Hence by equations (\ref{eqn:degeneracy-skip-middle1}), (\ref{eqn:degeneracy-skip-middle2}) and (\ref{eqn:degeneracy-skip-middle3}), we obtain
$$
\xi(t_{1},\dots,t_{k-1},\hat{t}_{k+s-1},t_{k+s},\dots,t_{n}) = (t'_{1},\dots,t'_{k-1},\hat{t}'_{k+s-1},t'_{k+s},\dots,t'_{n}).
$$
This completes the proof of the proposition.
\end{proof}

Finally, Proposition \ref{prop:degeneracy-3-1} yields the following results.

\begin{lem}\label{degeneracy-closed}
$\xi(\ass(n)) \subseteq \{0\} \times \ass(n{-}1)$ and $\xi(\mlta(n)) \subseteq \{0\} \times \mlta(n{-}1)$ which immediately implies $\xi(\mlt(n)) \subseteq \{0\} \times \mlt(n{-}1)$.
\end{lem}

\begin{lem}\label{lem:degeneracy1}
Let $(k,r,s) \in A(n)$ with $2 \leq r,s$.
For any $\rho \in \ass(r)$ and $\sigma \in \ass(s)$, the following equation holds.
\begin{align*}&
\xi(\partial_{k}(\sigma)(\rho)) = \begin{cases}\,
\partial_{k-1}(\sigma)(\xi(\rho)), & 1<k \ \& \ r>2,
\\\,
\tau, & k=2 \ \& \ r=2,
\\[1ex]\,
\partial_{1}(\xi(\sigma))(\rho), & k=1 \ \& \ s>2,
\\\,
\rho, & k=1 \ \& \ s=2,
\end{cases}
\end{align*}
\end{lem}

\begin{lem}\label{lem:degeneracy2}
Let $(k,r,s) \in A(n)$ with $2 \leq s$.
For any $\rho \in \mlta(r)$ and $\sigma \in \ass(s)$, the following equation holds.
\begin{align*}&
\xi(\delta^{a}_{k}(\sigma)(\rho)) = \begin{cases}\,
\delta^{a}_{k-1}(\sigma)(\xi(\rho)), & 1<k,
\\[1.5ex]\,
\delta^{a}_{1}(\xi(\sigma))(\rho), & k=1 \ \& \ s>2,
\\[1ex]\,
\rho, & k=1 \ \& \ s=2,
\end{cases}
\end{align*}
\end{lem}

\subsection{Degeneracies and the canonical embedding of $\zeta^{a}_{n}$}

We obtain the following.

\begin{prop}
$\zeta^{a}_{n-1}{\comp}d^{J,a}_{k}=d^{K}_{k+1}{\comp}\zeta^{a}_{n}$ %
for $1 \!\le\! k \!\le\! n$, $n \!\ge\! 2$.
\end{prop}
\begin{proof}
Let $\rho=(u_{1},\dots,u_{n}) \in \mlta(n)$ so that we have $\zeta^{a}_{n}(u_{1},\dots,u_{n}) = (0,u_{1},\dots,u_{n}{+}1{-}a) \in \ass(n{+}1)$ and $u_{1}+\cdots+u_{n} \!=\! n{-}1{+}a$.
When $k\!=\!n$, since $u_{n} \!\ge\! 1$, we have $\xi(u_{n})=u_{n}{-}1$ and $\xi(u_{n}{+}1{-}a)$ $=$ $u_{n}{-}a$, and hence we have 
\begin{align*}
\zeta^{a}_{n-1}{\comp}d^{J,a}_{n}(\rho) &= \zeta^{a}_{n}(u_{1},\dots,u_{n-2},u_{n-1}{+}u_{n}{-}1) 
\\&
= (0,u_{1},\dots,u_{n-2},u_{n-1}{+}u_{n}{-}a) 
\\&
= (0,u_{1},\dots,u_{n-2},u_{n-1}{+}\xi(u_{n}{+}1{-}a)) 
\\&
= d^{K}_{n+1}(0,u_{1},\dots,u_{n-2},u_{n-1},u_{n}{+}1{-}a) = d^{K}_{n+1}{\comp}\zeta^{a}_{n}(\rho).
\end{align*}
When $k\!=\!1$, by assuming $\xi(u_{1},\dots,u_{n}) = (u'_{1},\dots,u'_{n})$ in which $u'_{1}\!=\!0$ since $u_{1}\!<\!1$,
we obtain $\xi(u_{1},\dots,u_{n-1},u_{n}{+}1{-}a) = (u'_{1},\dots,u'_{n-1},u'_{n}{+}1{-}a)$ by Proposition \ref{prop:degeneracy-3-2} with $b\!=\!a$.
Hence we have 
\begin{align*}
\zeta^{a}_{n-1}{\comp}d^{J,a}_{1}(\rho) &= (0,u'_{2},\dots,u'_{n-1},u'_{n}{+}1{-}a) 
\\&
= (u'_{1},u'_{2},\dots,u'_{n-1},u'_{n}{+}1{-}a) 
\\&
= d^{K}_{2}(0,u_{1},\dots,u_{n}{+}1{-}a)=d^{K}_{2}{\comp}\zeta^{a}_{n}(\rho).
\end{align*}

When $1 \!<\! k \!<\! n$, we have $u_{1}+\cdots+u_{n}=n{-}1{+}a$ and $u_{1}+\cdots+u_{k-1} \le k{-}2{+}a$, and hence $u_{k}+\cdots+u_{n} \ge n{-}k{+}1$:
if $u_{k} \!\ge\! 1$, then we have $\xi(u_{k},\dots,u_{n})=(u_{k}{-}1,u_{k+1},\dots,u_{n})$ and $\xi(u_{k},\dots,u_{n-1},u_{n}{+}1{-}a)=(u_{k}{-}1,u_{k+1},\dots,u_{n-1},u_{n}{+}1{-}a)$.
Hence, in this case, we have
$$
d^{K}_{k+1}{\comp}\zeta^{a}_{n}(\rho)=(0,u_{1},\dots,u_{k-1},u_{k}{-}1,u_{k+1},\dots,u_{n})=\zeta^{a}_{n}{\comp}d^{J,a}_{k}(\rho).
$$
So, we assume that $u_{k} \!<\! 1$ and we can choose $s$ satisfying $k \!<\! s \!\le\! n$ and the hypothesis of Proposition \ref{prop:degeneracy-3-2} with $b\!=\!1$.
By assuming $\xi(u_{k},\dots,u_{s}) = (u'_{k},\dots,u'_{s})$, we have 
\begin{align*}&
\xi(u_{k},\dots,u_{n}) = (u'_{k},\dots,u'_{s},u_{s+1},\dots,u_{n})
\\[1ex]&
\xi(u_{k},\dots,u_{n-1},u_{n}{+}1{-}a) = (u'_{k},\dots,u'_{s},u_{s+1},\dots,u_{n-1},u_{n}{+}1{-}a).
\end{align*}
Hence we have the following formula also in this case:
\begin{align*}
d^{K}_{k+1}{\comp}\zeta^{a}_{n}(\rho)&=(0,u_{1},\dots,u_{k-2},u_{k-1}{+}u'_{k},\dots,u'_{s},u_{s+1},\dots,u_{n-1},u_{n}{+}1{-}a) = \zeta^{a}_{n}{\comp}d^{J,a}_{k}(\rho).
\end{align*}

It completes the proof of $\zeta^{a}_{n-1}{\comp}d^{J,a}_{k}=d^{K}_{k+1}{\comp}\zeta^{a}_{n}$ for $1\!<\!k\!<\!n$. 
\end{proof}

A direct calculation yields the following.

\begin{thm}\label{thm:degeneracy-boundaries}
For any $\rho \in \mlt(n)$ and $\sigma \in \ass(n)$, we have
\begin{enumerate}
\item
$d^{K}_{j}(t{\cdot}\Tail^{K}_{n} + (1{-}t){\cdot}\sigma) = t{\cdot}\Tail^{K}_{n-1} + (1{-}t){\cdot}d^{K}_{j}(\sigma)$.
\item
$d^{J}_{j}(t{\cdot}\Tail^{J}_{n} + (1{-}t){\cdot}\rho) = t{\cdot}\Tail^{J}_{n-1} + (1{-}t){\cdot}d^{J}_{j}(\rho)$, for $j>1$.
\end{enumerate}
\end{thm}

\subsection{A homeomorphism between an Associahedron and $J^{a}_{0}(n)$}\label{sect:homeo}

We define by induction on $n$ a homeomorphism between $\ass(n)$ and $\mlta[0](n) \subset \mlta(n)$.

Let us introduce $\Center^{a}_{n} = (0,a,1,\dots,1,2{-}a) \in \ass(n)$ for $n \geq 2$ and $\Center_{n} = \Center^{\fracinlines1/2}_{n}$, which lie in the interior of $\ass(n)$.
We remark that $\Tail^{J,1}_{n} \in \mlt[1](n)$ corresponds to $\Center^{1}_{n+1} \in \ass(n{+}1)$ by the natural homeomorphism $\zeta^{1}_{n} : \mlt[1](n) \homeo \ass(n{+}1)$.

As is first introduced in Stasheff \cite{MR0158400}, we can define another set of degeneracy operators $d^{S}_{j} : \ass(n) \to \ass(n{-}1)$ by induction using the following formulas:
\begin{align*}&
d^{S}_{j}(t{\cdot}\Center_{n} + (1{-}t){\cdot}\sigma) = t{\cdot}\Center_{n-1} + (1{-}t){\cdot}d^{S}_{j}(\sigma),\quad \sigma \in \partial\ass(n).
\end{align*}

These observations suggest us the following definition.

\begin{defn}\label{defn:omega}
Homeomorphisms $\omega^{a}_{n} : \ass(n) \rightarrow \mlta[0](n)$, $0 \!\le\! a \!\le\! 1$ are defined inductively by 
$$
\omega^{a}_{n}(t{\cdot}\Center_{n} + (1{-}t){\cdot}\partial_{k}(\tau)(\rho)) = t{\cdot}\Tail^{J,a}_{n} + (1{-}t){\cdot}\delta^{a}_{k}(\tau)(\omega^{a}_{r}(\rho)),
$$
where $\omega^{0}_{n}(\sigma)=\sigma \in \mlt[0](n)=\ass(n)$.
Then we define homeomorphisms $\eta_{n}^{a} : [0,1] \times \ass(n) \rightarrow \mlta(n)$, $0 \!<\! a \!\le\! 1$ and $\eta_{n} : [0,1] \times \ass(n) \rightarrow \mlt(n)$ by 
\par\vskip1ex\noindent\,\hfill$
\eta_{n}^{a}(t,\sigma)=\omega^{at}_{n}(\sigma), \ \eta_{n}=\eta_{n}^{1/2},
$\hfill\,\par\vskip1ex\noindent
which implies 
\begin{enumerate}
\item
$\eta^{a}_{n}(0,\sigma)=\delta^{a}_{1}(1,n)(\ast,\sigma)$ and $\eta^{1}_{n}(0,\sigma)=\partial_{2}(2,n)(\ast,\sigma)$.
\item
$\eta^{a}_{n}(1,\sigma) \in \mlta_{0}(n)$ and $\eta^{1}_{n}(1,\sigma) \in \ass_{1}(n{+}1)$.
\end{enumerate}
\end{defn}

For $\delta^{a}_{k}(\tau)(\eta^{1}_{r}(\rho)) = (u_{1},\dots,u_{n})$, we have $\sum^{\ell}_{i=1}u_{i} = \ell{-}1{+}a$ for some $\ell\leq n$.
Hence the sum of all entries of $t{\cdot}\Tail^{J,a}_{n} + (1{-}t){\cdot}\delta^{a}_{k}(\tau)(\eta^{1}_{r}(\rho))$ is $(t{\cdot}a{+}(1{-}t){\cdot}u_{1})+\sum^{\ell}_{i=2}(t{+}(1{-}t){\cdot}u_{i}) = t{\cdot}(a + \sum^{\ell}_{i=2}1) + (1{-}t){\cdot}(\sum^{\ell}_{i=1}u_{i}) = t{\cdot}(\ell{-}1{+}a) + (1{-}t){\cdot}(\ell{-}1{+}a) = \ell{-}1{+}a$.
Then we obtain $\omega^{a}_{n}(\ass(n)) \subset \mlta[0](n)$ since $\delta^{a}_{k}(\tau)(\mlta[0](r)) \subset \mlta[0](n)$. Hence $\eta^{1}_{n}$, $\eta_{n}$ and $\omega^{a}_{n}$ are well-defined.

By this definition, we easily see the following proposition.
\begin{prop}\label{prop:omega}
\begin{enumerate}
\item
$\omega_{n}^{a}(\partial_{k}(\tau)(\rho)) = \delta_{k}(\tau)(\omega_{r}^{a}(\rho))$ and $\eta_{n}^{a}(\partial_{k}(\tau)(\rho)) = \delta_{k}(\tau)(\eta_{r}^{a}(\rho))$.
\item
$d^{J}_{j}\omega_{n}^{a}(\sigma) = \omega_{n-1}^{a}(d^{S}_{j}(\sigma))$  and $d^{J}_{j}\eta_{n}^{a}(\sigma) = \eta_{n-1}^{a}(d^{S}_{j}(\sigma))$, $1 < j \leq n$.
\item
$d^{K}_{j+1}\omega_{n}^{1}(\sigma) = \omega_{n-1}^{1}(d^{S}_{j}(\sigma))$  and $d^{K}_{j+1}\eta_{n}^{1}(\sigma) = \eta_{n-1}^{1}(d^{S}_{j}(\sigma))$, $1 \leq j < n$.
\end{enumerate}
\end{prop}
Hence $\eta^{1}_{n} : [0,1] \times \ass(n) \to \ass(n{+}1)$ and $\eta_{n} : [0,1] \times \ass(n) \to \mlt(n)$ preserves both the faces and degeneracy operators except for $d_{1}$ and $d_{n+1}$.

\subsection{Other degeneracy operators}

We call a set $\{d_{j} : \ass(n) \to \ass(n-1)\,;\, 1 \leq j \leq n\}$ a \textit{degeneracy operators on $\ass(n)$} if they satisfies the condition given in Theorem \ref{thm:degeneracy-boundary-K} with $d^{K}_{j}$ replaced by $d_{j}$.
Similarly, we call a set $\{D_{j} : \mlta(n) \to \mlta(n-1)\,;\, 1 \leq j \leq n\}$ a \textit{degeneracy operators on $\mlta(n)$} if they satisfies the condition given in Theorem \ref{thm:degeneracy-boundary-K} with $d^{J}_{j}$ replaced by $D_{j}$.
We obtain the following.
\begin{thm}
The existence of an $A_{m}$-form ($m \leq \infty$) for a space $X$ for a set of degeneracies on $\ass(n)$ ($n \leq m$) implies the existence of that for any set of degeneracies on $\ass(n)$ ($n \leq m$).
\end{thm}
\begin{proof}
Let $\{d_{j}\}$ and $\{d'_{j}\}$ be two sets of degeneracy operators on $\ass(n)$.
Since $\ass(n)$ and its faces $\ass[k](r,s)$ are convex subspaces of the Euclidean space $\real^{n}$, we can construct a family of sets of degeneracy operators $\tilde{d}_{j} : I \times \ass(n) \to \ass(n{-}1), \ 1 \leq j \leq n$ as follows:
$$
\tilde{d}_{j}(u,\sigma)=(1{-}u)d_{j}(\sigma) + ud'_{j}(\sigma).
$$
Since $\partial_{k}(\tau)(\rho)$ is affine in $\rho$ and $\tau$, we can deduce that
\begin{align*}
\tilde{d}_{j}(u,\partial_{k}(\tau)(\rho)) &= (1{-}u)d_{j}\partial_{k}(\tau)(\rho) + ud'_{j}\partial_{k}(\tau)(\rho)
\\[.5ex]&= \begin{cases}\,
\partial_{k-1}(\tau)(\tilde{d}_{j}(u,\rho)), & 1 \leq j<k \ \& \ r>2,
\\[-.0ex]\,
\tau, & j=1 \ \& \ k=r=2,
\\[.5ex]\,
\partial_{k}(\tilde{d}_{j-k+1}(u,\tau))(\rho), & k \leq j < k{+}t \ \& \ t>2,
\\[-.0ex]\,
\rho, & j=k, k{+}1 \ \& \ t=2,
\\[.5ex]\,
\partial_{k}(\tau)(\tilde{d}_{j-t}(u,\rho)), & k{+}t \leq j \leq n \ \& \ r>2,
\\[-.0ex]\,
\tau, & j=n \ \& \ k=1 \ \& \ r=2.
\end{cases}
\end{align*}

Since the inclusion map $\ass[1](n) \hookrightarrow \ass(n)$ is a cofibration, the pair $(\tilde{K}(n),\tilde{L}(n))=(I,\{0\}) \times (\ass(n),\ass[1](n))$ is a DR-pair in the sense of G. W. Whitehead (see \cite{MR516508}).
It also follows that the pair $(\tilde{K}(n) \times X^{n}, \tilde{K}(n) \times \fatvee{n}X \cup \tilde{L}(n) \times X^{n})$ is naturally a DR-pair, if the base point of $X$ is non-degenerate (see also \cite{MR516508}).
Thus there exists a natural deformation retraction $R_{X}(n) : \tilde{K}(n) \times X^{n} \to \tilde{K}(n) \times \fatvee{n}X \cup \tilde{L}(n) \times X^{n}$.

Let $\{M_{i}\}$ be an $A_{m}$-form ($m \leq \infty$) of a space for a set of degeneracy operators $\{d_{j}\}$ on $\ass(n)$, and $\{d'_{j}\}$ be another set of degeneracies.
Using the above $\{\tilde{d}_{j}\}$, we can construct a homotopy $\tilde{M}_{n} : I \times \ass(n)\times X^{n} \to X$ given inductively by $\tilde{M}_{n} = (\bigcup_{j}\Psi^{j}_{X} \cup \bigcup_{k,r,s}\Phi^{k,r,s}_{X} \cup M_{n}){\circ}R_{X}(n)$;
\begin{equation*}
\begin{diagram}
\node{}
\node{\vtilde\ass(n)  \times  X^{j-1} \times \{\ast\} \times X^{n-j-1}}
 \arrow{sw,L}
 \arrow{s,L}
 \arrow{se,t}{\Psi^{j}_{X}}
\\
\node{\vtilde\ass(n) \times X^{n}}
 \arrow{e,t}{R_{X}(n)}
\node{\vtilde\ass(n)  \times  \fatvee{n}X \cup \tilde{L}(n)  \times  X^{n}}
 \arrow{e,..}
\node{X,}
\\
\node{}
\node{(I \times \ass[k](r,t) \cup \{0\} \times \ass(n)) \times X^{n}}
 \arrow{nw,J}
 \arrow{n,J}
 \arrow{ne,b}{\Phi^{k,r,t}_{X} \cup M_{n}}
\end{diagram}
\end{equation*}
where $\Psi^{j}_{X}$ and $\Phi^{k,r,t}_{X}$ are defined by 
\begin{align*}&
\Psi^{j}_{X}(u,\sigma;x_{1},\dots,x_{j-1},\ast,x_{j+1},\dots,x_{n})
\\&\qquad
=\tilde{M}_{n-1}(u,\tilde{d}_{j}(u,\sigma);x_{1},\dots,x_{j-1},x_{j+1},\dots,x_{n}),
\\[1ex]&
\Phi^{k,r,t}_{X}(u,\partial_{k}(\tau)(\rho);x_{1},\dots,x_{n})
\\&\qquad
=\tilde{M}_{r}(u,\rho;x_{1},\dots,x_{k-1},\tilde{M}_{s}(u,\tau;x_{k},\dots,x_{k+s-1}),x_{k+s},\dots,x_{n}).
\end{align*}

Then we obtain a new $A_{m}$-form $\{M'_{i}\}$ given by the formula
$$
M'_{n}(\sigma) = \tilde{M}_{n}(1,\sigma),
$$
which satisfies the {\em strict-unit} condition with respect to the set of degeneracy operators $\{d'_{j}\}$.
\end{proof}

Let us fix $a \in (0,1]$.
Similarly as above, we can show the following:
\begin{thm}
The existence of an $A_{m}$-form ($m \leq \infty$) for a map $f$ of $A_{\infty}$-spaces for a set of degeneracies on $\mlta(n)$ ($n \leq m$) implies the existence of that for another set of degeneracies on $\mlta(n)$ ($n \leq m$).
\end{thm}
\begin{proof}
Let $\{D_{j}\}$ and $\{D'_{j}\}$ be two sets of degeneracy operators on $\mlta(n)$ which are compatible with $\{d_{j}\}$ and $\{d'_{j}\}$ on $\ass(n)$, respectively.
Since $\mlta(n)$ and its faces $\mlta[k](r,s)$ and $\mlta(t,r_{1},\dots,r_{t})$ are convex subspaces of the Euclidean space $\real^{n}$, we can construct a family of sets of degeneracy operators $\tilde{D}_{j} : I \times \mlta(n) \to \mlta(n{-}1), \ 1 \!\leq\! j \!\leq\! n$ as follows:
$$
\tilde{D}_{j}(u,\sigma)=(1{-}u){\cdot}D_{j}(\sigma) + u{\cdot}D'_{j}(\sigma).
$$
Because $\delta^{a}_{k}(\tau)(\rho)$ is affine in $\rho$ and $\tau$, and $\delta^{a}(t,r_{1},\dots,r_{t})(\tau;\rho_{1},\dots,\rho_{t})$ is affine in $\tau$ and all $\rho_{i}$s, the definition of $\tilde{D}_{j}$ yields the following two formulae.
\begin{align*}&
\tilde{D}_{j}(u,\delta_{k}(\tau)(\rho)) = (1{-}u){\cdot}D_{j}\partial_{k}(\tau)(\rho) + u{\cdot}D'_{j}\partial_{k}(\tau)(\rho)
\\[.5ex]&\qquad= \begin{cases}\,
\delta^{a}_{k-1}(\tau)(\tilde{d}_{j}(u,\rho)), & 1 \leq j<k,
\\[.5ex]\,
\delta^{a}_{k}(\tilde{d}_{j-k+1}(u,\tau))(\rho), & k \leq j < k{+}t \ \& \ t>2,
\\[.5ex]\,
\delta^{a}_{k}(\tau)(\tilde{d}_{j-t}(u,\rho)), & k{+}t \leq j \leq n,
\\[.0ex]\,
\rho, & j=k \ \& \ s=2.
\end{cases}
\end{align*}
\begin{align*}&
\tilde{D}_{j}(u,\delta^{a}(t,r_{1},\dots,r_{t})(\tau;\rho_{1},\dots,\rho_{t})) 
\\&\qquad= (1{-}u){\cdot}D_{j}\delta^{a}(t,r_{1},\dots,r_{t})(\tau;\rho_{1},\dots,\rho_{t}) 
\\&\qquad\qquad\qquad\qquad+ u{\cdot}D'_{j}\delta^{a}(t,r_{1},\dots,r_{t})(\tau;\rho_{1},\dots,\rho_{t})
\\[1ex]&\qquad= \begin{cases}\,
\delta(\rho_{1},\dots, \tilde{D}_{j-s_{k-1}}(\rho_{k}),\dots,\rho_{t})(\tau), & s_{k-1} < j \leq s_{k} \ \& \ r_{k}>1,
\\[.5ex]\,
\delta(\rho_{1},\dots,\rho_{k-1},\rho_{k+1},\dots,\rho_{t})(\tilde{d}_{k}(\tau)), & j = s_{k}, \ r_{k}=1 \ \& \ t>2,
\\\,
\rho_{2}, & j=1, \ r_{1}=1 \ \& \ t=2,
\\\,
\rho_{1}, & j=n, \ r_{2}=1 \ \& \ t=2.
\end{cases}
\end{align*}

Since the inclusion map $\mlta[0](n) \hookrightarrow \mlta(n)$ is a cofibration, the pair $(\widetilde{J}^{a}(n),\widetilde{J}^{a}(n))=(I,\{0\}) \times (\mlta(n),\mlta[0](n))$ is a DR-pair.
It also follows that the pair $(\widetilde{J}^{a}(n) \times X^{n}, \widetilde{J}^{a}(n) \times \fatvee{n}X \cup \widetilde{J}^{a}(n) \times X^{n})$ is naturally a DR-pair, if the base point of $X$ is non-degenerate (see also \cite{MR516508}).
Thus there exists a natural deformation retraction $P_{X}(n) : \widetilde{J}^{a}(n) \times X^{n} \to \widetilde{J}^{a}(n) \times \fatvee{n}X \cup \widetilde{J}^{a}(n) \times X^{n}$.

Let $\{F_{i}\}$ be an $A_{m}$-form ($m \leq \infty$) of a map $f : X \to Y$ of $A_{m}$-spaces $X$ and $Y$ for a set of degeneracy operators $\{D_{j}\}$ on $\mlta(n)$ compatible with $\{d_{j}\}$ a set of degeneracy operators on $\ass(n)$, and $\{D'_{j}\}$ be another set of degeneracies compatible with $\{d'_{j}\}$.
Using the above $\{\tilde{D}_{j}\}$ and $\{\tilde{d}_{j}\}$, we can construct a homotopy $\tilde{F}_{n} : I \times \mlta(n)\times X^{n} \to Y$ given inductively by $\tilde{F}_{n} = (\bigcup_{j}\Psi^{j}_{X} \cup \bigcup_{k,r,s}\Phi^{k,r,s}_{X} \cup F_{n}){\circ}P_{X}(n)$;
\begin{equation*}
\begin{diagram}
\node{}
\node{\widetilde{J}^{a}(n)  \times  X^{i} \times \{\ast\} \times X^{n-i-1}}
 \arrow{sw,L}
 \arrow{s,L}
 \arrow{se,t}{\Psi^{j}_{f}}
\\
\node{\widetilde{J}^{a}(n) \times X^{n}}
 \arrow{e,t}{P_{X}(n)}
\node{\widetilde{J}^{a}(n)  \times  \fatvee{n}X \cup \widetilde{J}^{a}(n)  \times  X^{n}}
 \arrow{e,..}
\node{Y,}
\\
\node{}
\node{(I \times \mlta[k](r,s) \cup \{0\} \times \mlta(n))  \times  X^{n}}
 \arrow{nw,J}
 \arrow{n,J}
 \arrow{ne,b}{\Phi^{k,r,s}_{f} \cup F_{n}}
\end{diagram}
\end{equation*}
where $\Psi^{j}_{f}$ and $\Phi^{k,r,t}_{f}$ are defined by 
\begin{align*}&
\Psi^{j}_{f}(u,\sigma;x_{1},\dots,x_{j-1},\ast,x_{j+1},\dots,x_{n})
\\&\qquad
=\tilde{F}_{n-1}(u,\tilde{D}_{j}(u,\sigma);x_{1},\dots,x_{j-1},x_{j+1},\dots,x_{n}),
\\[.5ex]&
\Phi^{k,r,t}_{f}(u,\delta^{a}_{k}(\tau)(\rho);x_{1},\dots,x_{n})
\\&\qquad
=\tilde{F}_{r}(u,\rho;x_{1},\dots,x_{k-1},\tilde{M}_{s}(u,\tau;x_{k},\dots,x_{k+s-1}),x_{k+s},\dots,x_{n}),
\\[.5ex]&
\Phi^{k,r,t}_{f}(u,\delta^{a}(\rho_{1},\dots,\rho_{t})(\tau);x_{1},\dots,x_{n})
\\&\qquad
=\tilde{M}_{r}(u,\tau;\tilde{F}_{r_{1}}(u,\rho_{1};x_{1},\dots,x_{r_{1}}),\dots,\tilde{F}_{r_{1}}(u,\rho_{t};x_{\sum^{t-1}_{i=1}r_{i}+1},\dots,x_{n})).
\end{align*}
\end{proof}

From now on, we do not specify explicitly the set of degeneracy operators actually in use, unless it requires the detailed expression.

\section{Proof of Theorem \ref{thm:AS}}\label{appendix:prop:AS}

Assume that a CW complex $X$ admits an $A_{m}$-structure in the sense of Stasheff.
Then by definition, there is a sequence $\{\,q^{X}_{n},n\!\leq\!m\,\}$ of maps $q^{X}_{n} : (D^{n},E^{n}) \to (P^{n},P^{n-1})$ such that $p^{X}_{n}=q^{X}_{n}\vert_{E^{n}} : E^{n} \to P^{n-1}$ is a quasi-fibration and $E^{n}$ is contractible in $D^{n}$.

We replace the quasi-fibration $p^{X}_{m} : E^{m} \to P^{m-1}$ with a Hurewicz fibration $\tilde{p} : \widetilde{E} \to P^{m-1}$ with fibre denoted by $\widetilde{X}$.
Then there is a homotopy-equivalence inclusion $\tilde{j} : (E^{m},X) \to (\widetilde{E},\widetilde{X})$.
Let $j=\tilde{j}\vert_{X} : X \hookrightarrow \widetilde{X}$ and let $\widetilde{E}^{n}=(\tilde{p})^{-1}\vert_{P^{n-1}}$ and $\tilde{p}_{n}=\tilde{p}\vert_{\widetilde{E}^{n}} : \widetilde{E}^{n} \to P^{n-1}$.
Then by combining the arguments given in Theorem 5 of \cite{MR0158400} or \cite{Mimura86hopf} with \cite{Iwase:1983} or \cite{MR1000378}, we can construct an $A_{n}$-form for $\widetilde{X}$, using a two-sided bar construction with {\em strict-unit} $\unital{B}'(Y,X,Z)$, together with a commutative ladder between $A_{n}$-structures in the sense of Stasheff, inductively on $n \leq m$.

We remark that we can also proceed to show the existence of an $A_{m}$-form of the inclusion $j$ regarding {\em h-units} by \cite{MR1000378} which uses entirely the same arguments given in Stasheff \cite{MR0158400}, and so we skip the details in this paper.

\section{Proof of Theorem \ref{thm:ABS}}\label{appendix:thm:ABS}

It is a little bit tricky idea to consider an $A_{\infty}$-space {\em without unit}, because any $(X,e)$ a space $X$ with a base point $e$ has a sequence of maps $\{a(n)\,;\,n{\geq}1\}$ given by $a(n) : \ass(n) \times X^{n} \to \{e\} \hookrightarrow X$, which should give an $A_{\infty}$-form {\em without unit}.
Anyway, we will give a proof of Theorem \ref{thm:ABS}:
First, we define $M$ by $$\displaystyle M = \left( \bigcup_{n\geq1} \ass(n{+}1) \times X^{n} \right)/\sim,$$ where the equivalence relation `$\displaystyle \sim$' is defined as follows.
$$
(\partial_{k+1}(\sigma)(\rho);x_{1},\dots,x_{n}) \sim (\rho;x_{1},\dots,a(s)(\sigma;x_{k},\dots,x_{k+s-1}),\dots,x_{n}),
$$
for $\rho \in \ass(r{+}1)$, $\sigma \in \ass(s)$, $1 \leq k \leq r$ and $r+s{-}1=n$.

Second, we observe that $M$ has an associative multiplication `$\cdot$' given by
$$
[\rho;x_{1},\dots,x_{r}]{\cdot}[\sigma;y_{1},\dots,y_{s}] = [\rho{\cdot}\sigma;x_{1},\dots,x_{r},y_{1},\dots,y_{s}],
$$
where $\rho \in \ass(r{+}1)$ and $\sigma \in \ass(s{+}1)$ with $r+s=n$, and $\rho{\cdot}\sigma = \partial_{1}(\rho)(\sigma)=\partial_{1}(s{+}1,r{+}1)(\sigma,\rho) \in \ass(r{+}s{+}1)=\ass(n{+}1)$.
Then by the Stasheff's boundary formulas, we obtain the following proposition, which shows that the multiplication `$\cdot$' is well-defined on $M$.
\begin{prop}
For any $1 \leq k \leq r$ and $\rho \in \ass(r{+}1)$, $\sigma \in \ass(s)$ and $\tau \in \ass(t{+}1)$, we have the following relation.
\begin{enumerate}
\item
$(\partial_{k+1}(\sigma)(\rho)){\cdot}\tau = \partial_{k+1}(\sigma)(\rho{\cdot}\tau)$
\item
$\tau{\cdot}(\partial_{k+1}(\sigma)(\rho)) = \partial_{k+t}(\sigma)(\tau{\cdot}\rho)$
\end{enumerate}
\end{prop}

Again by the Stasheff's boundary formulas for $\rho \in \ass(r{+}1)$, $\sigma \in \ass(s{+}1)$ and $\tau \in \ass(t{+}1)$, we obtain
\begin{align*}
(\rho{\cdot}\sigma){\cdot}\tau 
&=(\partial_{1}(\rho)(\sigma)){\cdot}\tau = \partial_{1}(\partial_{1}(\rho)(\sigma))(\tau) 
= \partial_{1}(\partial_{1}(s{+}1,r{+}1)(\sigma,\rho))(\tau) 
\\&
= \partial_{1}(t,r{+}s{+}1)(\tau,\partial_{1}(s{+}1,r{+}1)(\sigma,\rho))
\\&
= \partial_{1}(t{+}s{+}1,r)(\partial_{1}(t{+}1,s{+}1)(\tau,\sigma),\rho)
\\&
= \partial_{1}(t{+}s{+}1,r)(\partial_{1}(\sigma)(\tau),\rho)
\\&
= \partial_{1}(\rho)(\partial_{1}(\sigma)(\tau))=\partial_{1}(\rho)(\sigma{\cdot}\tau)=\rho{\cdot}(\sigma{\cdot}\tau),
\end{align*}
which implies that $M$ has an associative multiplication {\em without unit}.
Let $j : X \hookrightarrow M$ be as follows.
$$
j(x) = [\Tail^{K}_{2};x], \ \Tail^{K}_{2}=(0,1) \in \ass(2).
$$
By using a homeomorphism $\eta^{1}_{n} : [0,1] \times \ass(n) \to \ass(n{+}1)$, we can define a homotopy $g_{n} : [0,1] \times \ass(n{+}1) \to [0,1] \times [0,1] \times \ass(n) \to [0,1] \times \ass(n) \to \ass(n{+}1)$ by the following formula:
$$
g_{n} = \eta^{1}_{n}{\circ}\kappa_{n}{\circ}(1 \times \eta^{1}_{n})^{-1}, \quad \kappa_{n}(s,t,\rho)=(st,\rho).
$$
Then we have $g_{n}(1,\tau)=\tau$ and $g_{n}(0,\tau) \in \Img\partial_{2}(2,n)=\ass[2](2,n) \subset \ass(n{+}1)$.
Since $\eta^{1}_{n}$ commutes with face operators, $g_{n}$ induces a deformation $G_{n} : [0,1] \times M \to M$ such that 
$$
G_{n}(1,x)=x \ \text{and} \ G_{n}(0,x) \in X \subset M,
$$
which implies that $X$ is a deformation retract of $M$.
Further we define a sequence of maps $h(n) : \ass(n{+}1) \times X^{n} \to M$ which gives an $A_{\infty}$-form $\{h(n);n\geq1\}$ in (our version of) the sense of Stasheff (see Appendix \ref{sect:monoid}) for the inclusion $j : X \hookrightarrow M$ as follows.
$$
h(n)(\tau;x_{1},\dots,x_{n}) = [\tau;x_{1},\dots,x_{n}]
$$
which satisfies the condition of an $A_{\infty}$-form for the inclusion $j : X \hookrightarrow M$.
We leave the details to the readers.

If further $\{a(n)\,;\,n\geq1\}$ the $A_{\infty}$-form with {\em strict unit} $e \in X$, we can replace $M$ by the following monoid $\hat{M}$ defined by the same way:
$$
\hat{M} = \left( \bigcup_{n\geq1}\ass(n{+}1) \times X^{n} \right)/\simeq,
$$
where `$\simeq$' the equivalence relation for $G$ is defined as follows.
\begin{align*}&
(\partial_{k{+}1}(\sigma)(\rho);x_{1},\dots,x_{n}) \simeq (\rho;x_{1},\dots,a(s)(\sigma;x_{k},\dots,x_{k+s-1}),\dots,x_{n}),
\\&
(\tau;x_{1},\dots,x_{j-1},e,x_{j},\dots,x_{n}) \simeq (d^{K}_{j+1}(\tau);x_{1},\dots,x_{n}).
\end{align*}
Then $\hat{e}=[\Tail^{K}_{2};e]$ gives the unit of the monoid $\hat{M}$.
The inclusion $\hat{j} : X \to \hat{M}$ and homotopy $\hat{H}_{n} : [0,1] \times \hat{M} \to \hat{M}$ are defined similarly.
Since $\eta^{1}_{n}$ commutes with degeneracy operators other than $d^{K}_{1}$, $\hat{H}_{n}$ is also well-defined and $X$ is a deformation retract of $\hat{M}$.
Similar to the case for $M$, we can observe that $j$ is an $A_{\infty}$-map regarding {\em strict-units}.

\section{$A_{m}$-form from an $A_{m}$-space to a monoid}\label{sect:monoid}

We will give here a slightly different formulation in \cite{MR0270372} from the original given of an $A_{m}$-form for a map from an $A_{m}$-space {\em without unit} to a space with an associative multiplication.

Let $(X,\{a(n)\,;\,n \!\leq\! m\})$ be an $A_{m}$-space with {\em strict-unit} $\ast \in X$.
Then it satisfies the following equations for any $\tau \in \ass(n)$, $(\rho,\sigma) \in \ass(r) \times \ass(s)$ with $r{+}s{-}1=n \geq 2$, $2 \leq r \leq n{-}1$ and $1 \leq k \leq r$.
\begin{align*}&
a(n)(\partial_{k}(\sigma)(\rho);x_{1},\dots,x_{n}) = a(r)(\rho;x_{1},\dots,a(s)(\sigma;x_{k},{\cdots}),\dots,x_{n}),
\\&
a(n)(\tau;x_{1},\dots,x_{j-1},{\ast},x_{j},\dots,x_{n-1}) = a(n{-}1)(d^{K}_{j}(\tau);x_{1},\dots,x_{n-1}).
\end{align*}
We then define an $A_{m}$-map from $X$ to $G$ a topological monoid.
\begin{defn}
A map $f : X \to G$ is an $A_{\infty}$-map if there exists an $A_{\infty}$-form $\{f(n);n \!\leq\! m\}$, $f(n) : \ass(n{+}1) \times X^{n} \to G$ ($f(1)=f$) satisfying 
\begin{align*}&
f(n)(\partial_{k+1}(\sigma)(\rho);x_{1},\dots,x_{n}) 
= f(r)(\rho;x_{1},\dots,a(s)(\sigma;x_{k},{\cdots}),\dots,x_{n}),
\\&
f(n)(\partial_{1}(\rho_{1})(\rho_{2});x_{1},\dots,x_{n}) 
= f(r_{1})(\rho_{1};x_{1},{\cdots}x_{r_{1}}) 
\cdot f(r_{2})(\rho_{2};{\cdots},x_{n}),
\\&
f(n)(\tau;x_{1},\dots,x_{j-1},{\ast},x_{j},\dots,x_{n-1}) 
= f(n{-}1)(d^{K}_{j}(\tau);x_{1},\dots,x_{n-1})
\end{align*}
and $f(\ast)=e$, for any $\tau \in \ass(n{+}1)$, $(\rho,\sigma) \in \ass(r{+}1) \times \ass(s)$ with $r{+}s{-}1=n \geq 1$, $2 \leq r \leq n{-}1$ and $1 \leq k \leq r$, and $(\rho_{1},\rho_{2}) \in \ass(r_{1}{+}1) \times \ass(r_{2}{+}1)$ with $r_{1}{+}r_{2}=n \geq 1$.
\end{defn}
Since $\partial{\ass(n{+}1)} \subset \ass(n{+}1) \smallsetminus \textrm{Int}\,\mlt(n)$ is a deformation retract, the existence of the above map implies that of an $A_{\infty}$-form $\{h(n),n \!\leq\! m\}$ for $f$, where $h(n)$ is given as a map $h(n) : \mlt(n) \times X^{n} \to G$.

If we disregard units, then we shall obtain the following definition.
Let $(X,\{a(n)\,;\,n{\geq}1\})$ be an $A_{m}$-space {\em without unit}.
Then it just satisfies the following equation for any $(\rho,\sigma) \in \ass(r) \times \ass(s)$ with $r{+}s{-}1=n \geq 2$, $2 \leq r \leq n{-}1$ and $1 \leq k \leq r$.
\begin{align*}&
a(n)(\partial_{k}(\sigma)(\rho);x_{1},\dots,x_{n}) = a(r)(\rho;x_{1},\dots,a(s)(\sigma;x_{k},{\cdots}),\dots,x_{n}).
\end{align*}
We then define an $A_{m}$-map {\em disregarding units} from $X$ to $G$ a topological space with associative multiplication.
\begin{defn}
A map $f : X \to G$ is an $A_{m}$-map {\em disregarding units} if there exists an $A_{m}$-form $\{f(n);n \!\leq\! m\}$, $f(n) : \ass(n{+}1) \times X^{n} \to G$ ($f(1)=f$) satisfying 
\begin{align*}&
f(n)(\partial_{k+1}(\sigma)(\rho);x_{1},\dots,x_{n}) 
= f(r)(\rho;x_{1},\dots,a(s)(\sigma;x_{k},{\cdots}),\dots,x_{n}),
\\&
f(n)(\partial_{1}(\rho_{1})(\rho_{2});x_{1},\dots,x_{n}) 
= f(r_{1})(\rho_{1};x_{1},{\cdots}x_{r_{1}}) 
\cdot f(r_{2})(\rho_{2};{\cdots},x_{n}),
\end{align*}
for any $(\rho,\sigma) \in \ass(r{+}1) \times \ass(s)$ with $r{+}s{-}1=n \geq 1$ and $1 \leq k \leq r$ and $(\rho_{1},\rho_{2}) \in \ass(r_{1}{+}1) \times \ass(r_{2}{+}1)$ with $r_{1}{+}r_{2}=n \geq 1$.
\end{defn}
Similarly, the existence of the above map implies the existence of an $A_{m}$-form {\em disregarding units} for $f$, $\{h(n),n \!\leq\! m\}$, where $h(n)$ is given as a map $h(n) : \mlt(n) \times X^{n} \to G$.

\section{$A_{m}$-homomorphism}\label{appendix:homomorphism}

Let $(X,\{\,a(n) : \ass(n) \times X^{n} \to X\,;\,1 \!\leq\! n \!\leq\! m\,\})$ and $(Y,\{\,b(n) : \ass(n) \times Y^{n} \to Y\,;\,1 \!\leq\! n \!\leq\! m\,\})$ be ${A}_{m}$-spaces, where we assume $a(1)=\id_{X}$ and $b(1)=\id_{Y}$.
We call a map $f : X \to Y$ an $A_{m}$-homomorphism, if it satisfies the following equation:
$$
f{\comp}a(n) = b(n){\comp}(\id \times f^{n}),\quad 1 \!\le\! n \!\le\! m.
$$

\begin{thm}\label{thm:homomorhism-map}
An $A_{m}$-homomorphism of $A_{m}$-spaces is an ${A}_{m}$-map for $1 \!\leq\! m \!\leq\! \infty$.
\end{thm}

To make our argument simpler, let us ignore the $n$-th (inessential) component of each element in $\mlta(n)$ or $\ass(n)$ in this subsection, to obtain the following implication for $0 \!\le\! a \!\le\! b \!\le\! 1$. 
\begin{align*}&
\ass(n) = \mlt[0](n) \subset \mlt[a](n) \subset \mlt[b](n) \supset \mlt^{b}_{0}(n),
\\&
\partial_{k}(\rho) = \delta^{0}_{k}(\rho) \subset \delta^{a}_{k} \subset \delta^{b}_{k}(\rho),
\qquad
d^{K}_{j} = d^{J,0}_{j} \subset d^{J,a}_{j} \subset d^{J,b}_{j}
\end{align*}
where $f \subset g$ for functions $f$ and $g$ means that $f$ is a restriction of $g$.

Let us define a map $\alpha_{n} : \mlt[1](n) \to [0,1]$ by 
\begin{equation*}
\alpha_{n}(\tau) = \underset{1 \le j < n}\Max\left\{\, \underset{i=1}{\overset{j}{\textstyle\sum}}(u_{i}\!-\!1) \,\right\}+1,\quad \tau=(u_{1},\dots,u_{n}) \in \mlt[1](n).
\end{equation*}

\begin{prop}\label{prop:alpha-function}
For $\tau \in \mlt[b](n)$, $0 \le a\!=\!\alpha_{n}(\tau) \le b$ $\iff$ $\tau \in \mlta[0](n) \subset \mlt[b](n)$.
\end{prop}
\begin{proof}
By the definition of $\mlt^{a}_{0}(n)$, $a=\alpha_{n}(\tau)$ implies $\tau \in \mlt^{a}_{0}(n)$, and vice versa.
\end{proof}

For $0 \!\le\! b \!\le\! 1$ and $(r_{1},\dots,r_{s}) \in B(s,n)$, $s \!\le\! 2$, we denote 
\begin{align*}
I(s;r_{1},\dots,r_{s}) &= \left\{\, (u_{1},\dots,u_{n}) \in \mlt[b](n) \midvert\, \forall\,i \ \underset{1 \le j \le r_{i}}{\textstyle\sum}u_{s_{i-1}+j} \ge r_{i}{-}1{+}\alpha_{n}(u_{1},\dots,u_{n}) \right\}
\\&= 
\underset{0 \le a \le b}{\textstyle\sum}\,\Theta^{a}(\ass(s) \times \mlta(r_{1}) \times \cdots \times \mlta(r_{s})),
\end{align*}
where $s_{i}=r_{1}+\cdots+r_{i}$, $1 \!\le\! i \!<\! s$.
On the other hand for $0 \!\le\! b \!\le\! 1$, we define
$$
\widetilde{I}(s;r_{1},\dots,r_{s}) = \ass(s) \times \left\{\, (a;\rho_{1},\dots,\rho_{s}) \in [0,b] \times \mlt[b](r_{1}) \times \cdots \times \mlt[b](r_{s}) \midvert \alpha_{r_{i}}(\rho_{i}) \le a \,\right\},
$$
which is a closed subset of an euclidean space.
By identifying $\{\,\rho \in \mlt[b](r) \mid \alpha_{r}(\rho) \le a\,\}$ with $\mlta(r)$, $0 \!\le\! a \!\le\! b$, which is given by changing inessential last coordinate, we obtain a map $\Psi(s;r_{1},\dots,r_{s}) : \widetilde{I}(s;r_{1},\dots,r_{s}) \to I(s;r_{1},\dots,r_{s})$ defined as follows:
$$
\Psi(s;r_{1},\dots,r_{s}) : \widetilde{I}(s;r_{1},\dots,r_{s})  \ni (\sigma;a;\rho_{1},\dots,\rho_{s}) \mapsto \Theta^{a}(\sigma,\rho_{1},\dots,\rho_{s}) \in I(s;r_{1},\dots,r_{s}).
$$
Since both $\widetilde{I}(s;r_{1},\dots,r_{s})$ and $I(s;r_{1},\dots,r_{s})$ are compact subspaces in some euclidean spaces, $\Psi(s;r_{1},\dots,r_{s})$ gives an identification map.

\begin{rem}
By definition, we have $\mlt[b](n) = \underset{(r_{1},\dots,r_{s}) \in B(s,n)}{\textstyle\bigcup}I(s;r_{1},\dots,r_{s})$.
\end{rem}

Let $0 \!\le\! a \!<\! 1$. 
Then we want to have a natural map $\pi_{n} : \mlt[a](n) \to \mlt[0](n) = \ass(n)$.

\begin{prop}\label{prop:homomorphism}
Let $0 \!<\! b \!<\! 1$ and $0 \!\le\! r \!<\! \fracinline1/b$. 
Then there is a map $\pi^{r}_{n} : \mlt[b](n) \to \mlt[br](n)$ satisfying the following three conditions.
\begin{enumerate}
\item\label{prop:homomorphism-degeneracy-space}
$\pi^{r}_{n}{\comp}d^{J,b}_{j} = d^{J,br}_{j}{\comp}\pi^{r}_{n}$,
\item\label{prop:homomorphism-domain-space}
$\pi^{r}_{n}{\circ}\delta^{b}_{k}(\rho)=\delta^{br}_{k}(\rho){\circ}\pi^{r}_{n}$, $\rho \in \ass(s)$,
\item\label{prop:homomorphism-target-space}
$\pi^{r}_{n}{\comp}\delta^{b}(\Tail^{J,b}_{1},\dots,\Tail^{J,b}_{1})=\delta^{br}(\Tail^{J,br}_{1},\dots,\Tail^{J,br}_{1})$.
\end{enumerate}
\end{prop}
\begin{proof}
Since $\mlt[b](n) = \underset{(r_{1},\dots,r_{s}) \in B(s,n)}{\textstyle\bigcup}I(s;r_{1},\dots,r_{s})$, we define $\pi^{r}_{n}$, $0 \!\le\! r \!<\! 1$, by induction on $n \ge 1$ for $\tau  = \delta^{a}(\rho_{1},\dots,\rho_{s})(\sigma) = \Psi(s;r_{1},\dots,r_{s})(\sigma;a;\rho_{1},\dots,\rho_{s}) \in I(s;r_{1},\dots,r_{s})$ as follows:\vskip1ex
\Par\begin{enumerate*}
\item $\pi^{r}_{1}(\Tail^{J,b}_{1}) = \Tail^{J,br}_{1}$.
\hitem
$\pi^{r}_{n}(\tau) = \delta^{r}(\pi^{r}_{r_{1}}(\rho_{1}),\dots,\pi^{r}_{r_{s}}(\rho_{s}))(\sigma)$.
\end{enumerate*}\vskip1ex
We can show directly that $\pi^{r}_{n}$ is well-defined, and so we left it to the readers.

As for Condition \ref{prop:homomorphism-degeneracy-space} in Proposition \ref{prop:homomorphism}, we obtain by Proposition \ref{thm:degeneracy-boundary-2} as follows:
\begin{align*}&
d^{J,br}_{j}{\comp}\pi^{r}_{n}(\tau) = d^{J,ar}_{j}{\comp}\pi^{r}_{n}{\comp}\delta^{a}(\rho_{1},\dots,\rho_{s})(\sigma) = d^{J,ar}_{j}{\comp}\delta^{ar}(\pi^{r}_{r_{1}}(\rho_{1}),\dots,\pi^{r}_{r_{s}}(\rho_{s}))(\sigma)
\\[1ex]& \ =\begin{cases}\,
\delta^{ar}(\pi^{r}_{r_{1}}(\rho_{1}),\dots,d^{J,ar}_{j-s_{k-1}}(\pi^{r}_{r_{k}}(\rho_{k})),\dots,\pi^{r}_{r_{s}}(\rho_{s}))(\sigma), & s_{k-1} \!<\! j \!\leq\! s_{k}, \ r_{k}\!>\!1,
\\[.5ex]\,
\delta^{ar}(\pi^{r}_{r_{1}}(\rho_{1}),\dots,\pi^{r}_{r_{k-1}}(\rho_{k-1}),\pi^{r}_{r_{k+1}}(\rho_{k+1}),\dots,\pi^{r}_{r_{s}}(\rho_{s}))(d^{K}_{k}(\sigma)), & j \!=\! s_{k}, \ r_{k}\!=\!1, \ t\!>\!2,
\\[.0ex]\,
\pi^{r}_{r_{2}}(\rho_{2}), & j\!=\!1, \ r_{1}\!=\!1, \ t\!=\!2,
\\[.5ex]\,
\pi^{r}_{r_{1}}(\rho_{1}), & j\!=\!n, \ r_{2}\!=\!1, \ t\!=\!2,
\end{cases}
\\[1ex]& \ =\begin{cases}\,
\pi_{n-1}{\comp}\delta^{a}(\rho_{1},\dots,d^{J,a}_{j-s_{k-1}}(\rho_{k}),\dots,\rho_{s})(\sigma), & s_{k-1} \!<\! j \!\leq\! s_{k}, \ r_{k} \!>\! 1,
\\[.5ex]\,
\pi^{r}_{n-1}{\comp}\delta^{a}(\rho_{1},\dots,\rho_{k-1},\rho_{k+1},\dots,\rho_{s})(d^{K}_{k}(\sigma)), & j \!=\! s_{k}, \ r_{k} \!=\! 1, \ t \!>\! 2,
\\[.0ex]\,
\pi^{r}_{n-1}(\rho_{2}), & j \!=\! 1, \ r_{1} \!=\! 1, \ t \!=\! 2,
\\[.5ex]\,
\pi^{r}_{n-1}(\rho_{1}), & j \!=\! n, \ r_{2} \!=\! 1, \ t \!=\! 2,
\end{cases}
\\[1ex]& \ = 
\pi^{r}_{n-1}{\comp}d^{J,a}_{j}{\comp}\delta^{a}(\rho_{1},\dots,\rho_{s})(\sigma) = \pi_{n-1}{\comp}d^{J,b}_{j}(\tau),
\end{align*}
where $s_{k}=r_{1}+\cdots+r_{k}$, $0 \le k \le t$, which implies Condition \ref{prop:homomorphism-degeneracy-space} in Proposition  \ref{prop:homomorphism}.

As for Condition \ref{prop:homomorphism-domain-space} in Proposition \ref{prop:homomorphism}, we obtain by Proposition \ref{prop:boundary-J} as follows. 
\begin{align*}
\delta^{br}_{k}(\rho){\circ}\pi^{r}_{n}(\tau) &= \delta^{ar}_{k}(\rho){\comp}\pi^{r}_{n}{\comp}\delta^{a}(\rho_{1},\dots,\rho_{s})(\sigma) 
\\&= \delta^{ar}_{k}(\rho){\comp}\delta^{ar}(\pi^{r}_{r_{1}}(\rho_{1}),\dots,\pi^{r}_{r_{s}}(\rho_{s}))(\sigma) 
\\&=
\delta^{ar}(\pi_{r_{1}}(\rho_{1}),\dots,\delta^{ar}_{k'}(\rho)(\pi^{r}_{r_{j}}(\rho_{j})),\dots,\pi^{r}_{r_{s}}(\rho_{s}))(\sigma)
\\&=
\delta^{ar}(\pi^{r}_{r_{1}}(\rho_{1}),\dots,\pi^{r}_{r_{j}}(\delta^{a}_{k'}(\rho)(\rho_{j})),\dots,\pi^{r}_{r_{s}}(\rho_{s}))(\sigma)
\\&=
\pi^{r}_{n}{\comp}\delta^{a}(\rho_{1},\dots,\delta^{a}_{k'}(\rho)(\rho_{j}),\dots,\rho_{s})(\sigma)
\\&=
\pi^{r}_{n}{\comp}\delta^{a}_{k}(\rho){\comp}\delta^{a}(\rho_{1},\dots,\rho_{s})(\sigma) = \pi^{r}_{n}{\comp}\delta^{b}_{k}(\rho)(\tau),
\end{align*}
where $k'=k-(r_{1}{+}\cdots{+}r_{j-1})$, \,$1 \leq k' \leq r_{j}$, which implies Condition \ref{prop:homomorphism-domain-space}.

As for Condition \ref{prop:homomorphism-target-space} in Proposition \ref{prop:homomorphism}, it is clear by definition.
\end{proof}

\begin{cor}\label{cor:homeomorphism}
Let $0 \!<\! a, \,b \!<\! 1$, Then there is a homeomorphism $\psi^{a,b} : \mlt[a](n) \homeo \mlt[b](n)$ satisfying
\begin{enumerate}
\item\label{cor:homeomorphism-degeneracy-space}
$\psi^{a,b}{\comp}d^{J,a}_{j}=d^{J,b}_{j}{\comp}\psi^{a,b}$,
\item\label{cor:homeomorphism-domain-space}
$\psi^{a,b}{\circ}\delta^{a}_{k}(\sigma)=\delta^{b}_{k}(\sigma){\circ}\psi^{a,b}$, where $\sigma \in \ass(s)$,
\item\label{cor:homeomorphism-target-space}
$\psi^{a,b}{\comp}\delta^{a}(\Tail^{J,a}_{1},\dots,\Tail^{J,a}_{1})=\delta^{b}(\Tail^{J,0}_{1},\dots,\Tail^{J,0}_{1})$. 
\end{enumerate}
\end{cor}

\begin{cor}\label{cor:homomorphism}
When $r \!=\! 0$, we obtain $\mlt[0](n)=\ass(n)$ and the following for $\pi_{n} = \pi^{0}_{n}$.
\begin{enumerate}
\item\label{cor:homomorphism-degeneracy-space}
$\pi_{n}{\comp}d^{J,a}_{j}=d^{K}_{j}{\comp}\pi_{n}$,
\item\label{cor:homomorphism-domain-space}
$\pi_{n}{\circ}\delta^{a}_{k}(\sigma)=\partial_{k}(\sigma){\circ}\pi_{n}$, where $\sigma \in \ass(s)$,
\item\label{cor:homomorphism-target-space}
$\pi_{n}{\comp}\delta^{a}(\Tail^{J,a}_{1},\dots,\Tail^{J,a}_{1})=\delta^{0}(\Tail^{J,0}_{1},\dots,\Tail^{J,0}_{1})=\id$, where $\Tail^{J,0}_{1} = \Tail^{K}_{1}=(0)$. 
\end{enumerate}
\end{cor}

Equations in Corollary \ref{cor:homomorphism} for $\pi_{n}$ immediately imply Theorem \ref{thm:homomorhism-map}.
In fact, we can define $A_{m}$-form $\{\,f(n) : \mlt(n) \times X^{n} \to Y ; 1  \!\le\! n \!\le\! m\,\}$ of $f$ by the following equation.
$$
f(n) = f{\comp}a(n){\comp}(\pi_{n}\times \id^{n}_{X}) = b(n){\comp}(\pi_{n}\times f^{n}).
$$

\section{Composition of $A_{m}$-forms}

First, we extend the definition of $\delta^{a}$:
let $0 \!\leq\! a \!<\! b \!\leq\! 1$ and $r=\frac{b-a}{1-a}$.
\begin{defn}%
For $(r_{1},\dots,r_{t}) \in B(t,n)$ with $t \ge 2$, we have $\Theta^{a}(\mlt[r](t) \times \mlta(r_{1}) \times \cdots \times \mlta(r_{t})) \subset \mlt[b](n)$, and so we define $\mlt_{a}^{b}(t;r_{1},\dots,r_{t}) = \Theta^{a}(\mlt[r](t) \times \mlta(r_{1}) \times \cdots \times \mlta(r_{t}))$.
\end{defn}
Then we can easily obtain the following proposition.
\begin{prop}%
$\mlt[b](n) = \mlta(n) \cup \underset{(r_{1},\dots,r_{t}) \in B(t,n)}{\bigcup} \mlt_{a}^{b}(t;r_{1},\dots,r_{t}),\quad n\geq2$.
\end{prop}
\begin{proof}
We clearly have $\mlt[b](n) \supset \mlta(n) \cup \underset{(r_{1},\dots,r_{t}) \in B(t,n)}{\textstyle\bigcup} \mlt_{a}^{b}(t;r_{1},\dots,r_{t})$.
A direct computation shows the converse implication.
\end{proof}

Now, let $X$, $Y$ and $Z$ be spaces with $A_{m}$-forms $\{\,a(n)\,;\,1 \!\le\! n \!\le\! m\,\}$, $\{\,b(n)\,;\,1 \!\le\! n \!\le\! m\,\}$ and $\{\,c(n)\,;\,1 \!\le\! n \!\le\! m\,\}$, respectively.
Further, let $f : X \to Y$ and $g : Y \to Z$ be maps with $A_{m}$-forms $\{\,F(n)\,;\,1 \!\le\! n \!\le\! m\,\}$ and $\{\,G(n)\,;\,1 \!\le\! n \!\le\! m\,\}$, respectively.
Then the composition $h = g{\comp}f : X \to Z$ has an $A_{m}$-form $\{\,H(n)\,;\,1 \!\le\! n \!\le\! m\,\}$:
if we take $a = \fracinline1/4$ and $b = \fracinline1/2$, then we must have $r = \fracinline1/3$.
By Corollary \ref{cor:homeomorphism}, we may assume that the adjoints $\ad{F_{n}}$ and $\ad{G_{n}}$ of $F_{n}$ and $G_{n}$ are defined on $\mlt[\fracinlines1/4](n)$ and $\mlt[\fracinlines1/3](n)$, respectively:
\begin{align*}&
\ad{F(n)} : \mlt[\fracinlines1/4](n) \to \Map(X^{n},Y),
\\&
\ad{G(n)} : \mlt[\fracinlines1/3](n) \to \Map(Y^{n},Z).
\end{align*}
Then the adjoint $\ad{H_{n}} : \mlt(n) \to \Map(X^{n},Z)$ of $H_{n}$ can be defined as follows:
\begin{align*}&
\ad{H_{n}}|_{\text{\small$\mlt[\fracinlines1/4](n)$}}(\tau) = g{\comp}F_{n}(\tau),
\\&
\ad{H_{n}}|_{\text{\small$\mlt_{\fracinlines1/4}^{\fracinlines1/2}(t;r_{1},\dots,r_{t})$}}{\comp}\Theta^{\fracinlines1/4}(\rho;\rho_{1},\dots,\rho_{t}) = G_{t}(\rho){\comp}(F_{r_{1}}(\rho_{1}) \times \cdots \times F_{r_{t}}(\rho_{t})).
\end{align*}
Since $\Theta^{a} : \mlt[r](t) \times \mlta(r_{1}) \times \cdots \times \mlta(r_{t}) \to \mlt[b](n)$ is a homeomorphism, the above data determines $\ad{H_{n}}$.
By taking adjoint, we obtain a desired $A_{m}$-form $\{\,H_{n} : \mlt(n) \times X^{n} \to Z \,;\, 1 \!\le\! n \!\le\! m\,\}$, which is well-defined and continuous satisfying all the conditions required for an $A_{m}$-form.
Details are left to the reader.

\section{Associahedron and Multiplihedron}

\subsection{Shadows of trivalent trees and Associahedron}\label{app:trivalent-tree}

Boardman and Vogt gave in \cite{MR420609} an alternative description of Stasheff's Associahedron $\ass(n)$ as the convex hull of the set of trivalent trees each of which has one root and $n$ top-branches.
A branching point is called a node, from which one edge is going down (to left {\em or} right) and two edges are going up (to left {\em and} right).
Hence, a trivalent tree $t$ with one root and $n$ top-branches has exactly $n{-}1$ nodes.
\par
For a trivalent tree $t$, we give an order to top-branches from the left as $1$-st top-branch, $2$-nd top-branch, $\cdots$, $n$-th top-branch.
Then we count the number of nodes lying on the straight line going down to left from the $k$-th top-branch and denote it by $a_{k}(t)$.
\begin{center}
\includegraphics[height=30ex]{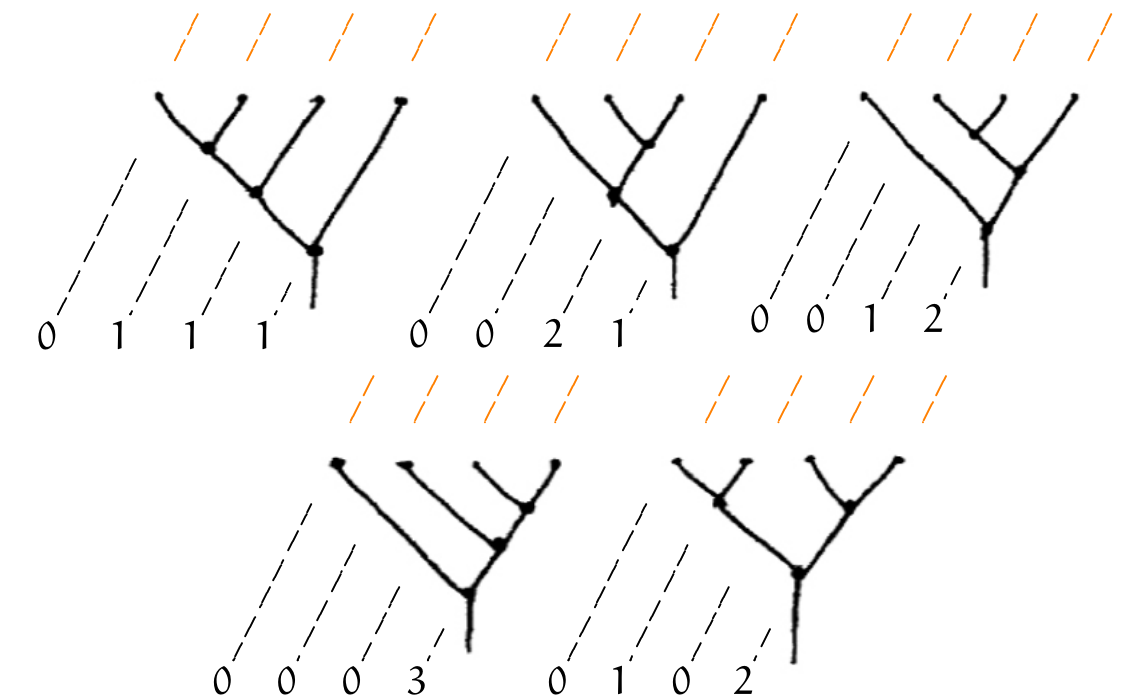}
\end{center}
\par\noindent
Similarly, we denote by $b_{k}(t)$ the number of nodes lying on the straight line going down to right from the $k$-th top-branch of $t$.

The sequence of numbers $a(t) = (a_{1}(t),a_{2}(t),\dots,a_{n}(t))$ is in $\integral_{+}^{n} \subset \integral^{n}$, where $\integral_{+}$ is the set of non-negative integers, and is satisfying 
\begin{align*}&
a_{1}(t)=0, \quad
a_{2}(t) \leq 1, \quad
a_{3}(t) \leq 2 - a_{2}(t),\quad
\\&\quad \cdots \quad
a_{k}(t) \leq k{-}1 - (a_{2}(t) + \cdots + a_{k-1}(t)), 
\quad (1<k<n)
\\&
\qquad
\cdots \quad 
a_{n}(t) = n{-}1 - (a_{2}(t) + \cdots + a_{n-1}(t)),
\end{align*}
since $a_{1}(t) + \cdots + a_{k}(t)$ is at most $k{-}1$ for any $k$ and $a_{1}(t) + \cdots + a_{n}(t) = n{-}1$ the total number of nodes.
Hence $a(t)$ is in the set
$$
\ass[L](n) = 
\left\{\, (a_{1},\dots,a_{n}) \in \integral_{+}^{n} \,\midvert\, \!\!\begin{array}{l}a_{j} \leq \sum^{j-1}_{i=1}(1{-}a_{i}), \ 1 \leq j {<} n \\[1ex] a_{n}=\sum^{n-1}_{i=1}(1{-}a_{i}) \end{array}\!\!\,\right\}
$$
Similarly, $b(t) = (b_{1}(t),b_{2}(t),\dots,b_{n}(t))$ is in the set
$$
K'_{L}(n) = 
\left\{\, (b_{1},\dots,b_{n}) \in \integral_{+}^{n} \,\midvert\, \!\!\begin{array}{l}b_{j} \leq \sum_{i=j+1}^{n}(1{-}b_{i}), \ 1 {<} j \leq n \\[1ex] b_{1}=\sum^{n}_{i=2}(1{-}b_{i}) \end{array}\!\!\,\right\}
$$
Then we can easily see the following.
\begin{prop}
\begin{enumerate}
\item
$\displaystyle
\# \ass[L](n) = C_{n-1}
$, the Catalan number\vspace{1ex}
\item
$\displaystyle
\ass[L](n) = 
\left\{\, a(t) \in \integral_{+}^{n} \,\midvert\, \text{\begin{minipage}[c]{50mm}\baselineskip18pt
$t$ is a trivalent tree with one root and $n$ top-branches
\end{minipage}} \,\right\}.
$\vspace{1ex}
\item
$\displaystyle
K'_{L}(n) = 
\left\{\, (b_{1},\dots,b_{n}) \in \integral_{+}^{n} \,\midvert\, (b_{n},\dots,b_{1}) \in \ass[L](n)\,\right\}.
$
\end{enumerate}
\end{prop}

Conversely assume that $(a_{1},\dots,a_{n})$ is in $\ass[L](n)$.
Then we can construct a trivalent tree $t$ with one root and $n$ top-branches such that $a(t) = (a_{1},\dots,a_{n})$ using the information that $t$ must have exactly $a_{k}$ nodes on the line going down to left from the $k$-th top-branch.
Thus $\ass[L](n)$ is in one-to-one correspondence with $K'_{L}(n)$.

Since the $k$-th top branch is going down either to left or to right, we have either $a_{k}>0$ and $b_{k}=0$, or $a_{k}=0$ and $b_{k}>0$.
\begin{prop}
If $(a_{1},\dots,a_{n}) \in \ass[L](n)$ and $(b_{1},\dots,b_{n}) \in K'_{L}(n)$ are shadows of the same tree, then $a_{k}{\cdot}b_{k}\!=\!0$ and $a_{k}\!+\!b_{k}\!>\!0$ for all $k$.
\end{prop}
Let $c_{k}(t)=a_{k}(t){-}b_{k}(t)$, $1 \leq k \leq n$, which is either $a_{k}(t)$ or $-b_{k}(t)$.
Thus we can recover $(a_{1}(t),\dots,a_{n}(t))\in \ass[L](n)$ and $(b_{1}(t),\dots,b_{n}(t)) \in K'_{L}(n)$ from $(c_{1}(t),\dots,c_{n}(t))$:
\begin{align*}&
a_{k}(t)=\Max\{0,c_{k}(t)\}, \ \ 1\!<\!k\!<\!n, \ \ a_{1}(t)=0, \ \ a_{n}=n{-}1-\!\textstyle\sum^{n-1}_{k=2}a_{k},
\\&
b_{k}(t)=\Max\{0,-c_{k}(t)\}, \ \ 1\!<\!k\!<\!n, \ \ b_{n}(t)=0, \ \ b_{1}=n{-}1-\!\textstyle\sum^{n-1}_{k=2}b_{k}.
\end{align*}
So, all the information of $t$ is in $c(t)=(c_{2}(t),\dots,c_{n-1}(t))$ which is in
$$
K^{T}_{L}(n) = \left\{\, c(t) \in \integral_{+}^{n-2} \,\midvert\, \text{\begin{minipage}[c]{50mm}\baselineskip18pt
$t$ is a trivalent tree with one root and $n$ top-branches
\end{minipage}} \,\right\}
$$

\medskip

We now introduce two more definitions similarly to our $\ass(n)$:
\begin{defn}
\begin{enumerate*}
\item
Let $K'(n)$ be the convex hull of $K'_{L}(n)$.
\vitem
\begin{minipage}[t]{0.9\textwidth}\baselineskip18pt
Let $K^{T}(n)$ be the convex hull of $K^{T}_{L}(n)$, which can be identified with the Associahedron due to Boardman and Vogt \cite{MR420609}.
\end{minipage}
\end{enumerate*}
\end{defn}
\begin{center}
\setlength\unitlength{.3mm}
\begin{picture}(150,140)(100,10)
\put(130,38)	{\makebox(0,0)[c]{$(-1,-1)$}}
\put(245,38)	{\makebox(0,0)[c]{$(1,-1)$}}
\put(105,100)	{\makebox(0,0)[c]{$(-2,1)$}}
\put(240,100)	{\makebox(0,0)[c]{$(1,1)$}}
\put(160,140)	{\makebox(0,0)[c]{$(-1,2)$}}
\put(190,58)	{\makebox(0,0)[c]{$(0,0)$}}
\put(190,15)	{\makebox(0,0)[c]{$(K^{T}(4) \subset \real^{2})$}}
\linethickness{1.0mm}		
\thicklines			
\put(130,100)	{\line(1,1){30}}
\put(130,100)	{\line(1,-2){30}}
\put(160,130)	{\line(2,-1){60}}
\put(220,40)	{\line(0,1){60}}
\put(220,40)	{\line(-1,0){60}}
\put(190,69)	{\line(0,1){2}}
\put(189,70)	{\line(1,0){2}}
\end{picture}
\end{center}

We then obtain the following proposition.
\begin{prop}
$K'(n) = \left\{(t_{1},\dots,t_{n})\,\midvert\,(t_{n},\dots,t_{1}) \in \ass(n)\right\}$.
\end{prop}
We can easily observe that the face operators for $K'(n)$ is nothing but $\partial'$ introduced in \S \ref{subsect:top-operad}.
In $\real^{n}$, we have hyper planes defined as
\begin{align*}&
H^{n-1} : x_{1}+{\cdots}+x_{n}=n{-}1, 
\\&
R^{n-1}_{1} : x_{1}=0\quad \text{and}\quad
R^{n-1}_{n} : x_{n}=0,
\\&
H_{1}^{n-2}=H^{n-1} \cap R^{n-1}_{1} \homeo \real^{n-2},
\\&
H_{n}^{n-2} = H^{n-1} \cap R^{n-1}_{n} \homeo \real^{n-2}.
\end{align*}
Then we can easily observe that $\ass(n) \subset H_{1}^{n-2}$ and $K'(n) \subset H_{n}^{n-2}$:
\begin{center}
\setlength\unitlength{.3mm}
\begin{picture}(120,130)(20,-25)
\linethickness{1.0mm}		
\thicklines			
\put(-54, 92)	{\makebox(0,0)[r]{\small$(0,0,2,1)$}}
\put(-54, 50)	{\makebox(0,0)[r]{\small$(0,0,1,2)$}}
\put(-54,  8)	{\makebox(0,0)[r]{\small$(0,0,0,3)$}}
\put( -6,  8)	{\makebox(0,0)[l]{\small$(0,1,0,2)$}}
\put( -6, 50)	{\makebox(0,0)[l]{\small$(0,1,1,1)$}}
\put( -30,-15)	{\makebox(0,0)[c]{$(\ass(4) \subset H_{1}^{2} \homeo \real^{2})$}}
\put(-50, 50)	{\line(0,1){40}}
\put(-50, 50)	{\line(0,-1){40}}
\put(-50, 90)	{\line(1,-1){40}}
\put( -10, 10)	{\line(0,1){40}}
\put( -10, 10)	{\line(-1,0){40}}
\put(-51, 50)	{\line(1,0){2}}
\put(214, 15)	{\makebox(0,0)[l]{\small$(1,2,0,0)$}}
\put(170, 10)	{\makebox(0,0)[c]{\small$(2,1,0,0)$}}
\put(126, 15)	{\makebox(0,0)[r]{\small$(3,0,0,0)$}}
\put(126, 65)	{\makebox(0,0)[r]{\small$(2,0,1,0)$}}
\put(173, 68)	{\makebox(0,0)[l]{\small$(1,1,1,0)$}}
\put(170,-15)	{\makebox(0,0)[c]{$(K'(4) \subset H_{4}^{2} \homeo \real^{2})$}}
\put(210, 20)	{\line(-1,1){40}}
\put(210, 20)	{\line(-1,0){80}}
\put(130, 60)	{\line( 1,0){40}}
\put(130, 60)	{\line(0,-1){40}}
\put(180, 19)	{\line(0,1){2}}
\end{picture}
\end{center}
Let us summarize properties of $\ass(n)$ family.
\par\noindent
\begin{prop}
Let $\displaystyle C_{n}=\frac{{}_{2n}\mathrm{C}_{n}}{n{+}1}$ the Catalan number.
\begin{enumerate}
\item
$\# \ass[L](n) = \# K'_{L}(n) = C_{n-1}$.
\item
$\ass(n)$ is a convex hull of $\ass[L](n)$.
\item
$K'(n)$ is a convex hull of $K'_{L}(n)$.
\item
$\ass(n) \homeo K'(n)$ as polytopes.
\item
$\ass(n) \cap L = \ass[L](n)$, if we ignore first and last coordinates.
\item
$K'(n) \cap L = K'_{L}(n)$, if we ignore first and last coordinates.
\end{enumerate}
\end{prop}
$\ass(n)$ and $K'(n)$ are mirror images with each other, and are constructed directly by taking shadows of trivalent trees on the integral lattice, where we can play our games.

\subsection{Language of bearded trees and Multiplihedra}

First, we introduce a language of trees in terms of (Reverse) Polish Notation.
For any trivalent tree $t$ with one root and $n$ top-branches, $n \geq 1$, we define a word $w(t)$ of a tree $t$ by the following way:
\begin{enumerate}
\item
assign a word `$x_{i}$' to the $i$-th top-branch from the left.
\item
if the two upper branches of a node is assigned by a word `$w_{1}$' and `$w_{2}$', then assign a word `$w_{1}w_{2}\mnode$' to its lower branch.
\item
if the root branch is assigned by a word `$w$', we define $w(t)$ the word of a tree $t$ to be $w$, i.e, $w(t)=w$.
\end{enumerate}
This defines the set $\Wis(n)$ of all words $w(t)$ of trivalent trees $t$ with one root and $n$ top-branches:
$$
\Wis(n)=\left\{w(t)\,\midvert\,\text{\begin{minipage}{50mm}\baselineskip18pt
$t$ is a trivalent tree with one root and $n$ top-branches\end{minipage}}\right\}
$$
Similarly, we obtain another word $w(t)$ for $t$.
\begin{enumerate}
\item
assign a word `$x_{i}$' to the $i$-th top-branch from the left.
\item
if the two upper branches of a node is assigned by a word `$w'_{1}$' and `$w'_{2}$', then assign a word `$\mnode w'_{1}w'_{2}$' to its lower branch.
\item
if the root branch is assigned a word `$w'$', we define $w'(t)$ the word of a tree $t$ to be $w'$, i.e, $w'(t)=w'$.
\end{enumerate}
\begin{center}
\includegraphics[height=17ex]{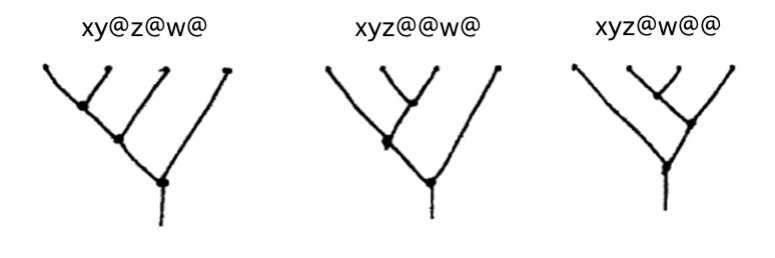}
\vskip-2ex
\includegraphics[height=17ex]{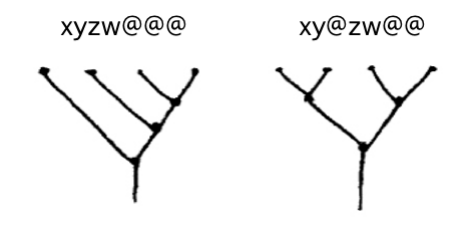}
\end{center}
\par\vskip-2ex\noindent
This defines another set $\Wis'(n)$ of all words $w(t)$ of trivalent trees $t$ with one root and $n$ top-branches.

Since the number of at-marks (`$\mnode$') between $x_{i}$ and $x_{i+1}$ in the word `$w(t)$' gives the number of nodes in the down-to-left line from the $i$-th top-branch of $t$ a trivalent tree with one root and $n$ top-branches:
\begin{align*}&
a_{i}(t) = \left\langle\text{\begin{minipage}{70mm}\baselineskip18pt
the number of at-marks appearing between $x_{i}$ and $x_{i+1}$ in the word `$w(t)$'\end{minipage}}\right\rangle, \ i<n,
\\&
a_{n}(t) = \left\langle\text{\begin{minipage}{70mm}\baselineskip18pt
the number of at-marks appearing after $x_{n}$ in the word `$w(t)$'\end{minipage}}\right\rangle.
\end{align*}
Thus we can identify $\ass[L](n)$ with $\Wis(n)$.
Similarly, we obtain
\begin{align*}&
b_{i}(t) = \left\langle\text{\begin{minipage}{70mm}\baselineskip18pt
the number of at-marks appearing between $x_{i-1}$ and $x_{i}$ in the word `$w(t)$'\end{minipage}}\right\rangle, \ i>1,
\\&
b_{1}(t) = \left\langle\text{\begin{minipage}{70mm}\baselineskip18pt
the number of at-marks appearing before $x_{1}$ in the word `$w(t)$'\end{minipage}}\right\rangle.
\end{align*}
Thus we can also identify $K'_{L}(n)$ with $\Wis'(n)$.

\smallskip

Second, we extend the idea to the one for Multiplihedra.
Let us consider a `bearded tree' which is a trivalent tree with one root, $n$ top-branches and several beards each of which comes out from just below a node or the top-edge of a top-branch, and every way from a top-edge down to the root meets exactly one beard.
Since a node is on a way down to the root from a top-edge, we may call it upper or lower, if it is upper a beard or lower a beard, resp.
For any bearded tree $\check{t}$ of one root and $n$ top-branches, we define $w(\check{t})$ a word of $\check{t}$ as follows:
\begin{enumerate}
\item
assign a word `$x_{i}$' to the $i$-th top-branch from the left.
\item
if a top-branch or a node is assigned by a word `$w$', where a beard is attached to its lower part, then assign $w\natural$ to the beard.
\item
if the two upper branches of a node have no beared and are assigned by words `$w_{1}$' and `$w_{2}$' then assign a word `$w_{1}w_{2}\unode$' to the node. 
\item
if the two upper branches of a node is assigned by a word `$w_{1}$' and `$w_{2}$' and its lower branch has no beard, then assign a word `$w_{1}w_{2}\lnode$' to the node.
\item
if the root branch is assigned by a word `$w$', we define $w(\check{t})$ the word of a tree $\check{t}$ to be $w$, i.e, $w(\check{t})=w$.
\end{enumerate}
\begin{center}
$x_{1}x_{2}\unode\natural x_{3}\natural\lnode x_{4}\natural\lnode$
\hskip4mm
$x_{1}\natural x_{2}x_{3}\unode\natural\lnode x_{4}\natural\lnode$
\hskip4mm
$x_{1}\natural x_{2}x_{3}\unode x_{4}\unode\natural \lnode$
\\
\includegraphics[height=13ex]{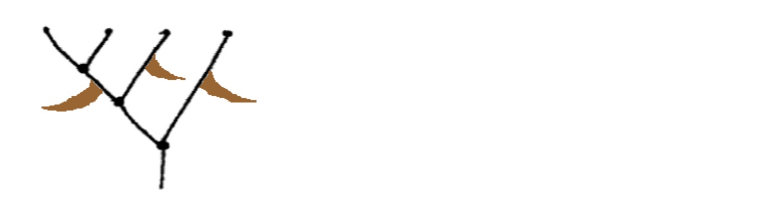}
\vskip-13ex
\includegraphics[height=13ex]{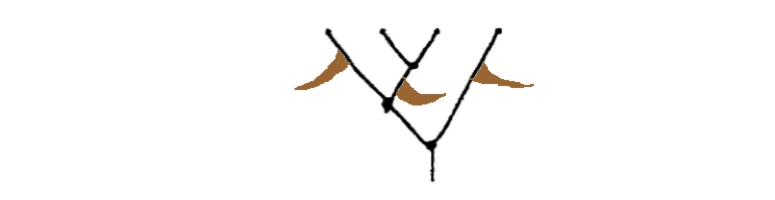}
\vskip-13ex
\includegraphics[height=13ex]{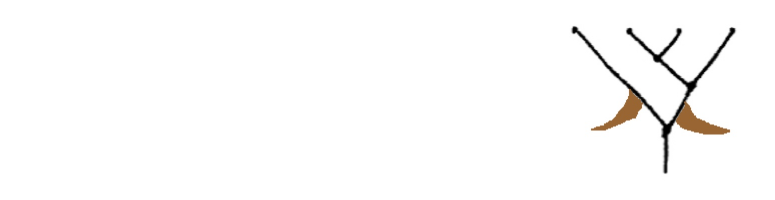}
\end{center}
\par\vskip-2ex\noindent
This defines the set $\Ent(n)$ of all extended words $w(\check{t})$ of bearded trees $\check{t}$ with one root and $n$ top-branches:
$$
\Ent(n)=\left\{\, w(\check{t})\,\midvert\,\text{\begin{minipage}{50mm}\baselineskip18pt
$\check{t}$ is a bearded tree with one root and $n$ top-branches\end{minipage}} \,\right\}
$$
Let us fix a real number $a \in [0,1]$.
For a word $w$ in $\Ent(n)$, we obtain an $n$-tuple $(v^{a}_{1}(\check{t}),\dots,v^{a}_{n}(\check{t}))$ as follows:
\begin{align*}&
v^{a}_{i}(\check{t}) = \begin{cases}\,
u,&\text{if $w(\check{t})$ contains $x_{i}\unode^{u}x_{i+1}$,}
\\[.5ex]\,\displaystyle
u+a+\ell(1{-}a),&\text{if $w(\check{t})$ contains $x_{i}\unode^{u}\natural\lnode^{\ell}x_{i+1}$,}
\end{cases} \ i<n,
\\[1ex]&
v^{a}_{n}(\check{t}) = \begin{cases}\,
u,&\text{if $w(\check{t})$ ends as $x_{n}\unode^{u}$,}
\\[.5ex]\,\displaystyle
u+a+\ell(1{-}a),&\text{if $w(\check{t})$ ends as $x_{n}\unode^{u}\natural\lnode^{\ell}$,}
\end{cases} \ i=n.
\end{align*}
Let $u(\check{t})\geq0$ and $\ell(\check{t})\geq0$ be the total numbers of upper and lower nodes, respectively, and $b(\check{t})\geq1$ be the number of beards.
Then, since $u(\check{t})+\ell(\check{t})=n\!-\!1$, we have
$$
v^{a}_{1}(\check{t})+\cdots+v^{a}_{n}(\check{t})=u(\check{t})+b(\check{t}){\cdot}a+\ell(\check{t}){\cdot}(1{-}a)=n\!-\!1+a(b(\check{t}){-}\ell(\check{t})),
$$
where $\check{t}$ is a bearded tree with one root and $n$ top-branches.
Then we have the following proposition.
\begin{prop}\label{prop:Multiplihedra-vertex}
$\displaystyle v^{a}_{1}(\check{t})+\cdots+v^{a}_{n}(\check{t})=n\!-\!1+a$.
\end{prop}
To prove this, we first show the following lemmas.
\begin{lem}\label{lem:beard-node0}
A bearded tree $\check{t}$ has only one beard, if and only if there is a  beard on the root branch.
\end{lem}
\begin{proof}
Assume that $\check{b}$ be the only beard of the bearded tree $\check{t}$.
If there were a node under $\check{b}$, we may assume that $\check{b}$ is on an upper branch of a node $\check{p}$ of $\check{t}$, and the whole upper part $\check{t}_{0}$ of the other branch of the node $\check{p}$ gives a smaller tree without beards, which contradicts to the hypothesis on a bearded tree that every top-branch meets exactly one beard on the way down to the root. Thus the beard is on the root branch. The converse is clear by definition of a bearded tree.
\end{proof}
\begin{lem}\label{lem:beard-node1}
A bearded tree $\check{t}$ has a node such that each of the upper branches of the node has a beard, unless $\check{t}$ has only one beard.
\end{lem}
\begin{proof}
We show the lemma by induction on the number of nodes.
\par
Firstly we assume that $\check{t}$ has only one node.
Then the claim is clear.
\par
Secondly, we assume that $\check{t}$ has multiple nodes and multiple beards, and we fix one beard $\check{b}$.
If there were no nodes under the beard $\check{b}$, the bearded tree $\check{t}$ has no other beards, by Lemma \ref{lem:beard-node0}, which contradicts to the hypothesis that the bearded tree $\check{t}$ has multiple beards.
Then, we may assume that $\check{b}$ is on an upper branch of a node $\check{p}$ in $\check{t}$.
\par
The whole upper part of the other upper branch of the node $\check{p}$ gives a smaller bearded tree $\check{t}_{0}$ which satisfy our claim by induction hypothesis, since $\check{t}_{0}$ has smaller number of nodes.
Hence $\check{t}_{0}$ has a node such that each of the upper branches of the node has a beard, unless $\check{t}_{0}$ has only one beard.
In the latter case, the only beard of $\check{t}_{0}$ is on its root branch, by Lemma \ref{lem:beard-node0}, which is the other upper branch of $\check{p}$ than the upper branch with $\check{b}$.
Hence each of the upper branches of $\check{p}$ has a beard.
\end{proof}
\begin{lem}\label{lem:beard-node2}
For a bearded tree $\check{t}$, the total number of lower nodes is one less than the total number of beards.
\end{lem}
\begin{proof}
We show the lemma by induction on the number of beards.
\par
Firstly, we assume that $\check{t}$ has only one beard.
Then, $\check{t}$ has no lower nodes by Lemma \ref{lem:beard-node0}, which implies that the total number of lower nodes of $\check{t}$ is $0$ which is one less than $1$ the total number of beards of $\check{t}$.
\par
Secondly, we assume that $\check{t}$ has multiple beards.
Then, by Lemma \ref{lem:beard-node1}, $\check{t}$ has a node $\check{p}$ such that each of the two upper branches of $\check{p}$ has a beard.
We then modify $\check{t}$ to obtain a new bearded tree $\check{t}'$ with smaller number of beards than $\check{t}$, by removing two beards from two upper branches of the node $\check{p}$ and by adding one beard to the lower branch of $\check{p}$.
Then, by the induction hypothesis, $\check{t}'$ satisfies the lemma.
\par
Then, the number of beards of $\check{t}$ is one more than the number of beards of $\check{t}'$, and also the set of upper nodes of $\check{t}'$ is consisting of those of $\check{t}$ with one new member $\check{p}$.
Thus the total number of lower nodes of $\check{t}$ is one less than the total number of beards of $\check{t}$.
\end{proof}
\begin{proof}[Proposition \ref{prop:Multiplihedra-vertex}]
Let $k\geq1$ be the number of beard in a bearded tree $\check{t}$ with one root and $n$ top-branches. Then we have $b(\check{t})-\ell(\check{t})=1$, and hence 
\par\vskip1ex\noindent\hfill$\displaystyle
v^{a}_{1}(\check{t})+\cdots+v^{a}_{n}(\check{t})=n\!-\!1+a(b(\check{t})\!-\!\ell(\check{t}))= n\!-\!1+a.
$\hfil%
\end{proof}
A similar consideration yields 
$v^{a}_{1}(\check{t})+\cdots+v^{a}_{i}(\check{t}) \leq i\!-\!1+a$,
which immediately implies that $(v^{a}_{1}(\check{t}),\dots,v^{a}_{n}(\check{t})) \in \mlta(n)$.
When $a=\fracinline{1}/{2}$, we denote $v_{i}(\check{t})=v^{a}_{i}(\check{t})$ and then we observe that $(v_{1}(\check{t}),\dots,v_{n}(\check{t})) \in \mlt(n)$.

\begin{defn}
For $n\geq1$, we define a set $\mlt[L](n)$ on the half-lattice.
\par\vskip1ex\noindent\hfil$\displaystyle
\mlta[L](n)=\left\{\,
(v^{a}_{1}(\check{t}),\dots,v^{a}_{n}(\check{t})) \,\midvert\, 
\text{\begin{minipage}{50mm}\baselineskip18pt
$\check{t}$ is a bearded tree with one root and $n$ top-branches
\end{minipage}}
\,\right\}
$\hfil\par\vskip1ex\noindent
and $\mlt[L](n)=J^{\fracinlines{1}/{2}}_{L}(n)$ which is in the half lattice $\fracinline{1}/{2}{\cdot}L$.
\end{defn}
Since $\mlta[L](1)=\mlta(1)$, we can show by induction on $n$ that $\mlta[L](n)$ gives the set of all vertices of $\mlta(n)$, and hence we have
\begin{prop}
$\mlta(n)$ is the convex hull of $\mlta[L](n)$.
\end{prop}

Similarly to the above, for any bearded tree $\check{t}$ of one root and $n$ top-branches, we define $w(\check{t})$ a word of $\check{t}$ as follows:
\begin{enumerate}
\item
assign a word `$x_{i}$' to the $i$-th top-branch from the left.
\item
if a top-branch or a node is assigned by a word `$w$', where a beard is attached to its lower part, then assign $\natural w$ to the beard.
\item
if the two upper branches of a node have no beared and are assigned by words `$w_{1}$' and `$w_{2}$' then assign a word `$\unode w'_{1}w'_{2}$' to the node.
\item
if the two upper branches of a node is assigned by a word `$w_{1}$' and `$w_{2}$' and its lower branch has no beard, then assign a word `$\lnode w'_{1}w'_{2}$' to the node.
\item
if the root branch is assigned by a word `$w$', we define $w(\check{t})$ the word of a tree $\check{t}$ to be $w$, i.e, $w(\check{t})=w$.
\end{enumerate}
We then define $\Ent(n)$ as the set of all words $w(t)$ of bearded trees $t$ with one root and $n$ top-branches:
\par\vskip1ex\noindent\hfil$\displaystyle
\Ent'(n)=\left\{\, w(\check{t})\,\midvert\,\text{\begin{minipage}{50mm}\baselineskip18pt
$\check{t}$ is a bearded tree with one root and $n$ top-branches\end{minipage}} \,\right\}
$\hfil\par\noindent
For a word $w$ in $\Ent'(n)$, we obtain an $n$-tuple $(u_{1}(\check{t}),\dots,u_{n}(\check{t}))$ of half integers as follows:
\begin{align*}&
u_{i}(\check{t}) = \begin{cases}\,
k,&\text{if $w(\check{t})$ contains $x_{i-1}\unode^{k}x_{i}$,}
\\\,\displaystyle
k+\frac{\ell{+}1}{2},&\text{if $w(\check{t})$ contains $x_{i-1}\lnode^{\ell}\natural\unode^{k}x_{i}$,}
\end{cases} \ i>1,
\\[1ex]&
u_{1}(\check{t}) = \begin{cases}\,
k,&\text{if $w(\check{t})$ starts as $\unode^{k}x_{1}$,}
\\\,\displaystyle
k+\frac{\ell{+}1}{2},&\text{if $w(\check{t})$ ends as $\lnode^{\ell}\natural\unode^{k}x_{1}$,}
\end{cases} \ i=1.
\end{align*}
\begin{defn}
For $n\geq1$, we define a set $J'_{L}(n)$ on the half-lattice.
$$
J'_{L}(n)=\left\{\,
(u_{1}(\check{t}),\dots,u_{n}(\check{t})) \,\midvert\, 
\text{\begin{minipage}{50mm}\baselineskip18pt
$\check{t}$ is a bearded tree with one root and $n$ top-branches
\end{minipage}}
\,\right\},
$$
where each entry of an element of $J'_{L}(n)$ is a half integer.
Further, we define $J'(n)$ as the mirror image of $\mlt(n)$ by taking convex hull of $J'_{L}(n)$.
\end{defn}

In $\real^{n}$, we take another hyper plane $H_{0}^{n-1} : x_{1}+\cdots+x_{n}=n{-}\fracinline{1}/{2}$.
Then we can easily observe that $\mlt(n) \subset H_{0}^{n-2}$ and $J'(n) \subset H_{0}^{n-2}$:
\begin{center}
\setlength\unitlength{.15mm}
\begin{picture}(200,380)(-70,-130)
\linethickness{1.0mm}		
\thicklines			
\put(-245,-40)	{\line(0,1){160}}
\put(-245,120)	{\line(0,1){80}}
\put(-245,200)	{\line(1,-1){80}}
\put(-165,-40)	{\line(0,1){160}}
\put(-165,-40)	{\line(-1,0){80}}
\put(-247,120)	{\line(1,0){4}}
\put(-167, 40)	{\line(1,0){4}}
\put(-195,-100)	{\makebox(0,0)[cc]{$(\mlt(3) \subset H_{0}^{2} \homeo \real^{2})$}}
\put(-250,-45)	{\makebox(0,0)[rc]{$(0,0,\fracinline{5}/{2})$}}
\put(-250,120)	{\makebox(0,0)[rc]{$(0,1,\fracinline{3}/{2})$}}
\put(-250,205)	{\makebox(0,0)[rc]{$(0,\fracinline{3}/{2},1)$}}
\put(-160,-45)	{\makebox(0,0)[lc]{$(\fracinline{1}/{2},0,2)$}}
\put(-160, 40)	{\makebox(0,0)[lc]{$(\fracinline{1}/{2},\fracinline{1}/{2},\fracinline{3}/{2})$}}
\put(-160,125)	{\makebox(0,0)[lc]{$(\fracinline{1}/{2},1,1)$}}
\put(320, -5)	{\makebox(0,0)[l]{$(1,\fracinline{3}/{2},0)$}}
\put(235,-25)	{\makebox(0,0)[c]{$(\fracinline{3}/{2},1,0)$}}
\put(165,105)	{\makebox(0,0)[c]{$(1,\fracinline{1}/{2},1)$}}
\put( 70, -5)	{\makebox(0,0)[r]{$(\fracinline{5}/{2},0,0)$}}
\put( 70, 90)	{\makebox(0,0)[r]{$(2,\fracinline{1}/{2},0)$}}
\put(240,105)	{\makebox(0,0)[l]{$(1,1,\fracinline{1}/{2})$}}
\put(165,-70)	{\makebox(0,0)[c]{$(J'(3) \subset H_{0}^{2} \homeo \real^{2})$}}
\put(315,  0)	{\line(-1,1){80}}
\put(235,  0)	{\line(-1,0){160}}
\put(235,  0)	{\line( 1,0){80}}
\put( 75, 80)	{\line(0,-1){80}}
\put( 75, 80)	{\line(1,0){160}}
\put(235, -8)	{\line(0,1){4}}
\put(165, 78)	{\line(0,1){4}}
\end{picture}
\end{center}

Let us summarize the properties of $\mlt(n)$ family.

\begin{prop}
\begin{enumerate}
\item
$\# \mlt[L](n) %
= \# J'_{L}(n)$. %
\item
$\mlt(n)$ is a convex hull of $\mlt[L](n)$.
\item
$J'(n)$ is a convex hull of $J'_{L}(n)$.
\item
$\mlt(n) %
\homeo J'(n)$ as polytopes.
\end{enumerate}
\end{prop}

By induction on $n$, we can show that there is a combinatorial homeomorphism $\mlt(n) \homeo J'(n)$ using bijection between $\mlt[L](n)$ and $J'_{L}(n)$.
$\mlt(n)$ and $J'(n)$ are easy to manipulate and mirror images to each other, and are constructed directly by using the language of bearded trees with one root and $n$ top-branches.

\section*{acknowledgement}
The author would like to thank Yusuke Kawamoto, Mitsunobu Tsutaya and Daisuke Kishimoto who made many important suggestions on the earlier versions of this paper.
The author would also like to thank John Hubbuck, Stephen Theriault and Jarek Kedra for making my sabbatical leave at Aberdeen possible, my close colleagues in Europe especially Cristina Costoya, Ran Levi, Carles Broto, Aniceto Murillo and also Stephen Theriault who made my stay in Europe extraordinary, Nobuyuki Izumida who endured to communicate by on-line with the author during his away from Japan.

%
%

\bibliographystyle{alpha}
\bibliography{2020unital}

\begin{thebibliography}{MHPS12}

\bibitem[Agu97]{MR2696373}
Marcelo Aguiar.
\newblock {\em Internal categories and quantum groups}.
\newblock ProQuest LLC, Ann Arbor, MI, 1997.
\newblock Thesis (Ph.D.)--Cornell University.

\bibitem[AK12]{MR2975601}
Hideto Asashiba and Mayumi Kimura.
\newblock Quiver presentations of {G}rothendieck constructions.
\newblock In {\em Proceedings of the 44th {S}ymposium on {R}ing {T}heory and
  {R}epresentation {T}heory}, pages 16--22. Symp. Ring Theory Represent. Theory
  Organ. Comm., Nagoya, 2012.

\bibitem[BV73]{MR420609}
J.~M. Boardman and R.~M. Vogt.
\newblock {\em Homotopy invariant algebraic structures on topological spaces}.
\newblock Lecture Notes in Mathematics, Vol. 347. Springer-Verlag, Berlin-New
  York, 1973.

\bibitem[For08]{MR2475119}
Stefan Forcey.
\newblock Convex hull realizations of the multiplihedra.
\newblock {\em Topology Appl.}, 156(2):326--347, 2008.

\bibitem[Fuk93]{MR1270931}
Kenji Fukaya.
\newblock Morse homotopy, {$A^\infty$}-category, and {F}loer homologies.
\newblock In {\em Proceedings of {GARC} {W}orkshop on {G}eometry and {T}opology
  '93 ({S}eoul, 1993)}, volume~18 of {\em Lecture Notes Ser.}, pages 1--102.
  Seoul Nat. Univ., Seoul, 1993.

\bibitem[Hai84]{haiman1984constructing}
Mark Haiman.
\newblock Constructing the associahedron.
\newblock Manuscript,
  https://math.berkeley.edu/~mhaiman/ftp/assoc/manuscript.pdf, 1984.

\bibitem[IM89]{MR1000378}
Norio Iwase and Mamoru Mimura.
\newblock Higher homotopy associativity.
\newblock In {\em Algebraic topology ({A}rcata, {CA}, 1986)}, volume 1370 of
  {\em Lecture Notes in Math.}, pages 193--220. Springer, Berlin, 1989.

\bibitem[Iwa83a]{iwase1983map}
Norio Iwase.
\newblock On the ${A}_{n}$-structures of mappings (study of unstable homotopy
  theory).
\newblock {\em S\=uri Kaiseki Kenky\=usho K\=oky\=uroku}, 505:63--75, 1983.
\newblock
  https://www.kurims.kyoto-u.ac.jp/~kyodo/kokyuroku/contents/pdf/0505-06.pdf.

\bibitem[Iwa83b]{Iwase:1983}
Norio Iwase.
\newblock On the ring structure of {X}-projective n-space.
\newblock Master's thesis, Kyushu University, 1983.
\newblock http://hdl.handle.net.anywhere.lib.kyushu-u.ac.jp/2324/1498289.

\bibitem[Jam60]{MR133132}
I.~M. James.
\newblock On {$H$}-spaces and their homotopy groups.
\newblock {\em Quart. J. Math. Oxford Ser. (2)}, 11:161--179, 1960.

\bibitem[MHPS12]{MR3235205}
Folkert M\"{u}ller-Hoissen, Jean~Marcel Pallo, and Jim Stasheff, editors.
\newblock {\em Associahedra, {T}amari lattices and related structures}, volume
  299 of {\em Progress in Mathematics}.
\newblock Birkh\"{a}user/Springer, Basel, 2012.
\newblock Tamari memorial Festschrift.

\bibitem[Mim86]{Mimura86hopf}
M.~Mimura.
\newblock {\em Hopf spaces}.
\newblock Kinokuniya Sugaku Sosho. NetLibrary, 1986.

\bibitem[MSS02]{MR1898414}
Martin Markl, Steve Shnider, and Jim Stasheff.
\newblock {\em Operads in algebra, topology and physics}, volume~96 of {\em
  Mathematical Surveys and Monographs}.
\newblock American Mathematical Society, Providence, RI, 2002.

\bibitem[MW10]{MR2660682}
S.~Ma'u and C.~Woodward.
\newblock Geometric realizations of the multiplihedra.
\newblock {\em Compos. Math.}, 146(4):1002--1028, 2010.

\bibitem[Sta63]{MR0158400}
James~Dillon Stasheff.
\newblock Homotopy associativity of {$H$}-spaces. {I}, {II}.
\newblock {\em Trans. Amer. Math. Soc. 108 (1963), 275-292; ibid.},
  108:293--312, 1963.

\bibitem[Sta70]{MR0270372}
James Stasheff.
\newblock {\em {$H$}-spaces from a homotopy point of view}.
\newblock Lecture Notes in Mathematics, Vol. 161. Springer-Verlag, Berlin-New
  York, 1970.

\bibitem[Tam51]{MR51833}
Dov Tamari.
\newblock {\em Mono\"{\i}des pr\'{e}ordonn\'{e}s et cha\^{i}nes de {M}alcev}.
\newblock Universit\'{e} de Paris, Paris, 1951.
\newblock Th{\`e}se.

\bibitem[Tam09]{Tamaki:2009aa}
Dai Tamaki.
\newblock The grothendieck construction and gradings for enriched categories.
\newblock 07 2009.
\newblock arXiv preprint, https://arxiv.org/abs/0907.0061.

\bibitem[Whi78]{MR516508}
George~W. Whitehead.
\newblock {\em Elements of homotopy theory}, volume~61 of {\em Graduate Texts
  in Mathematics}.
\newblock Springer-Verlag, New York-Berlin, 1978.

\bibitem[Zab76]{MR440542}
Alexander Zabrodsky.
\newblock {\em Hopf spaces}.
\newblock North-Holland Mathematics Studies, Vol. 22. North-Holland Publishing
  Co., Amsterdam-New York-Oxford, 1976.
\newblock Notas de Matem\'{a}tica [Mathematical Notes], No. 59.

\end{thebibliography}

\end{document}